\documentclass[a4paper,11pt]{article}

\usepackage[T1]{fontenc}
\usepackage{lmodern}

\usepackage{dsfont}

\usepackage[latin1]{inputenc}
\usepackage{amsmath}
\usepackage{amsthm}
\usepackage{amssymb}
\usepackage{mathrsfs}
\usepackage{graphicx}
\usepackage[all]{xy}
\usepackage{hyperref}

\usepackage{makeidx}

\usepackage{stmaryrd}
\usepackage{caption}

\usepackage{abstract} 

\newtheorem{thm}{Theorem}[section]
\newtheorem{cor}[thm]{Corollary}

\newtheorem{fact}[thm]{Fact}

\newtheorem{lemma}[thm]{Lemma}
\newtheorem{prop}[thm]{Proposition}

\theoremstyle{definition}
\newtheorem{definition}[thm]{Definition}
\newtheorem{ex}[thm]{Example}
\newtheorem{remark}[thm]{Remark}
\newtheorem{question}[thm]{Question}
\newtheorem{problem}[thm]{Problem}
\newtheorem{conj}[thm]{Conjecture}

\def\rquotient#1#2{%
	\makeatletter
	\raise.3ex\hbox{$#1$}/\lower.3ex\hbox{$#2$}%
	\makeatother
}	

\makeatletter
\newcommand{\subjclass}[2][2010]{%
	\let\@oldtitle\@title%
	\gdef\@title{\@oldtitle\footnotetext{#1 \emph{Mathematics subject classification.} #2}}%
}
\newcommand{\keywords}[1]{%
	\let\@@oldtitle\@title%
	\gdef\@title{\@@oldtitle\footnotetext{\emph{Key words and phrases.} #1.}}%
}
\makeatother

\newcommand{\Address}{{
		\bigskip
		\small
		
		\textsc{University of Montpellier\\ 
Institut Math\'ematiques Alexander Grothendieck\\
Place Eug\`ene Bataillon\\
34090 Montpellier (France)}\par\nopagebreak
		\textit{E-mail address}: \texttt{anthony.genevois@umontpellier.fr}
		
}}

\makeindex

\title{An introduction to diagram groups}
\date{\today}
\author{Anthony Genevois}


\begin{document}

\maketitle

\begin{abstract}
To every semigroup presentation $\mathcal{P}= \langle \Sigma \mid \mathcal{R} \rangle$ and every baseword $w \in \Sigma^+$ can be associated a \emph{diagram group} $D(\mathcal{P},w)$, defined as the fundamental group of the \emph{Squier complex} $S(\mathcal{P},w)$. Roughly speaking, $D(\mathcal{P},w)$ encodes the lack of asphericity of $\mathcal{P}$. Examples of diagram groups include Thompson's group $F$, the lamplighter group $\mathbb{Z} \wr \mathbb{Z}$, the pure planar braid groups, and various right-angled Artin groups. This survey aims at summarising what is known about the family of diagram groups. 
\end{abstract}

\small
\tableofcontents
\normalsize

\section{Introduction}

\noindent
Defined by J. Meakin and M. Sapir (unpublished), and first investigated by V. Kilibarda in her thesis \cite{MR1448329}, diagram groups have been mainly promoted by V. Guba and M. Sapir, in particular through their monograph \cite{MR1396957}. Since then, many articles have been dedicated to the understanding of which groups can be described as (subgroups of) diagram groups and which properties can be deduced from such a description. Perspectives on the subject include computational problems such that the word, conjugacy, and commutation problems \cite{MR1396957, MR1725439}, finiteness properties \cite{MR1978047}, homology \cite{MR2193190}, orderability \cite{MR1983088, MR2193191}, median geometry \cite{MR1978047, MR3868219, MR4033512}, Hilbertian geometry \cite{MR2271228}, negatively curved geometry \cite{HypDiag, MR4071367}.

\medskip \noindent
In this survey, our goal is to organise in a coherent and essentially self-contained way what is known about the family of diagram groups. 

\medskip \noindent
In Section~\ref{section:FirstTaste}, we first present three different, but equivalent, descriptions of diagram groups: as a way to encode the lack of asphericity of a semigroup presentation (Section~\ref{section:Asphericity}); as a two-dimensional analogue of free groups, where words are replaced with diagrams over semigroup presentations (Section~\ref{section:Diagrams}); and finally as second fundamental groups of directed $2$-complexes (Section~\ref{section:Free}). Next, we record explicit examples of diagram groups in Section~\ref{section:Examples}, including Thompson's group $F$, its commutator subgroup $F'$, the lamplighter group $\mathbb{Z} \wr \mathbb{Z}$, the pure planar braid groups, and various right-angled Artin groups. Most properties satisfied by diagram groups are listed in Section~\ref{section:Properties}, sometimes accompanied with ideas of proofs. 

\medskip \noindent
Section~\ref{section:Algo}, we are mainly concerned with algorithmic and computational aspects of diagrams. In Section~\ref{section:Presentation}, we show how one can compute efficiently presentations of diagram groups. Sections~\ref{section:WordConj} and~\ref{section:Commutation} essentialy deal with the combinatorics of semigroup diagrams, with a focus on the conjugacy problem in Section~\ref{section:WordConj} and on centralisers in Section~\ref{section:Commutation}. In Section~\ref{section:ComSub}, we show how to determine efficiently whether an element of a diagram group belongs to the commutator subgroup. Finally, in Section~\ref{section:Fold}, we exploit the point of view given by directed $2$-complexes in order to adapt the well-known foldings of Stallings from free groups to diagram groups. This allows us to solve the membership problem for some specific subgroups.

\medskip \noindent
In Section~\ref{section:Median}, we consider diagram groups from a more geometric perspective. After a crash course on median geometry in Section~\ref{section:Crash}, we show in Section~\ref{section:MedianDiag} that diagram groups naturally act on median graphs. This geometry allows us to extract some valuable information about diagram groups regarding finiteness properties, once combined with a combinatorial version of Morse theory (Section~\ref{section:Morse}); Hilbert space compression (Section~\ref{section:Hilbert}); and acylindrical hyperbolicity (Section~\ref{section:Acyl}). The structure of hyperplanes in the nonpositively curved cube complexes whose fundamental groups are the diagram groups is studied in Section~\ref{section:Hyperplanes}, with applications to residual finiteness and subgroup separability. Finally, in Section~\ref{section:DiagramProducts}, we introduce diagram products, a version of diagram groups with coefficients coming from a family of groups indexed by the alphabet of the underlying semigroup presentation. We show that these products naturally act on quasi-median graphs, and we exploit this geometry in order to extract some information on the structure of diagram products. 

\medskip \noindent
We conclude this survey by describing various possible generalisations of diagram groups, most of them already available in the literature (Section~\ref{section:generalisations}); and by listing several open questions and problems that we find appealing (Section~\ref{section:OpenQuestion}), with the hope that this could motivate future works on diagram groups.


\section{A first taste}\label{section:FirstTaste}

\subsection{Diagram groups as lack of asphericity}\label{section:Asphericity}

\noindent
Let $\mathcal{P} = \langle \Sigma \mid \mathcal{R} \rangle$ be a semigroup presentation, i.e.\ an alphabet $\Sigma$ and a collection of relations $\mathcal{R}$ of the form $u=v$ where $u,v$ are positive words written over $\Sigma$. In the sequel, we will always assume that $\mathcal{R}$ does not contain obvious redundancy, i.e.\ if $u=v$ is in $\mathcal{R}$ we will assume that $v=u$ is not in $\mathcal{R}$. In particular, $\mathcal{R}$ does not contains a relation of the form $u=u$. As an example, we can take $\mathcal{P}= \langle a,b \mid ab=ba, a=a^2 \rangle$. 

\medskip \noindent
The presentation $\mathcal{P}$ is said to be \emph{aspherical} if, given two positive words $p,q \in \Sigma^+$ representing the same element in the semigroup, there is essentially a unique way to derive $q$ from $p$ by applying relations in $\mathcal{R}$. Depending on the meaning one gives to such a uniqueness, one may obtain distinct notions of asphericity. For instance, the words $abba$ and $bab$ represent the same element in the semigroup given by $\mathcal{P} = \langle a,b \mid ab=ba,a=a^2 \rangle$, a natural derivation being
$$abba \to baba \to baab \to bab.$$
However, such a derivation is far from being unique. For instance, instead of applying the relation $ab=ba$ to first half of $abba$ and next to the second half, we can apply the relation to the second half of the word and next to the first half:
$$abba \to abab \to baab \to bab.$$
Even worth, we can apply a relation and next undo what we have just did, e.g.\ 
$$abba \to bab \to baab \to bab \to bba \to bab.$$
However, we do not consider such variations as being \emph{essentially} different from our initial derivations. In order to formalise this idea, we use some algebraic topology.

\begin{definition}\label{def:SquierSquare}
Let $\mathcal{P}= \langle \Sigma \mid \mathcal{R} \rangle$ be a semigroup presentation. The \emph{Squier square complex} $S(\mathcal{P})$ is the complex whose vertices are the positive words written over $\Sigma$; whose edges $[a,u=v,b]$ connect two words $aub$ and $avb$ if one can be obtained from the other by applying a relation $u=v$ from $\mathcal{R}$; and whose squares $[a,u=v,b,p=q,c]$ bound the $4$-cycles given by edges $[a,u=v,bpc]$, $[a,u=v,bqc]$, $[aub,p=q,c]$, $[avb,p=q,c]$. 
\end{definition}
\begin{figure}[h!]
\begin{center}
\includegraphics[width=0.25\linewidth]{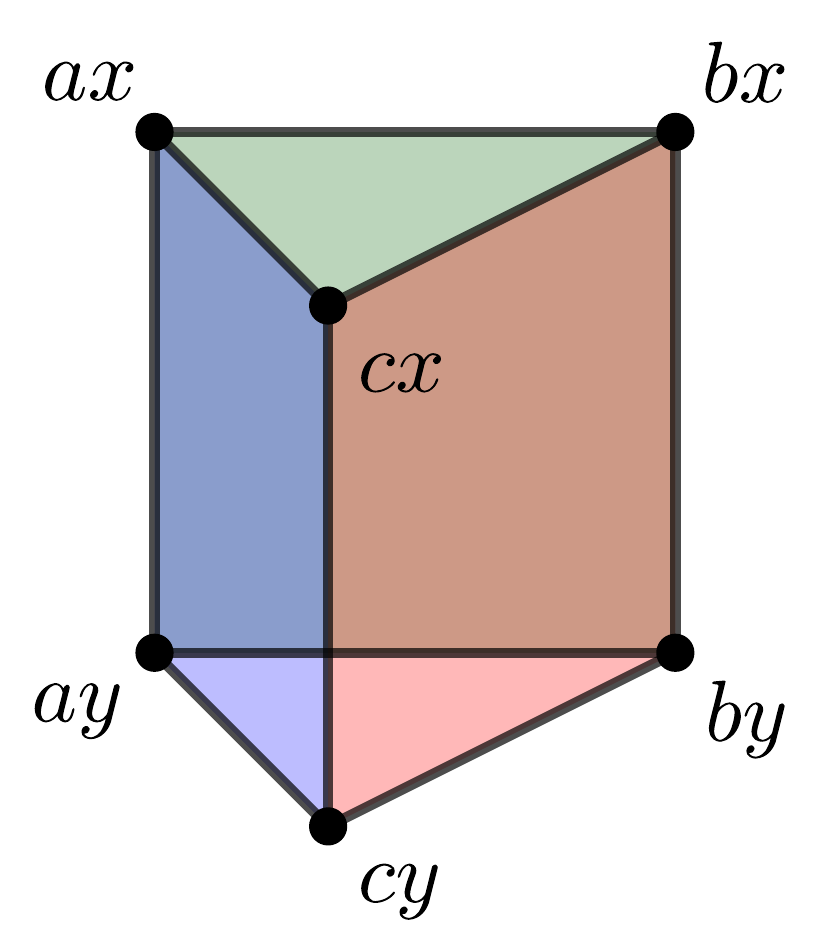}
\caption{The complex $S(\mathcal{P},ad)$ where $\mathcal{P} = \left\langle a,b,c,x,y \left|  \begin{array}{c} a=b,b=c \\ c=a, x=y \end{array} \right. \right\rangle$.}
\label{Squier}
\end{center}
\end{figure}

\noindent
Then, one says that our presentation $\mathcal{P}$ is \emph{aspherical} if the fundamental group of each connected component of the correspond Squier square complex is trivial. Thus, the fundamental groups of the connected components of $S(\mathcal{P})$ encode the lack of asphericity of $\mathcal{P}$. These groups are our \emph{diagram groups}.

\begin{definition}\label{def:DiagramGroups}
Let $\mathcal{P}= \langle \Sigma \mid \mathcal{R} \rangle$ be a semigroup presentation and $w \in \Sigma^+$ a baseword. The \emph{diagram group} $D(\mathcal{P},w)$ is the fundamental group $\pi_1(S(\mathcal{P}),w)$ of the Squier square complex $S(\mathcal{P})$ based at $w$. 
\end{definition}

\noindent
This is the shortest definition one can get of diagram groups, and it already allows us to construct examples. For instance, the diagram group given by Figure~\ref{Squier} is a free group of rank two. However, Definition~\ref{def:DiagramGroups} does not explain why we are focusing on semigroup presentations instead of monoid or group presentations, and its formulation has nothing to do with \emph{diagrams}, so the terminology we are using remains mysterious. These points will be clarified in the next section.

\subsection{Diagrammatic representation}\label{section:Diagrams}

\noindent
Let $\mathcal{P} = \langle \Sigma \mid \mathcal{R} \rangle$ be a semigroup presentation. Given two positive words $w_1,w_2 \in \Sigma^+$ representing the same element in the semigroup given by $\mathcal{P}$ a derivation from $w_1$ to $w_2$, or equivalently a path in the Squier complex $S(\mathcal{P})$ from $w_1$ to $w_2$, can be encoded by a \emph{diagram over $\mathcal{P}$}. For instance, if $\mathcal{P}=\langle a,b \mid ab=ba, a=a^2 \rangle$, then the derivation
$$abba \to baba \to baab \to bab \to abb$$
can be represented by the diagram 
\begin{center}
\includegraphics[width=0.4\linewidth]{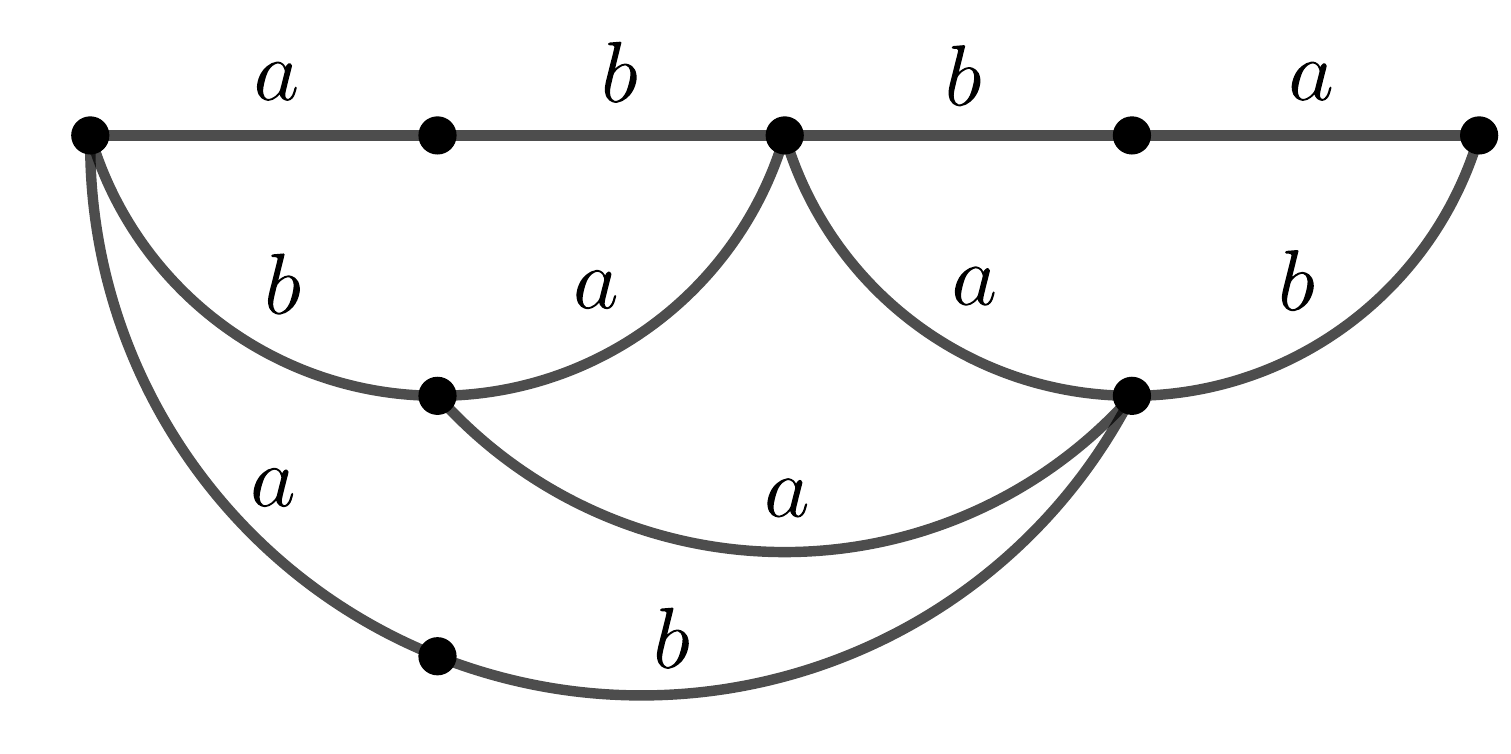}
\end{center}
The diagram is constructed as follows. We start with a top path (oriented from left to right) labelled by our initial word $abba$. The first step $abba \to baba$ in our derivation applies the relation $ab=ba$ to the first half of the word $abba$, so we glue a $2$-cell to our path with a boundary which is the disjoint union of a top path $ab$ and a bottom path $ba$. Similarly, following the second step $baba\to baab$ in our derivation, we glue a $2$-cell with a boundary which is the disjoint union of a top path $ba$ and a bottom path $ab$. So far, we have a diagram with two $2$-cells, whose top path is labelled by $abba$, and whose bottom path is labelled by $baab$. Following the third step $baab \to bab$ of our derivation, we glue a $2$-cell with a boundary which is the disjoint union of a top bath labelled $aa$ and a bottom path labelled $a$. Finally, we glue a forth $2$-cell corresponding the forth step $bab \to abb$ from our derivation. 

\medskip \noindent
Formally, a diagram of a semigroup presentation $\mathcal{P}= \langle \Sigma \mid \mathcal{R} \rangle$ is a planar oriented graph $\Delta$ whose oriented edges are labelled by letters from $\Sigma$ such that:
\begin{itemize}
	\item $\Delta$ has a unique source and a unique sink;
	\item the boundary of every cell is a disjoint union of two oriented paths, referred to as the \emph{top} and \emph{bottom} paths of the cell;
	\item for every cell, if $u$ and $v$ denote the labels of its top and bottom paths, then $u=v$ or $v=u$ belongs to $\mathcal{R}$. 
\end{itemize}
By definition, a diagram comes with a fixed embedding in the plane. Two diagrams only differing by an isotopy of the plane are considered as identical. In the sequel, orientations of edges in diagrams will be clear from our embeddings (from left to right) so they will not be specified. 

\medskip \noindent
Our diagrams are the analogues for semigroups of van Kampen diagrams for groups. In particular, two positive words $w_1,w_2 \in \Sigma^+$ represent the same element of the semigroup given by $\mathcal{P}$ if and only if there exists a diagram $\Delta$ over $\mathcal{P}$ with $\mathrm{top}(\Delta)=w_1$ and $\mathrm{bot}(\Delta)=w_2$. 

\medskip \noindent
Thus, every path in the Squier complex $S(\mathcal{P})$ from some vertex $w_1$ to another vertex $w_2$ can be represented by a diagram over $\mathcal{P}$ whose top and bottom words respectively are $w_1$ and $w_2$. Notice that two paths can be represented by the same diagram. For instance, the diagram given above also represents the path
$$abba \to abab \to baab \to baba \to abb.$$
Clearly, the paths represented by a given diagram $\Delta$ correspond to the different ways one has to construct $\Delta$ by gluing $2$-cells successively. In our example, we can glue the top left cell first and next the top right cell, or we can glue the top right cell first and next the top left cell. These two possibilities correspond to the two paths we gave. 

\medskip \noindent
It is worth noticing that that all the paths encoded by a diagram are pairwise homotopic. However, two homotopic paths may not be represented by the same diagram. 
\begin{center}
\includegraphics[width=0.7\linewidth]{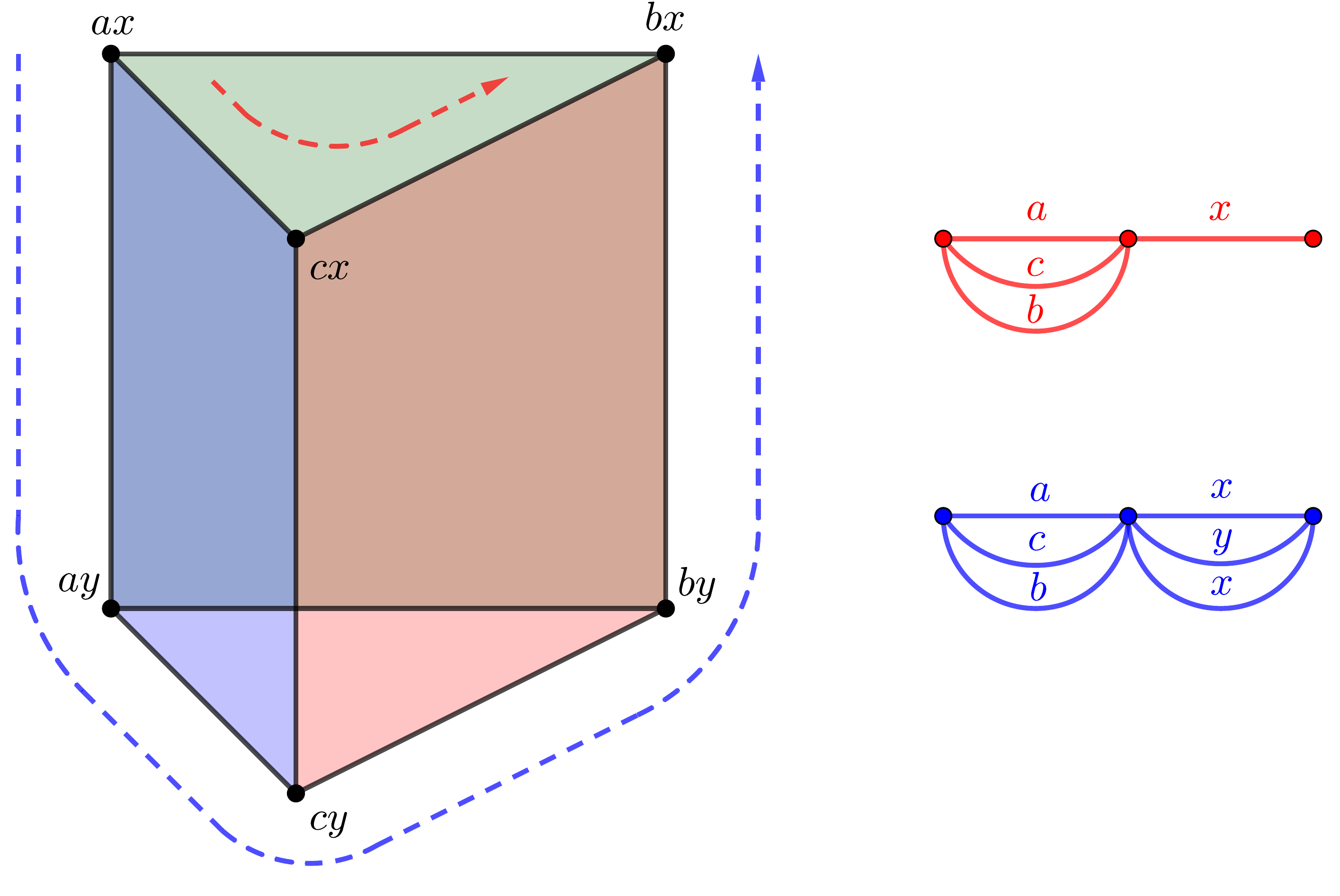}
\end{center}

\noindent
In a square complex, two paths in the one-skeleton are homotopy equivalent if and only if one can be obtained from the other by a sequence of elementary moves, namely flipping a square and adding or removing a backtrack.
\begin{figure}[h!]
\begin{center}
\includegraphics[width=0.9\linewidth]{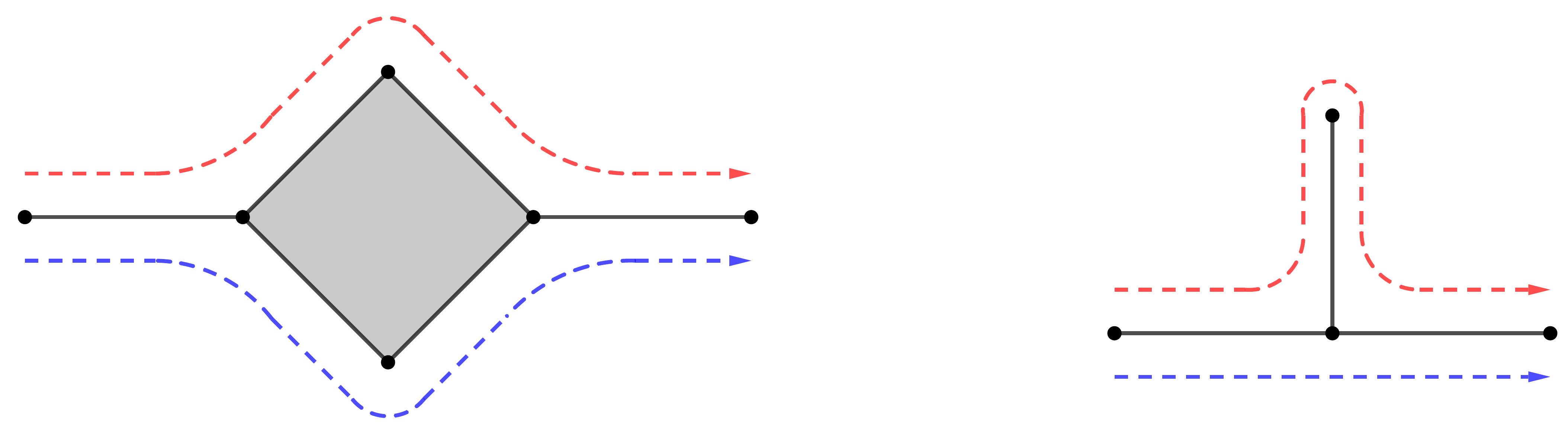}
\caption{Flipping a square and removing / adding a backtrack.}
\label{Moves}
\end{center}
\end{figure}

\noindent
It is clear that flipping a square does not modify the diagram representing our path. However, adding or removing a backtrack does modify the diagram. This motivates the following definition.

\begin{definition}
Let $\mathcal{P}=\langle \Sigma \mid \mathcal{R} \rangle$ be a semigroup presentation and $\Delta$ a diagram over $\mathcal{P}$. A \emph{dipole} in $\Delta$ is the data of two cells $\pi_1,\pi_2 \subset \Delta$ such that $\mathrm{bot}(\pi_1)=\mathrm{top}(\pi_2)$ and such that $\pi_1,\pi_2$ correspond to a relation of $\mathcal{R}$ and its inverse. One \emph{reduces} the dipole by identifying the paths $\mathrm{top}(\pi_1)$ and $\mathrm{bot}(\pi_2)$. 
\end{definition}

\noindent
For instance, the blue diagram above contains a dipole, namely the two right cells. Reducing the dipole provides the red diagram. It worth noticing that the diagram associated to a path with a backtrack necessarily contains a dipole. However, the converse is not true, as shown by the blue diagram above. 

\begin{prop}\label{prop:PathDiag}
Let $\mathcal{P}$ be a semigroup presentation. Two paths in the Squier complex $S(\mathcal{P})$ having the same endpoints are homotopy equivalent if and only if the diagrams over $\mathcal{P}$ representing them have the same reduction. 
\end{prop}

\noindent
Here, the \emph{reduction} of a diagram refers to the diagram obtained after reducing all the dipoles. A priori, the diagram thus obtained could differ depending on the order we following when reducing the dipoles. It turns out not to be the case; see \cite[Theorem~3.17]{MR1396957}. Thus, the reduction of a diagram is well-defined. 

\medskip \noindent
Proposition~\ref{prop:PathDiag} immediately implies:

\begin{cor}\label{cor:PathDiag}
Let $\mathcal{P}$ be a semigroup presentation. The map that sends a path in $S(\mathcal{P})$ to its diagram over $\mathcal{P}$ induces an isomorphism from the fundamental groupoid of $S(\mathcal{P})$ to the \emph{diagram groupoid} $D(\mathcal{P})$ whose elements are the diagrams over $\mathcal{P}$ up to dipole reduction and whose product sends any two diagrams $\Delta_1,\Delta_2$ with $\mathrm{bot}(\Delta_1), \mathrm{top}(\Delta_2)$ labelling by the same word to the diagram $\Delta_1 \circ \Delta_2$ obtained by gluing $\Delta_2$ below $\Delta_2$. 
\end{cor}

\noindent
Here, we think as a groupoid as a set endowed with a product only partially defined but satisfying basically the same axioms as groups: the product, when it is defined, is associative; there are neutral elements (not unique); and every element admits a unique inverse. The example to keep in mind is the fundamental groupoid of a topological space: its elements are the oriented paths up to homotopy and the product between two paths $\alpha,\beta$ is well-defined when the terminal point of $\alpha$ coincides with the initial vertex point of $\beta$ in which case the product is just the concatenation of $\alpha$ and $\beta$. Neutral elements are single points and the inverse of an element is obtained by reversing the orientation. In $D(\mathcal{P})$, the neutral elements are the diagrams with no $2$-cells and the inverse of a diagram is given by its mirror image (along a horizontal axis).
\begin{figure}[h!]
\begin{center}
\includegraphics[width=\linewidth]{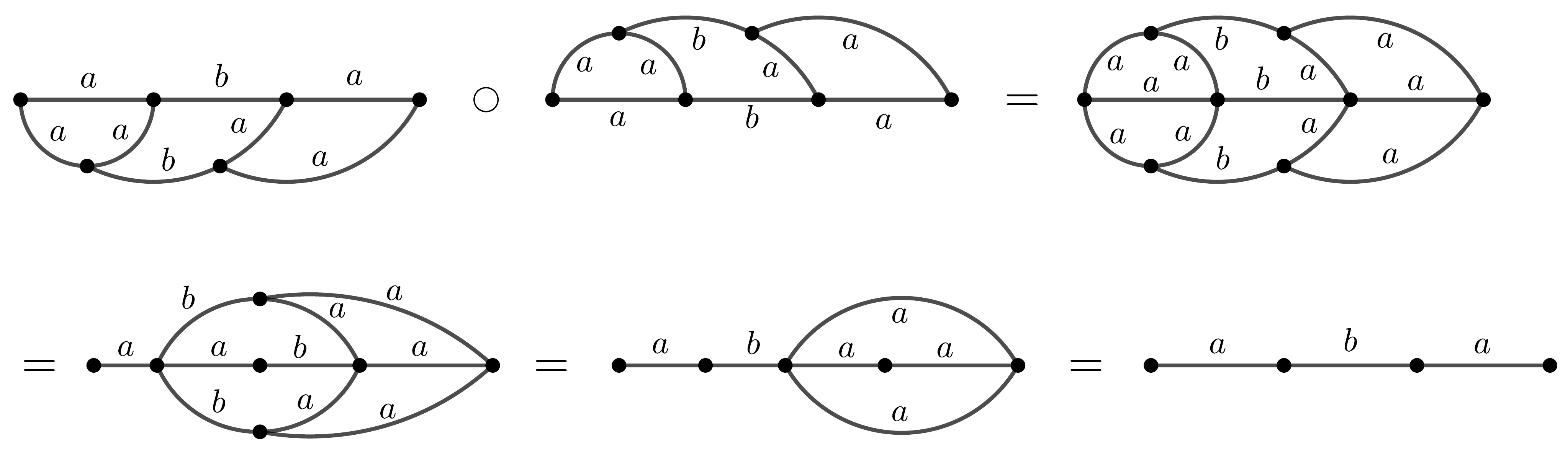}
\caption{Product of diagrams over $\mathcal{P}= \langle a,b \mid ab=ba, a=a^2 \rangle$.}
\label{Product}
\end{center}
\end{figure}

\noindent
As a particular case of Corollary~\ref{cor:PathDiag}, given a semigroup presentation $\mathcal{P}= \langle \Sigma \mid \mathcal{R} \rangle$ and a baseword $w \in \Sigma^+$, the diagram group $D(\mathcal{P},w)$ can be described as the set of \emph{$(w,w)$-diagrams} over $\mathcal{P}$ (i.e.\ the diagrams whose top and bottom paths are labelled by $w$) up to dipole reductions endowed with the concatenation product.

\subsection{Diagram groups as two-dimensional free groups}\label{section:Free}

\noindent
There is an analogy between our previous description of diagram groups and the construction of free groups. Elements of free groups are words (which can be thought of as linear objects) up to reductions $xx^{-1}=1$ (corresponding to dipole reductions), which we concatenate together (similarly to concatenations of diagrams). The analogy can be pushed further. In the same way that free groups can be described as fundamental groups of graphs, diagram groups can be described as \emph{second fundamental groups} of \emph{directed $2$-complexes}. This point of view was developed in \cite{MR2193190}, and will be crucial in Section~\ref{section:Fold}. However, since it will not be essentially used elsewhere and is rather technical, the present section can be skept in a first reading. 

\begin{definition}
A \emph{directed $2$-complex} is the data of a directed graph $\Gamma$, a set of \emph{$2$-cells} $F$, two maps $\mathrm{top}$, $\mathrm{bot}$ sending each $2$-cell to a directed path in $\Gamma$, and a map $\iota : F \to F$ such that:
\begin{itemize}
	\item for every $f \in F$, $\mathrm{bot}(f)$ and $\mathrm{top}(f)$ are non-empty and have common initial and terminal vertices;
	\item $\iota$ is a fixed-point free involution satisfying $\mathrm{top}(\iota(f))= \mathrm{bot}(f)$ and $\mathrm{bot}(\iota(f))= \mathrm{top}(f)$ for every $f \in F$.
\end{itemize}
\end{definition}

\noindent
For instance, given a semigroup presentation $\mathcal{P}$, a diagram $\Delta$ over $\mathcal{P}$ is naturally a directed $2$-complex (drawn on the plane). If the top and bottom paths of $\Delta$ have the same label, then one can identify $\mathrm{top}(\Delta)$ with $\mathrm{bot}(\Delta)$ in order to obtain another directed $2$-complex (drawn on the sphere). 

\medskip \noindent
Mimicking the construction of a cellular $2$-complex from a group presentation, one can also associate a directed $2$-complex $\mathcal{X}(\mathcal{P})$ to our semigroup presentation $\mathcal{P}= \langle \Sigma \mid \mathcal{R}\rangle$. Namely, $\mathcal{X}(\mathcal{P})$ has a single vertex, one directed edge for each generator in $\Sigma$, and, for each relation $u=v$ in $\mathcal{R}$, we add two $2$-cells whose top and bottom paths are respectively labelled by $u,v$ and $v,u$. 

\medskip \noindent
Now, our goal is to define a \emph{second fundamental groupoid} of a directed $2$-complex. Roughly speaking, it is given by homotopy classes of \emph{$2$-paths} (i.e.\ paths of paths) endowed with the usual concatenation. 

\begin{definition}
Let $\mathcal{X}$ be a directed $2$-complex. A \emph{$1$-path} is a directed path in the underlying directed graph. An \emph{elementary transformation} of a $1$-path $\alpha$ is the replacement of a subsegment $\beta \subset \alpha$ which coincides with the bottom path of some $2$-cell $\pi$ with $\mathrm{top}(\pi)$. A \emph{$2$-path} is a finite sequence of $1$-paths such that each $1$-paths is obtained from the previous one by an elementary transformation. 
\end{definition}

\noindent
One should think of an elementary transformation as \emph{pushing} a $1$-path $\alpha$ through a $2$-cell having its bottom path in $\alpha$. Observe that, given a semigroup presentation $\mathcal{P}$, a $2$-path in the directed $2$-complex $\mathcal{X}(\mathcal{P})$ is naturally encoded by a diagram over $\mathcal{P}$. 

\begin{definition}
Let $\mathcal{X}$ be a directed $2$-complex. Two $2$-paths are \emph{elementary equivalent} if they are of the form
$$\alpha_1, \ldots, \alpha_r, \zeta \mathrm{bot}(\pi) \xi, \zeta \mathrm{top}(\pi) \xi, \zeta \mathrm{bot}(\pi) \xi, \beta_1 ,\ldots \beta_s$$
$$\alpha_1, \ldots, \alpha_r, \zeta \mathrm{bot}(\pi) \xi, \beta_1, \ldots, \beta_s$$
for some $2$-cell $\pi$, or
$$\alpha_1, \ldots, \alpha_r, \alpha  \mathrm{bot}(\mu)  \beta  \mathrm{bot}(\nu)  \gamma, \alpha \mathrm{top}(\mu) \beta \mathrm{bot}(\nu) \gamma, \alpha \mathrm{top}(\mu) \beta \mathrm{top}(\nu) \gamma, \beta_1 ,\beta_s$$
$$\alpha_1, \ldots, \alpha_r, \alpha \mathrm{bot}(\mu) \beta \mathrm{bot}(\nu) \gamma, \alpha \mathrm{bot}(\mu) \beta \mathrm{top}(\nu) \gamma, \alpha \mathrm{top}(\mu) \beta \mathrm{top}(\nu) \gamma, \beta_1 ,\beta_s$$
for some $2$-cells $\mu,\nu$. Two $2$-paths are \emph{homotopy equivalent} if one can pass from one to the other by a sequence of elementary equivalences.
\end{definition}

\noindent
In other words, the first elementary equivalence tells us that pushing a $1$-path through a $2$-cell and next through its inverse amounts to doing nothing. The second elementary equivalence tells us that, given two $2$-cells $\pi_1,\pi_2$ having their bottom paths in our $1$-path $\zeta$ and disjoint, pushing $\alpha$ through $\pi_1$ and next through $\pi_2$ amounts to pushing $\alpha$ through $\pi_2$ and next through $\pi_1$. 

\medskip \noindent
As previously mentioned, a $2$-path in the directed $2$-complex $\mathcal{X}(\mathcal{P})$ associated to a semigroup presentation $\mathcal{P}$ is naturally encoded by a diagram over $\mathcal{P}$. It should be clear from our definition that any two homotopy equivalent $2$-paths are encoded by the same diagram modulo dipole reduction. 

\begin{definition}
Let $\mathcal{X}$ be a directed $2$-complex. The \emph{second fundamental groupoid} $\Pi_2(\mathcal{X})$ of $\mathcal{X}$ is the set of $2$-paths up to homotopy equivalence endowed with concatenation. Given a $1$-path $\alpha$, the \emph{fundamental group} $\pi_2(\mathcal{X},\alpha)$ is the set of $2$-paths starting and ending at $\alpha$ up to homotopy equivalence endowed with concatenation. 
\end{definition}

\noindent
Observe that neutral elements of $\Pi_2(\mathcal{X})$ are $2$-paths reduced to single $1$-paths, and that the inverse of a $2$-path is the same $2$-path but read in the reversed order. 

\medskip \noindent
Interestingly, second fundamental groups of directed $2$-complexes define the same class of groups as diagram groups \cite[Theorem~4.3]{MR2193190}.

\begin{thm}
For every semigroup presentation $\mathcal{P}$, the map that sends a $2$-path in $\mathcal{X}(\mathcal{P})$ to the diagram encoding it induces an groupoid isomorphism $\Pi_2(\mathcal{X}(\mathcal{P})) \to D(\mathcal{P})$. Conversely, every directed $2$-complex $\mathcal{X}$ can be turned into the directed $2$-complex associated to a semigroup presentation $\mathcal{P}$ such that $\Pi_2(\mathcal{X})$ and $D(\mathcal{P})$ are isomorphic. 
\end{thm}

\noindent
The first assertion should be clear from what we already said. We refer to \cite{MR2193190} for more information on the second assertion.

\subsection{Examples of diagram groups}\label{section:Examples}

\noindent
The simplest diagram group is given by the semigroup presentation 
$$\mathcal{P}= \langle a,b,c \mid a=b,b=c,c=a \rangle.$$ 
The connected component $S(\mathcal{P},a)$ containing $a$ of the Squier complex $S(\mathcal{P})$ is just a $3$-cycle with vertices $a,b,c$, so $D(\mathcal{P},a)$ is infinite cyclic. From this example, it is possible to construct new diagram groups by combination. 

\begin{prop}[\cite{MR1396957}]\label{prop:FreeProduct}
Free products of diagram groups are diagram groups.
\end{prop}

\begin{proof}
Let $I$ be a set. For every $i \in I$, let $\mathcal{P}_i:= \langle \Sigma_i \mid \mathcal{R}_i \rangle$ be a semigroup presentation and $w_i \in \Sigma_i^+$ a baseword. Without loss of generality, we assume that the alphabets $\Sigma_i$ are pairwise disjoint. Set
$$\mathcal{Q}:= \left\langle \{o\} \sqcup \bigsqcup\limits_{i \in I} \Sigma_i \mid \bigsqcup\limits_{i \in I} \mathcal{R}_i \sqcup \{o=w_i, i \in I\} \right\rangle.$$
Then $D(\mathcal{Q},o)$ is isomorphic to the free product $\ast_{i \in I} D(\mathcal{P}_i,w_i)$ because $S(\mathcal{Q},o)$ coincides with the disjoint union of the $S(\mathcal{P}_i,w_i)$ together with the new vertex $o$ which is adjacent to all the $w_i$. 
\end{proof}

\noindent
Thus, free groups are diagram groups. By applying the construction above, a group freely generated by a set $I$ can be represented as the diagram group given by the semigroup presentation
$$\langle o, a_i,b_i,c_i \ (i \in I) \mid o=a_i, a_i=b_i, b_i=c_i, c_i=a_i (i \in I) \rangle$$
and the baseword $o$. 

\begin{prop}[\cite{MR1396957}]\label{prop:DirectSum}
Direct sums of finitely many diagram groups are diagram groups.
\end{prop}

\begin{proof}
Fix an integer $n \geq 1$, and, for every $1 \leq i \leq n$, let $\mathcal{P}_i:= \langle \Sigma_i \mid \mathcal{R}_i \rangle$ be a semigroup presentation and $w_i \in \Sigma_i^+$ a baseword. Without loss of generality, we assume that the alphabets $\Sigma_i$ are pairwise disjoint. Set
$$\mathcal{Q}:= \left\langle \bigsqcup\limits_{i=1}^n \Sigma_i \mid \bigsqcup\limits_{i=1}^n \mathcal{R}_i \right\rangle$$
and $w:=w_1 \cdots w_n$. Then $D(\mathcal{Q},w)$ is isomorphic to $D(\mathcal{P}_1,w_1) \oplus \cdots \oplus D(\mathcal{P}_n,w_n)$ because $S(\mathcal{Q},w)$ coincides with the ($2$-skeleton of the) product $S(\mathcal{P}_1,w_1) \times \cdots \times S(\mathcal{P}_n,w_n)$. 
\end{proof}

\noindent
Thus, free abelian groups of finite ranks are diagram groups. By applying the construction above, if
$$\mathcal{P}= \langle a,b,c \mid a=b,b=c,c=a \rangle$$ 
then $D(\mathcal{P},a^n)$ is free abelian of rank $n$. 

\medskip \noindent
Comparing Proposition~\ref{prop:DirectSum} with Proposition~\ref{prop:FreeProduct}, it is natural to ask whether direct sums of infinitely many diagram groups are still diagram groups. The construction used in Proposition~\ref{prop:DirectSum} does not work anymore, but one can apply a trick:

\begin{prop}[\cite{MR1725439}]
Directs sums of countably many diagram groups are diagram groups.
\end{prop}

\begin{proof}[Sketch of proof.]
For every $i \geq 1$, let $\mathcal{P}_i:= \langle \Sigma_i \mid \mathcal{R}_i \rangle$ be a semigroup presentation and $w_i \in \Sigma_i^+$ a baseword. Without loss of generality, we assume that the alphabets $\Sigma_i$ are pairwise disjoint. Following Proposition~\ref{prop:DirectSum}, what we would like to do is to consider the infinite word $w:=w_1w_2 \cdots$ and to apply derivations on the $w_i$ independently according to the $\mathcal{P}_i$ respectively. However, infinite words are not allowed. The trick is to introduce a new letter $x$ in the alphabet and to ``hidde'' the infinite part of our $w$ inside $x$. More precisely, set
$$\mathcal{Q}:= \left\langle \{x\} \sqcup \bigsqcup\limits_{i \geq 1} \Sigma_i \mid \bigsqcup\limits_{i \geq 1} \mathcal{R}_i \sqcup \{w_1 \cdots w_nx=w_1 \cdots w_{n+1}x, n\geq 0\} \right\rangle.$$
The claim is that $D(\mathcal{Q},x)$ is isomorphic to $\bigoplus_{i\geq 1} D(\mathcal{P}_i,w_i)$. For every $i \geq 1$, derivations of the form
$$x \to w_1 x \to \cdots \to \underset{\text{derivation in $\mathcal{P}_i$ applied to $w_i$}}{\underbrace{w_1 \cdots w_i x \to \cdots \to w_1 \cdots w_i x}} \to w_1 \cdots w_{i-1}x \to \cdots \to x$$
correspond to the elements of the factor $D(\mathcal{P}_i,w_i)$ in our direct sum. 
\end{proof}

\noindent
Thus, a free abelian group of countable rank is a diagram group. By applying the construction above, it can be represented as the diagram group given by the semigroup presentation
$$\left\langle x,a,b,c \mid a=b,b=c,c=a, a^nx=a^{n+1}x \ (n \geq 0) \right\rangle$$
and the baseword $x$. 

\begin{prop}[\cite{MR1725439}]
If $G$ is a diagram group, then so is $G \wr \mathbb{Z}$. 
\end{prop}

\noindent
Recall that the \emph{wreath product} $A \wr B$ of two groups $A,B$ is the semidirect product $\bigoplus_B A \rtimes B$ where $B$ acts on the direct sum by permuting the factors according to the action on itself by left multiplication. 

\begin{proof}[Sketch of proof.]
Let $\mathcal{P} = \langle \Sigma \mid \mathcal{R} \rangle$ be a semigroup presentation and $w \in \Sigma^+$ a baseword such that $G$ coincides with the diagram group $D(\mathcal{P},w)$. Let $a,b,s$ be two new letters that do not belong to $\Sigma$. Set
$$\mathcal{Q}:= \langle \Sigma \sqcup \{a,b,s\} \mid \mathcal{R} \sqcup \{a=asws,b=swsb\} \rangle.$$
The letter $s$ can be thought of as a \emph{separation letter}. It prevents us from applying a relation from $\mathcal{R}$ to two adjacent $w$. This letter can be removed when no relation from $\mathcal{R}$ can be applied to a subword of some $w \cdots w$ overlapping between two copies of $w$. 

\medskip \noindent
The claim is that $D(\mathcal{Q},aswsb)$ is isomorphic to $D(\mathcal{P},w) \wr \mathbb{Z}$. For every diagram $\Delta$ over $\mathcal{Q}$, let $\eta(\Delta)$ denote the number of $(b,bsws)$-cells in $\Delta$ minus the number of $(bsws,b)$-cells in $\Delta$; it is worth noticing that $\eta(\Delta)$ depends only on the class of $\Delta$ modulo dipole reduction. One gets a morphism $\eta : D(\mathcal{Q},aswsb) \twoheadrightarrow \mathbb{Z}$. For instance, the diagram $\Xi$ associated to the derivation 
$$aswsb \ \underset{b=swsb}{\longrightarrow} \ aswsswsb \ \underset{asws=a}{\longrightarrow} \ aswsb$$
is sent to $1$ under $\eta$. For every $i \geq 0$, $D(\mathcal{Q},aswsb)$ contains a copy $D_i$ of $D(\mathcal{P},w)$ whose elements correspond to the derivations of the form
\small $$aswsb \underset{b=swsb}{\longrightarrow} \cdots \underset{b=swsb}{\longrightarrow} \underset{\text{derivation in $\mathcal{P}$ applied to the rightmost $w$}}{\underbrace{a(sws)^{i-1}swsb \longrightarrow \cdots \longrightarrow a(sws^{i-1})swsb}} \underset{swsb=b}{\longrightarrow} \cdots \underset{swsb=b}{\longrightarrow} aswsb.$$
\normalsize Similarly, for every $i < 0$, $D(\mathcal{Q},aswsb)$ contains a copy $D_i$ of $D(\mathcal{P},w)$ whose elements correspond to the derivations of the form
\small $$aswsb \underset{a=swsa}{\longrightarrow} \cdots \underset{a=swsa}{\longrightarrow} \underset{\text{derivation in $\mathcal{P}$ applied to the leftmost $w$}}{\underbrace{asws(sws)^{i-1}b \longrightarrow \cdots \longrightarrow asws(sws)^{i-1}b}} \underset{swsa=a}{\longrightarrow} \cdots \underset{swsa=a}{\longrightarrow} aswsb.$$
\normalsize Clearly, each $D_i$ lies in the kernel of $\eta$. In fact, one can shows that the $D_i$ generate $\mathrm{ker}(\eta)$. Therefore, $D(\mathcal{Q},aswsb)$ decomposes as $\langle D_i, i \in \mathbb{Z} \rangle \rtimes \langle \Xi \rangle$. In order to conclude, it suffices to verify that $\langle D_i , i \in \mathbb{Z} \rangle$ decomposes as $\bigoplus_{i \in \mathbb{Z}} D_i$ and that $\Xi \cdot D_i \cdot \Xi^{-1} = D_{i+1}$ for every index $i \in \mathbb{Z}$. 

\medskip \noindent
It is worth noticing that $D(\mathcal{Q},ab)$ is isomorphic to $D(\mathcal{Q},aswsb)$, because $ab$ and $aswsb$ belong to the same connected component of the Squier complex $S(\mathcal{P})$. Therefore, $D(\mathcal{Q},ab)$ can also be taken as a description of $D(\mathcal{P},w) \wr \mathbb{Z}$ as a diagram group.
\end{proof}

\noindent
Thus, the wreath product $\mathbb{Z} \wr \mathbb{Z}$ is a diagram group. By applying the construction above, it can be represented as the diagram group given by the semigroup presentation
$$\left\langle a,b,p_1,p_2,p_3 \mid \begin{array}{c} p_1=p_2, p_2=p_3, p_3=p_1 \\ a=ap_1, b=p_1b \end{array} \right\rangle$$
and the baseword $ab$. 

\medskip \noindent
Let us mention two last group operations preserving diagram groups. First, given two groups $A$ and $B$, let $A \bullet B$ denote the group given by the (relative) presentation
$$\left\langle A,B,t \mid \left[ a,t^n bt^{-n} \right]=1 \text{ for all $a \in A$, $b \in B$, and $n \geq 0$} \right\rangle.$$
As an alternative description \cite[Example~4.29]{MR4033512}, $A \bullet B$ can be described as the subgroup $\langle A,B, yx \rangle$ in $(A \ast \mathbb{Z}) \times (B \ast \mathbb{Z})$ where $x$ (resp. $y$) denotes a generator of the left (resp. right) factor $\mathbb{Z}$. For instance, the group $\mathbb{Z} \bullet \mathbb{Z} = \langle a,b,t \mid [a,t^nbt^{-n}]=1, n \geq 0\rangle$ coincides with the kernel of the morphism $\mathbb{F}_2 \times \mathbb{F}_2 \twoheadrightarrow \mathbb{Z}$ that sends each generator to $1$ (sometimes referred to as the \emph{Bestvina-Brady group} of $\mathbb{F}_2 \times \mathbb{F}_2$). 

\begin{prop}[\cite{MR1396957}]
The $\bullet$-product of two diagram groups is again a diagram group.
\end{prop}

\begin{proof}[Sketch of proof.]
Let $\mathcal{P}_1= \langle \Sigma_1 \mid \mathcal{R}_1 \rangle$ and $\mathcal{P}_2= \langle \Sigma_2 \mid \mathcal{R}_2 \rangle$ be two semigroup presentations. Without loss of generality, we assume that $\Sigma_1 \cap \Sigma_2 = \emptyset$. Fix two basewords $w_1 \in \Sigma_1^+$ and $w_2 \in \Sigma_2^+$. The claim is that, given the semigroup presentation
$$\mathcal{Q}:= \left\langle \Sigma_1 \sqcup \Sigma_2 \sqcup \{p\} \mid \mathcal{R}_1 \sqcup \mathcal{R}_2 \sqcup \{ w_1=w_1p, w_2=pw_2\} \right\rangle,$$
the diagram group $D(\mathcal{Q},w_1w_2)$ coincides with $D(\mathcal{P}_1,w_1) \bullet D(\mathcal{P}_2,w_2)$. More precisely, $D(\mathcal{Q},w_1w_2)$ naturally contains copies $D_1$ and $D_2$ of $D(\mathcal{P}_1,w_1)$ and $D(\mathcal{P}_2,w_2)$, corresponding to applications of derivations in $\mathcal{P}_1$ and $\mathcal{P}_2$ to $w_1$ and $w_2$. Then an isomorphism $D(\mathcal{P}_1,w_1) \bullet D(\mathcal{P}_2,w_2) \to D(\mathcal{Q},w_1w_2)$ is obtained by sending $D(\mathcal{P}_1,w_1)$ to $D_1$, $D(\mathcal{P}_2,w_2)$ to $D_2$, and the generator $t$ to the diagram $\Xi$ associated to the derivation $w_1w_2 \underset{w_1=w_1p}{\longrightarrow} w_1pw_2 \underset{pw_2=w_2}{\longrightarrow} w_1w_2$. 
\end{proof}

\noindent
Thus, the group $\mathbb{Z} \bullet \mathbb{Z}$ previously mentioned is a diagram group, and, by applying the construction above, it can be described as the diagram group given by the semigroup presentation
$$\left\langle a_1,a_2,a_3,b_1,b_2,b_3,p \mid \begin{array}{c} a_1=a_2, a_2=a_3, a_3=a_1 \\ b_1=b_2, b_2=b_3, b_3=b_1 \end{array}, a_1=a_1p, b_1=pb_1 \right\rangle$$
and the baseword $a_1b_1$. 

\noindent
Next, given two groups $A$ and $B$, let $A \square B$ denote the group given by the (relative) presentation
$$\langle A,B,t \mid [a,b]=\left[ a,tbt^{-1} \right]=1 \text{ for all } a \in A, b \in B \rangle.$$
For instance, $\mathbb{Z} \square \mathbb{Z}= \langle a,b,t \mid [a,b]=[a,tbt^{-1}]=1 \rangle$. 

\begin{prop}[\cite{QM}]
The $\square$-product of two diagram groups is again a diagram group.
\end{prop}

\begin{proof}
Let $\mathcal{P}_1= \langle \Sigma_1 \mid \mathcal{R}_1 \rangle$ and $\mathcal{P}_2= \langle \Sigma_2 \mid \mathcal{R}_2 \rangle$ be two semigroup presentations. Without loss of generality, we assume that $\Sigma_1 \cap \Sigma_2 = \emptyset$. Fix two basewords $w_1 \in \Sigma_1^+$ and $w_2 \in \Sigma_2^+$. The claim is that, given the semigroup presentation
$$\mathcal{Q}:= \langle \Sigma_1 \sqcup \Sigma_2 \sqcup \{x,y\} \mid \mathcal{R}_1 \sqcup \mathcal{R}_2 \sqcup \{ w_1x=w_1y, xw_2=yw_2\} \rangle,$$
the diagram group $D(\mathcal{Q},w_1xw_2)$ coincides with $D(\mathcal{P}_1,w_1) \square D(\mathcal{P}_2,w_2)$. The key observation is that the Squier complex $S(\mathcal{Q},w_1xw_2)$ contains two natural copies (of the two-skeleton) of $S(\mathcal{P}_1,w_1) \times S(\mathcal{P}_2,w_2)$, which we denote by $S(\mathcal{P}_1,w_1)x S(\mathcal{P}_2,w_2)$ and $S(\mathcal{P}_1,w_1)y S(\mathcal{P}_2,w_2)$. Observe that the two copies of $S(\mathcal{P}_1,w_1)$ span (the two-skeleton of) a product $S(\mathcal{P}_1,w_1) \times [0,1]$; and similarly for $S(\mathcal{P}_2,w_2)$. Thus, we get a decomposition of $S(\mathcal{Q},w_1xw_2)$ as a graph of spaces which implies that $D(\mathcal{Q},w_1xw_2)$ decomposes as the fundamental group of graph of groups with two vertices, both labelled by $D(\mathcal{P}_1,w_1) \times D(\mathcal{P}_2,w_2)$, and two edges, respectively identifying the two copies of $D(\mathcal{P}_1,w_1)$ and $D(\mathcal{P}_2,w_2)$. Hence the relative presentation
$$\langle A_1,B_1,A_2,B_2, t \mid [A_1,B_1]=[A_2,B_2]=1, tB_1t^{-1}=B_2 \rangle,$$
where $A_1,A_2$ (resp. $B_1,B_2$) are two copies of $A$ (resp. of $B$), which can be simplified as the presentation given above for $A \square B$. 
\end{proof}

\noindent
As an application, we can describe our group $\mathbb{Z} \square \mathbb{Z}$ as the diagram group given by the semigroup presentation
$$\left\langle a_1,a_2,a_3, b_1,b_2,b_3, x,y \mid \begin{array}{c} a_1=a_2,a_2=a_3,a_3=a_1 \\ b_1=b_2,b_2=b_3,b_3=b_1 \end{array}, a_1x=a_1y, xb_1=yb_1 \right\rangle$$
and the baseword $a_1xb_1$. 

\medskip \noindent
Given an integer $n \geq 1$, the \emph{planar braid group} $T_n$, also referred to as the \emph{twin group} or the group of \emph{flat braids}, is the fundamental group of
$$\{ (x_1, \ldots, x_n) \in \mathbb{R}^n \mid \forall i, \# \{ j \mid x_j=x_i \} \leq 1 \} / S_n,$$
i.e.\ the configuration space of $n$ (unordered) points moving in $\mathbb{R}$, where collisions between two points, but no more, is allowed. Elements of $T_n$ can be represented as \emph{planar braids} subject to the following relations:
\begin{center}
\includegraphics[width=0.7\linewidth]{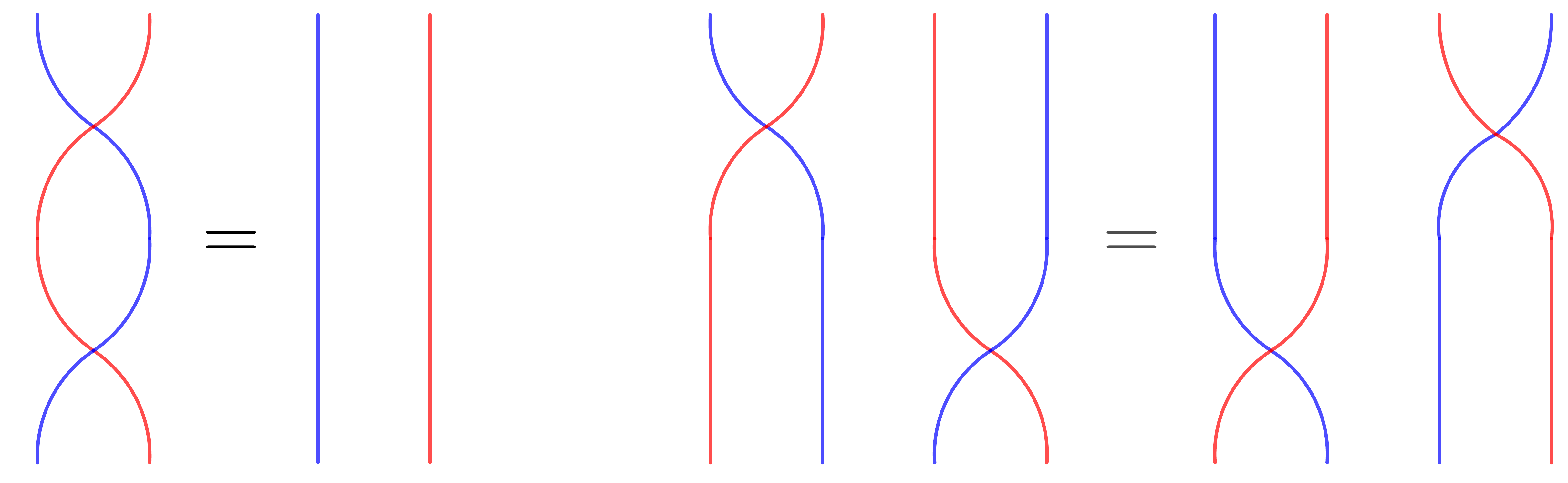}
\end{center}

\medskip \noindent
The planar braid group $T_n$ turns out to be naturally isomorphic to the right-angled Coxeter group
$$\langle \tau_1, \ldots, \tau_{n-1} \mid \tau_i^2=1 \ (1 \leq i \leq n-1), \ [\tau_i, \tau_j]=1 \ (|i-j| \geq 2) \rangle.$$
The kernel of the natural morphism $T_n \to S_n$ is the \emph{pure planar braid group} $PT_n$. 

\begin{prop}[\cite{MR4071367, FarleyTwin}]\label{prop:PlanarBraid}
For every $n \geq 1$, the pure planar braid group $PT_n$ is isomorphic to $D(\mathcal{P}_n,x_1 \cdots x_n)$ where $\mathcal{P}_n:= \langle x_1, \ldots, x_n \mid x_ix_j=x_jx_i \ (1 \leq i < j \leq n) \rangle$. 
\end{prop}

\noindent
\begin{minipage}{0.38\linewidth}
\includegraphics[width=0.9\linewidth]{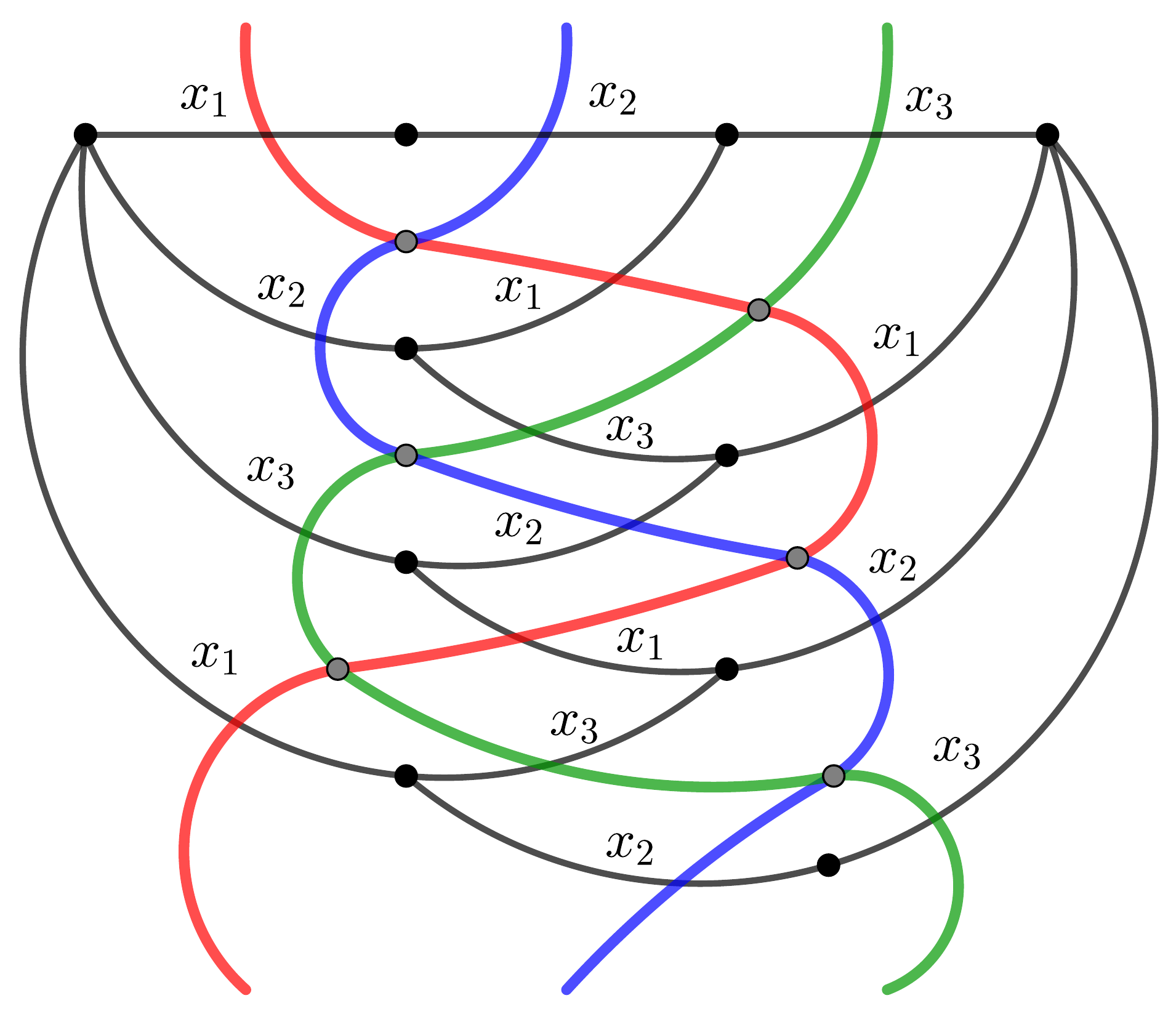}
\end{minipage}
\begin{minipage}{0.6\linewidth}
The isomorphism between the diagram group $D(\mathcal{P}_n,x_1 \cdots x_n)$ and the planar braid group $PT_n$ is illustrated by the figure on the left. It shows that every diagram over $\mathcal{P}_n$ with top path labelled $x_1 \cdots x_n$ is naturally dual to a planar braid. The correspondence is clearly compatible with the group structures. 
\end{minipage}

\medskip \noindent
It is worth noticing that diagram groups given by the same semigroup presentation $\mathcal{P}_n$ as in Proposition~\ref{prop:PlanarBraid} but with different basewords lead to interesting variations. For instance, $D(\mathcal{P}_n,x_1^{r_1} \cdots x_n^{r_n})$ corresponds to the group of the pure planar braids with $r_1 + \cdots + r_n$ strands, coloured with $n$ colours such that there are $r_i$ strands coloured with the $i$th colour, such that no two strands with the same colour are allowed to cross. 

\medskip \noindent
Now, let us turn probably to the most famous example of diagram group, namely Thompson's group $F$. There are many possible definitions of $F$. The shortest is the following: $F$ is the group of the increasing piecewise linear homeomorphisms $[0,1] \to [0,1]$ whose breakpoints have dyadic coordinates and whose slopes are powers of $2$. We refer to \cite{MR1426438} for more information on Thompson's groups. 

\begin{prop}[\cite{MR1448329}]\label{prop:Thompson}
The diagram group $D(\mathcal{P},x)$ where $\mathcal{P}:= \langle x \mid x=x^2 \rangle$ is isomorphic to the Thompson group $F$. 
\end{prop}

\noindent
The key observation is that a diagram $\Delta$ over $\mathcal{P}$ with no dipoles decomposes as the concatenation of two diagrams $\Delta^+ \cdot \Delta^-$ where the $2$-cells of $\Delta^+$ (resp. $\Delta^-$) are labelled by the relation $x\to x^2$ (resp. $x^2 \to x$). The figure below illustrates how to associated an element of $F$ from an element of $D(\mathcal{P},w)$ both in terms of pairs of trees and of piecewise linear homeomorphisms.
\begin{center}
\includegraphics[width=0.9\linewidth]{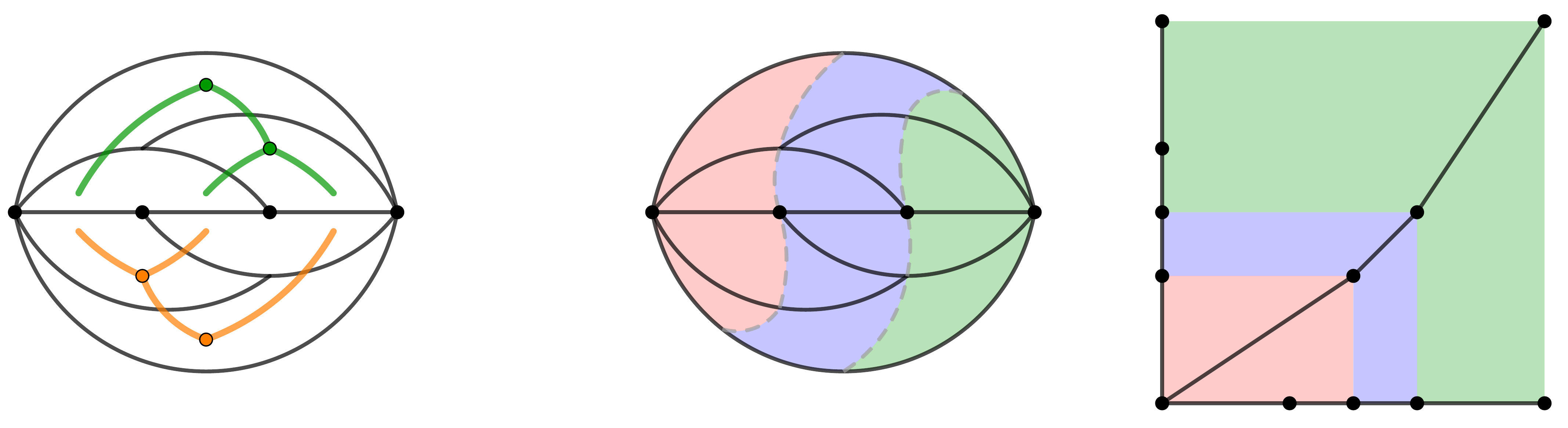}
\end{center}

\noindent
In the right part of the figure, one should think of a $2$-cell from the top part of the diagram as a map $[0,1] \to [0,2]$ defined by $x \mapsto 2x$ and of a $2$-cell from the bottom part of the diagram as a map $[0,2] \to [0,1]$ defined by $x \mapsto x/2$. This is a particular case of the construction mentioned in Section~\ref{section:OpenQuestion} after Question~\ref{question:ResiduallyPL}. 

\medskip \noindent
Clearly, the arguments above generalise to every Thompson's group $F_n$ (defined from $n$-adic numbers instead of $2$-adic numbers). Thus, the diagram groups $D(\mathcal{P}_n,x)$ given by $\mathcal{P}_n:= \langle x \mid x=x^n \rangle$ is isomorphic to $F_n$ for every $n \geq 2$. 

\medskip \noindent
Interestingly, there turns out to exist (algebraically) simple diagram groups:

\begin{prop}[\cite{MR1725439}]\label{prop:CommutatorF}
Given the semigroup presentation
$$\mathcal{P}:= \left\langle x, a_i,b_i \ (i \geq 1) \mid x=x^2, a_i=a_{i+1}x, b_i=xb_{i+1} \ (i \geq 1) \right\rangle,$$ 
the diagram group $D(\mathcal{P},a_1b_1)$ is isomorphic to the commutator subgroup $F'$ of Thompson's group $F$. 
\end{prop}

\noindent
Recall that the elements in $F'$ coincide with the homeomorphisms $[0,1] \to [0,1]$ fixing pointwise neighbourhoods of $0$ and $1$. In terms of pairs of trees, the elements in $F'$ correspond to the pairs $(T_1,T_2)$ such that the leftmost leaves of $T_1$ and $T_2$ have the same \emph{dyadic representations}, and same thing for the rightmost leaves of $T_1$ and $T_2$. Here, the dyadic representation of a vertex in a planar $2$-regular rooted tree refers to the digits associated to it when we label the root by $\emptyset$, its left child by $0$, its right child by $1$, the left child of $0$ by $00$, the right child of $0$ by $01$, the left child of $1$ by $10$, etc. For instance, the leftmost leaves of the top and bottom trees in the figure below have $00$ as dyadic representatives. However, the leftmost leaves of the top and bottom trees in the figure above have $0$ and $00$ as dyadic representatives; therefore, the element of $F$ represented does not belong to the commutator subgroup $F'$. (See also Example~\ref{ex:CommutatorF} below for an equivalent description of $F'$ in $F$.)

\medskip \noindent
The figure below shows how to associate a pair of trees, and a fortiori an element of $F$, to every diagram with no dipole over the presentation given by Proposition~\ref{prop:CommutatorF}. The key observation is that such a diagram corresponds to a derivation starting with with the moves
$$a_1 \to a_2x \to a_3xx \to \cdots \to a_rx^{r-1} \text{ and } b_1 \to xb_2 \to xxb_3 \to \cdots \to x^{s-1}b_s$$
and ending with the moves
$$a_rx^{r-1} \to a_{r-1}x^{r-2} \to \cdots a_1 \text{ and } x^{s-1}b_s \to x^{s-2}b_{s-1} \to \cdots \to b_1.$$
This is why the leftmost (resp. rightmost) leaves in the top and bottom trees have the same dyadic representation. 

\begin{center}
\includegraphics[width=0.7\linewidth]{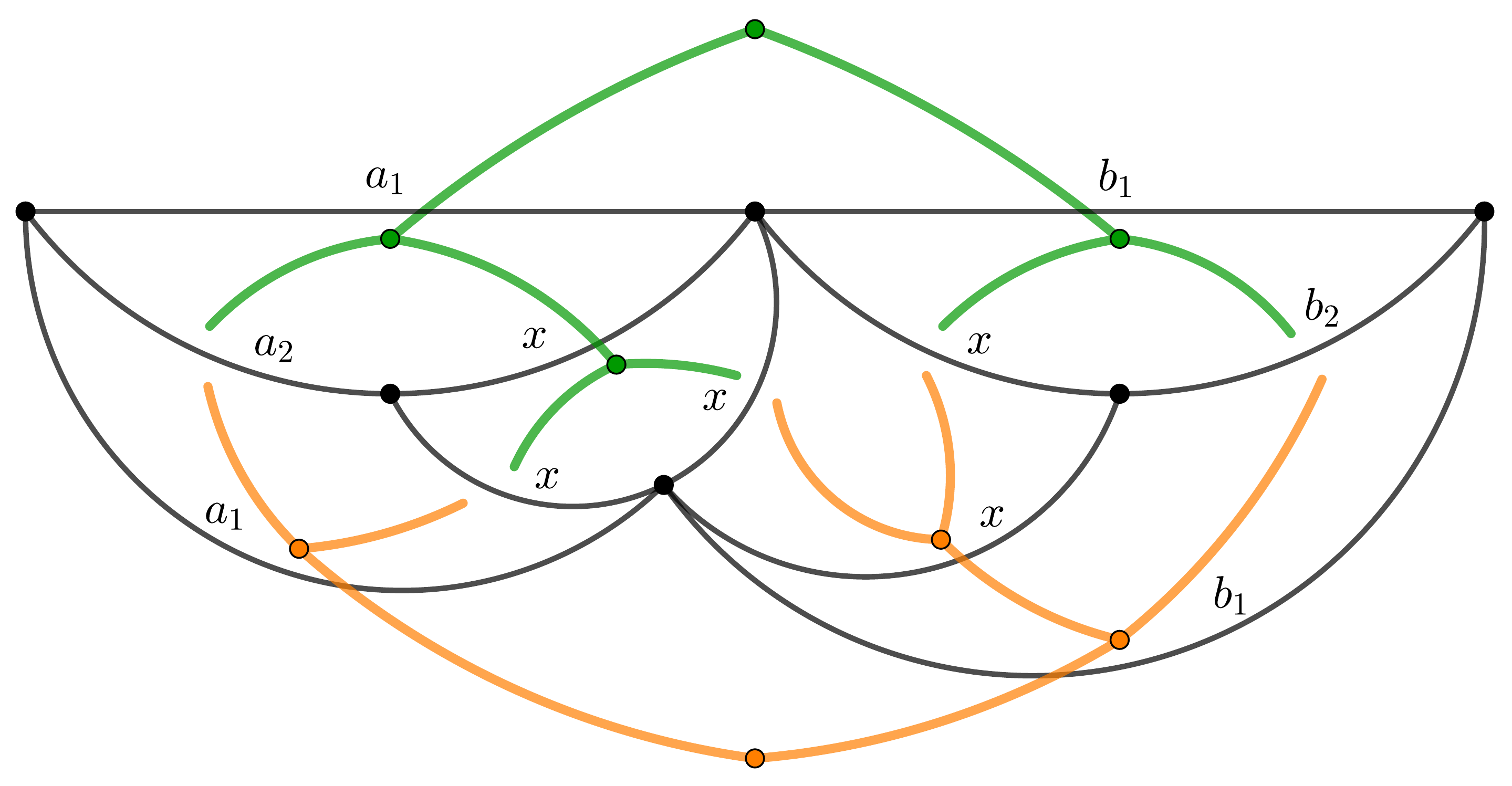}
\end{center}

\noindent
Let us end this section by mentioning a last interesting families of groups. Given a (simplicial) graph $\Gamma$, the \emph{right-angled Artin group} $A(\Gamma)$, also referred to as \emph{partially commutative group}, is given by the presentation
$$\left\langle u \in V(\Gamma) \mid [u,v]=1, \ \{u,v\} \in E(\Gamma) \right\rangle$$
where $V(\Gamma)$ and $E(\Gamma)$ denote the vertex- and edge-sets of $\Gamma$. Because diagram groups include $\mathbb{Z}$ and are stable under direct sums and free products, we already know that many right-angled Artin groups are diagram groups. However, not every right-angled Artin group is a diagram group, and the question of which right-angled Artin groups are diagram groups is still open; see Question~\ref{question:RAAG}. 

\medskip \noindent
An easily recognisable family of right-angled Artin groups that are diagram groups is the following. Recall that a graph $\Gamma$ is an \emph{interval graph} if there exist an $n \geq 1$ and a collection $\mathcal{C}$ of intervals in $\{1, \ldots, n\}$ such that $\mathcal{C}$ is the vertex-set of $\Gamma$ and such that two intervals are connected by an edge whenever they are disjoint. 

\begin{prop}[\cite{MR3868219}]\label{prop:RAAGinterval}
If $\Gamma$ is a finite interval graph, then $A(\Gamma)$ is a diagram group.
\end{prop}

\noindent
More precisely, if our interval graph $\Gamma$ is given by a collection $\mathcal{C}$ of intervals in $\{1, \ldots, n\}$, then the diagram group given by the baseword $x_1 \cdots x_n$ and the semigroup presentation
$$\left\langle x_1, \ldots, x_n , a_I,b_I,c_I \ (I \in \mathcal{C}) \mid x_I=a_I, a_I=b_I, b_I=c_I \ (I \in \mathcal{C}) \right\rangle,$$
where $x_I$ is a shorthand for $x_{i_1} x_{i_2} \cdots x_{i_s}$ if $I= \{i_1< i_2< \cdots < _s \}$, is naturally isomorphic to the right-angled Artin group $A(\Gamma)$. 

\medskip \noindent
Another source of examples can be found in \cite{MR2193191}. As a particular case of interest:

\begin{prop}[\cite{MR2193191}]\label{prop:RAAGtree}
Right-angled Artin groups over finite trees are diagram groups.
\end{prop}

\noindent
In fact, one can prove something stronger. One can prove that, if $\Gamma$ is either a finite interval graph or a finite tree, then a graph product of diagram groups over $\Gamma$ is again a diagram group. Recall that, given a (simplicial) graph $\Gamma$ and a collection of groups $\mathcal{G}= \{G_u \mid u \in V(\Gamma)\}$ indexed by the vertices of $\Gamma$, the \emph{graph product} $\Gamma \mathcal{G}$ is given by the (relative) presentation
$$\left\langle G_u, \ u \in V(\Gamma) \mid [G_u,G_v]=1, \ \{u,v\} \in E(\Gamma) \right\rangle,$$
where $[G_u,G_v]=1$ is a shorthand for $[a,b]=1$ for all $a \in G_u$ and $b \in G_v$. Notice that graph products of infinite cyclic groups coincide with right-angled Artin groups.

\subsection{First properties of diagram groups}\label{section:Properties}

\noindent
So far, we have described diagram groups from several perspectives in Sections~\ref{section:Asphericity},~\ref{section:Diagrams}, and~\ref{section:Free}, and we have shown in Section~\ref{section:Examples} that there exist many interesting groups that turn out to be diagram groups. Now, the natural question is: what can we learn from the fact that a given group can be described as a diagram group? In this section, our goal is to show that, despite the fact that the family of diagram groups is very broad, being a diagram groups turns out to be quite restrictive; or, in other words, it provides valuable information on the group. As a complement of the previous section, this will also allow us to give examples of groups that are not diagram groups. 

\medskip \noindent
For the statements below, we include ideas of proofs, following the arguments given in the original articles. However, as it will be explained in Section~\ref{section:Median}, some of these arguments based on the combinatorics of diagrams follow almost for free from the median geometry described in Section~\ref{section:MedianDiag}.

\medskip \noindent
Let us first explain that diagram groups are torsion-free, which already prevents many groups from being diagram groups. 

\begin{thm}[\cite{MR1448329, MR1396957}]\label{thm:TorsionFree}
Diagram groups are torsion-free.
\end{thm}

\noindent
Let $\mathcal{P}=\langle \Sigma \mid \mathcal{R} \rangle$ be a semigroup presentation. A diagram $\Delta$ over $\mathcal{P}$ is \emph{reduced} if it does not contain any dipole. Assuming that $\Delta$ is \emph{spherical} (i.e.\ if its top and bottom paths have the same label), one says that $\Delta$ is \emph{absolutely reduced} if $\Delta^n$ is reduced for every $n \geq 1$. A key property is that every reduced spherical diagram $\Delta$ decomposes (up to dipole reduction) as $\Psi \circ \Delta_0 \circ \Psi^{-1}$ for some absolutely reduced diagram $\Delta_0$ \cite[Lemma~15.10]{MR1396957}. We emphasize that $\Psi$ may not be spherical. 

\medskip \noindent
Consequently, given a baseword $w \in \Sigma^+$ and a reduced diagram $\Delta \in D(\mathcal{P},w)$, we decompose $\Delta$ as $\Psi \circ \Delta_0 \circ \Psi^{-1}$ for some absolutely reduced diagram $\Delta_0$. Then $\Delta^n$ is trivial in $D(\mathcal{P},w)$ if and only if so is $\Delta_0^n$ in $D(\mathcal{P},w_0)$, where $w_0$ denotes the word labelling the top and bottom paths of $\Delta_0$. But $\Delta_0$ is absolutely reduced, so $\Delta_0^n$ is trivial if and only if $\Delta_0$ is itself trivial, which amounts to saying that $\Delta$ is trivial. Thus, we have shown that the only torsion element in $D(\mathcal{P},w)$ is the trivial element. 

\begin{thm}[\cite{MR1725439}]\label{thm:Nilpotent}
In diagram groups, nilpotent and polycyclic subgroups are free abelian and undistorted when they are finitely generated. 
\end{thm}

\noindent
We refer to Sections~\ref{section:Commutation} and~\ref{section:MedianDiag} for elements of proofs. 

\medskip \noindent
As a consequence of Theorem~\ref{thm:Nilpotent}, there are no interesting examples of nilpotent or polycyclic diagram groups. For instance, the fundamental group of the Klein bottle or the Heisenberg groups are not diagram groups. Nevertheless, there are metabelian but not abelian diagram groups, such as $\mathbb{Z} \wr \mathbb{Z}$; as well as solvable diagram groups of arbitrary lengths, such as $((\mathbb{Z} \wr \mathbb{Z}) \wr \mathbb{Z}) \cdots \wr \mathbb{Z}$. The absence of distortion given by Theorem~\ref{thm:Nilpotent} also allows us to discard non-amenable groups from diagram groups, such as most Baumslag-Solitar groups. 

\begin{thm}[\cite{HypDiag}]\label{thm:HypDiag}
A diagram group with no $\mathbb{Z}^2$ is locally free. 
\end{thm}

\noindent
The proof of this proposition is fundamentally based on the median geometry that will be described in Section~\ref{section:MedianDiag}. The idea is the following. Let $D(\mathcal{P},w)$ be a diagram group with no $\mathbb{Z}^2$ and let $M(\mathcal{P},w)$ denote the median graph on which it acts naturally. Given a finitely generated subgroup $H \leq D(\mathcal{P},w)$, we want to prove that $H$ is free. We can assume without loss of generality that $H$ is freely irreducible. The trick is to choose carefully some hyperplane $J$ in $M(\mathcal{P},w)$ such that no two $H$-translates of $J$ are transverse and the $H$-stabiliser of $J$ is trivial. Thus, the orbit $H \cdot J$ induces an arboreal structure on $M(\mathcal{P},w)$ and $H$ acts on the dual tree (non-trivially and) with trivial edge-stabilisers. This implies that $H$ splits over the trivial group, which amounts to saying that $H$ is either trivial or infinite cyclic, as desired. Of course, the difficulty is to choose the good hyperplane $J$, and this is where we use the assumption that there is not $\mathbb{Z}^2$ in $D(\mathcal{P},w)$. 

\medskip \noindent
It follows from Theorem~\ref{thm:HypDiag} that free groups are the only hyperbolic groups that are diagram groups. In particular, surface groups are not diagram groups. Nevertheless, it is worth mentioning that surface groups are subgroups of diagram groups. Indeed, the \emph{complement graph} $\overline{P_7}$ of the path $P_7$ of length seven (i.e.\ the graph whose vertices are the vertices of $P_7$ and whose edges connect two vertices whenever they are not adjacent in $P_7$) is an interval graph, so it follows from Proposition~\ref{prop:RAAGinterval} that the right-angled Artin group $A(\overline{P_7})$ is a diagram group. This group contains $A(C_n)$ for every $n \geq 5$, where $C_n$ denotes the cycle of length $n$ \cite[Corollary~4.4]{MR3072113}; and a fortiori every surface group \cite[Corollary~4.5]{MR3072113}. (Here, by a surface group, we mean the fundamental group of an orientable closed surface of genus $\geq 2$.) As a corollary, it follows that subgroups of diagram groups may not be diagram groups themselves.

\begin{thm}[\cite{MR2193191}]\label{thm:BiOrderable}
Diagram groups are bi-orderable, hence locally indicable.
\end{thm}

\noindent
Recall that a group $G$ is \emph{(bi-)orderable} if there exists a total order $\leq$ and $G$ such that, for all $a,b,c \in G$, if $a \leq b$ then $ac \leq bc$ (and $ca \leq cb$). A group is \emph{locally indicable} if all its finitely generated subgroups surjects onto $\mathbb{Z}$. Locally indicable groups are automatically orderable, but the converse does not hold. However, bi-orderable groups are locally indicable.

\medskip \noindent
It has been first proved in \cite{MR1983088} that diagram groups are orderable. (See also Corollary~\ref{cor:LocallyIndicable} below.) Both \cite{MR1983088} and \cite{MR2193191} use the same strategy: first, they construct a ``universal'' group containing all the diagram groups; and next they show this universal group is (bi-)orderable. The universal groups are however very different in the two articles. In \cite{MR1983088}, this is a braid groups on infinitely many strands. In \cite{MR2193191}, this is a well-chosen diagram group which turns out to splits as a semi-direct product between a right-angled Artin group and Thompson's group $F$. 

\medskip \noindent
Observe that, as an immediate consequence of Theorem~\ref{thm:BiOrderable}, no finitely generated diagram group can be simple. This contrasts with the fact that the commutator subgroup $F'$ of Thompson's group $F$, which is simple but not finitely generated, is a diagram group (Proposition~\ref{prop:CommutatorF}). 

\begin{cor}[\cite{MR1396957}]\label{cor:Roots}
Every diagram group $G$ satisfies the \emph{unique extraction of roots property}, i.e.\ for all $a,b \in G$ and $k \geq 1$, if $a^k=b^k$ then $a=b$.
\end{cor}

\noindent
This a consequence of the orderability of $G$. Indeed, given some $a,b \in G$ satisfying $a^k=b^k$ for some $k \geq 1$, it is clear that $a<b$ implies that 
$$a^k = a \cdots aa < a \cdots ab < a \cdots bb < \cdots < b \cdots bb=b^k,$$
which would contradict that $a^k=b^k$; and similarly for $b<a$. Therefore, we must have $a=b$, as desired. Nevertheless, it is possible and instructive to prove directly Corollary~\ref{cor:Roots} thanks to the combinatorics of diagrams. See \cite[Theorem~15.25]{MR1396957} for more details. 

\medskip \noindent
Let us record a couple of nice consequences of the unique extraction of roots property. Given a diagram group $G$ and two elements $a,b \in G$,
\begin{itemize}
	\item if $[a^m,b^n]=1$ for some $m,n \neq 0$, then $[a,b]=1$;
	\item if $a^m=b^n$ for some $m,n \neq 0$, then there exist $c \in G$ and $k, \ell \geq 0$ satisfying $km=\ell n$ such that $a=c^k$ and $b=c^\ell$.
\end{itemize}
See \cite[Corollaries~15.27 and~15.28]{MR1396957} for more details.

\begin{prop}[\cite{MR1396957}]\label{prop:Roots}
Every element of a diagram group $G$ has only finitely many roots. More precisely, if $a,b \in G\backslash \{1\}$ satisfies $a=b^k$ for some $k \neq 0$, then $|k|$ does not exceed half the number of cells in a reduced diagram representing $a$.
\end{prop}

\noindent
We refer to Section~\ref{section:WordConj} for a proof of the proposition. 

\medskip \noindent
We conclude this section with some information on homology groups of diagram groups. The main result in this direction is:

\begin{thm}[\cite{MR2193190}]\label{thm:Homology}
Homology groups over $\mathbb{Z}$ of diagram groups are free abelian.
\end{thm}

\noindent
Some elements from the proof of the theorem are given in Section~\ref{section:Presentation}. Interestingly, with some restrictions on the underlying semigroup presentations, further information can be extracted from the proof of theorem:

\begin{thm}[\cite{MR2193190}]
Let $\mathcal{P}$ be a finite and complete presentation of a finite semigroup. A diagram group over $\mathcal{P}$ has a rational Poincar\'e series.
\end{thm}

\noindent
Recall that the \emph{Poincar\'e series} of a group $G$ is the series $P_G(t):= \sum_{n \geq 0} r_nt^n$ where $r_n$ denotes the rank of $H_n(G,\mathbb{Z})$ for every $n \geq 0$. Complete semigroup presentation will be defined in Section~\ref{section:Presentation}. 

\medskip \noindent
As a particular case of Theorem~\ref{thm:Homology}, abelianisations of diagram groups are free abelian (also proved in \cite{MR1396957}, see Corollary~\ref{cor:Abelianisation} below). This contrasts with Pride's result from \cite{MR2727671}: for every abelian group $A$ which is a finite or countably infinite direct product of cyclic groups, there exists a diagram grop $G$ such that $[G,G] / [[G,G],G] \simeq A$.

\section{Algorithmic and computational properties}\label{section:Algo}

\subsection{Computation of presentations and homology groups}\label{section:Presentation}

\noindent
In this section, we are concerned with the following problem: Given a semigroup presentation $\mathcal{P}:= \langle \Sigma \mid \mathcal{R} \rangle$ and a baseword $w \in \Sigma^+$, how can we compute efficiently a presentation for the diagram group $D(\mathcal{P},w)$? Or equivalently, how can we compute efficiently a presentation for the fundamental group of the Squier square complex $S(\mathcal{P},w)$? We describe here the solution proposed in \cite[Chapter~9]{MR1396957}. See also Section~\ref{section:Hyperplanes} for a decomposition of diagram groups as fundamental groups of graphs of groups. 

\medskip \noindent
The starting idea is simple. We want to fix, in our Squier square complex $S(\mathcal{P},w)$, an orientation of the edges and a spanning tree $T(\mathcal{P},w)$ so that
$$\left\langle \text{oriented edges in $S(\mathcal{P},w)$} \left| \begin{array}{c} \text{edges in $T(\mathcal{P},w)$ are trivial} \\ \text{the boundary of a square is trivial} \end{array} \right. \right\rangle$$
defines a presentation of our diagram group $D(\mathcal{P},w)$. 
As an illustration, consider the semigroup presentation 
$$\mathcal{P}:= \langle a,b,c, x,y \mid a=b, x=y, c=ax, c=bx \rangle.$$

\medskip \noindent
\begin{minipage}{0.4\linewidth}
\includegraphics[width=0.95\linewidth]{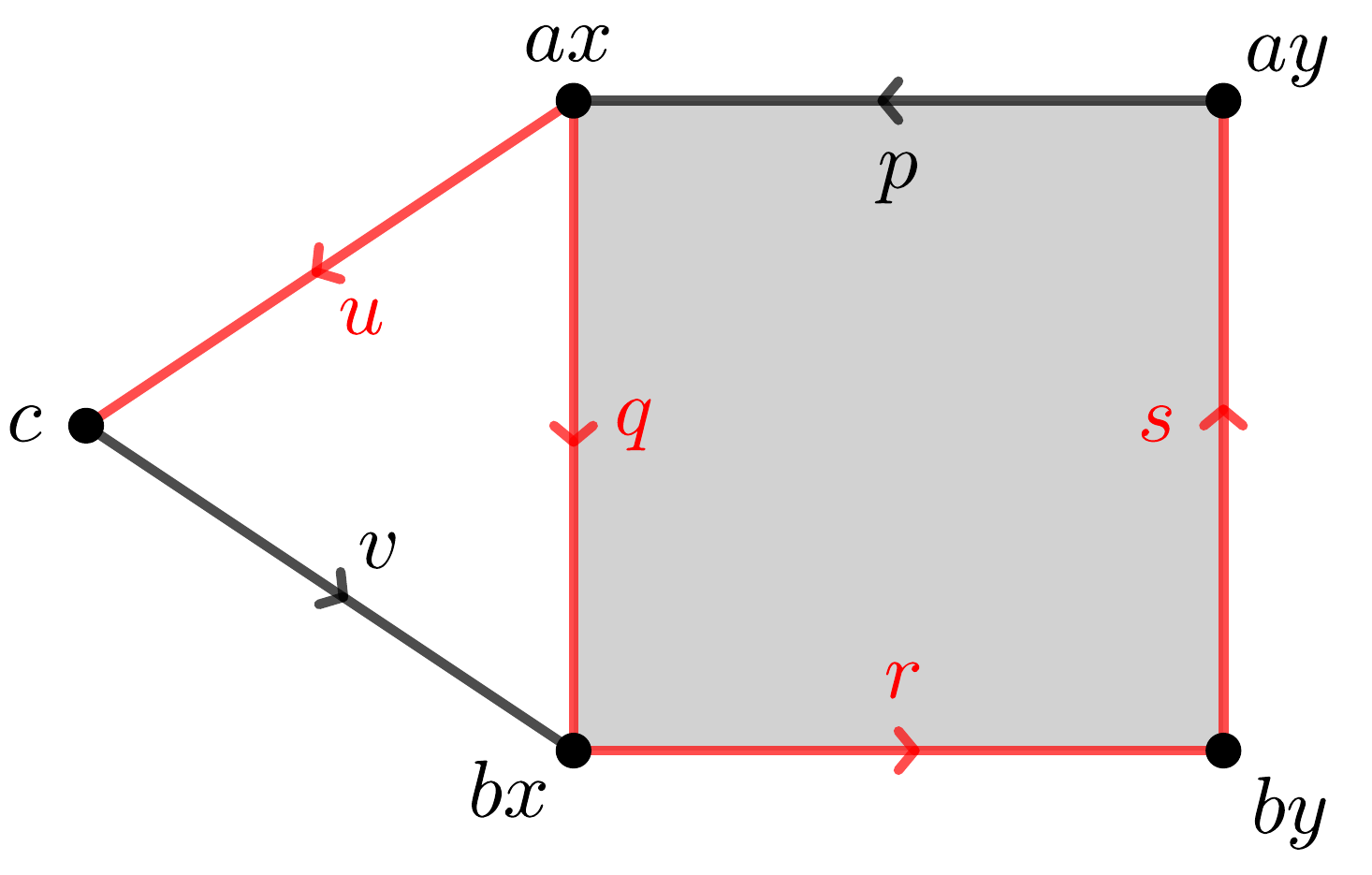}
\end{minipage}
\begin{minipage}{0.58\linewidth}
Draw the Squier complex $S(\mathcal{P},c)$, orient the edges, and choose a spanning tree as illustrated on the left. One gets the presentation
$$\langle p,q,r,s,u,v \mid u=q=r=s=1, qrsp=1 \rangle$$
of the diagram group $D(\mathcal{P},c)$, which can be simplified as $\langle v \mid \ \rangle$. 
\end{minipage}

\medskip \noindent
The difficulty is to find an automatic procedure constructing a spanning tree. This is why we restrict ourselves to \emph{complete} semigroup presentations, which we define now.

\medskip \noindent
Let $\mathcal{P}= \langle \Sigma \mid \mathcal{R} \rangle$ be a semigroup presentation. Given two positive words $w_1,w_2 \in \Sigma^+$, we write $w_1 \to w_2$ if $w_2$ can be obtained from $w_1$ by applying a(n ordered) relation of $\mathcal{R}$, i.e.\ if we can write $w_1=aub$ and $w_2=avb$ for some $u=v$ in $\mathcal{R}$. We emphasize that, because $v=u$ does not belong to $\mathcal{R}$, $w_1 \to w_2$ does not necessarily imply $w_2 \to w_1$. We denote by $\overset{\star}{\to}$ the transitive closure of $\to$, i.e.\ given two words $u,v \in \Sigma^+$, we write $u \overset{\star}{\to} v$ if there exist $w_1, \ldots, w_n \in \Sigma^+$ satisfying $u \to w_1 \to \cdots \to w_n \to v$. The presentation $\mathcal{P}$ is \emph{complete} if the relation $\overset{\star}{\to}$ is \emph{confluent} (i.e.\ if $u \overset{\star}{\to} a$ and $u \overset{\star}{\to} b$, then there exists $v$ such that $a \overset{\star}{\to} v$ and $a \overset{\star}{\to} v$) and \emph{terminating} (i.e.\ every sequence $w_1 \to w_2 \to \cdots$ eventually terminates). For instance, the presentation $\langle x \mid x^2=x \rangle$ is complete, but not the presentation $\langle x \mid x=x^2 \rangle$ because $x \to x^2 \to x^3 \to \cdots$ does not terminate.

\medskip \noindent
From now on, we assume that our semigroup presentation $\mathcal{P}= \langle \Sigma \mid \mathcal{R} \rangle$ is complete. As a consequence, every positive word $w \in \Sigma^+$ has a unique \emph{reduced form}, i.e.\ a positive word $\bar{w} \in \Sigma^+$ such that $w \overset{\star}{\to} \bar{w}$ and such that there is no $u \in \Sigma^+$ such that $\bar{w} \to u$. (Because $\overset{\star}{\to}$ is terminating, such a word exists; and because $\overset{\star}{\to}$ is confluent, such a word is unique.) However, a derivation from a word to its reduction may not be uniquely defined. What we want to do is to define a canonical way to reduced words and to define our spanning tree in the Squier square complex as the union of all the paths given by these reductions. These canonical reductions are defined by the elementary moves associated to the following edges in our Squier complex:

\begin{definition}\label{def:PrincipalLeft}
An oriented edge $(u, \ell \to r,v)$ in $S(\mathcal{P})$ is a \emph{principal left edge} if the following conditions hold:
\begin{itemize}
	\item each proper prefix of the word $u \ell$ is reduced;
	\item $\ell$ is the longest suffix of $u \ell$ which is the left side of a relation in $\mathcal{R}$;
	\item if there is a relation $\ell = r'$ in $\mathcal{R}$ with $r \neq r'$, then $r$ is smaller than $r'$ in the ShortLex order.
\end{itemize}
\end{definition}

\noindent
Here, we think of the letters in $\Sigma$ as totally ordered (say following the left-right order when we write the presentation $\mathcal{P}$) so that positive words are naturally totally ordered by the corresponding ShortLex order, namely: given $w_1,w_2 \in \Sigma^+$, we write $w_1 \leq w_2$ if $w_1$ is shorter than $w_2$ or if $w_1,w_2$ have the same length and $w_1$ is smaller than $w_2$ with respect to the lexicographic order given by our total order on $\Sigma$. 

\medskip \noindent
Let us illustrate with an example how to reduce a word by following the pattern given by the previous definition (which seems to be difficult to digest, but which turns out to quite natural when we understand how it works). So let us consider the semigroup presentation
$$\mathcal{P}:= \langle a,b \mid a^3=a, a^3=a^2, ba^3=a^3b \rangle,$$
which is complete. The word $aba^3ba^4$ is not reduced. 
\begin{itemize}
	\item Identify the smallest prefix of $aba^3ba^4$ that is not reduced. This is $aba^3$. 
	\item Identify the longest suffix of $aba^3$ that can be modified by applying a relation. This is $ba^3$.
	\item Observe that $ba^3=a^3b$ is the only relation having $ba^3$ as its left side, so this the relation we have to apply.
\end{itemize}
Thus, the first step in our derivation from $aba^3ba^4$ to its reduction is $aba^3ba^4 \to a^4b^2a^4$. Our new word is again not reduced, so another step is necessary.
\begin{itemize}
	\item Identify the smallest prefix of $a^4b^2a^4$ that is not reduced. This is $a^3$.
	\item Identify the longest suffix of $a^3$ that can be modified by applying a relation. This is $a^3$.
	\item We have two relation having $a^3$ as their left sides, namely $a^3=a$ and $a^3=a^2$. Compare the right sides $a$ and $a^2$ with respect to the ShortLex order. Since $a$ is shorter than $a^2$, the relation we must apply is $a^2=a$.
\end{itemize}
Thus, the second step in our derivation is $a^4b^2a^4 \to a^2 b^2a^4$. Keeping applying the procedure, we eventually find 
$$aba^3ba^4 \to a^4b^2a^4 \to a^2b^2a^4 \to a^2ba^3ba \to a^5b^2a \to a^3b^2a \to ab^2a,$$
which is the canonical reduction we are looking for.

\medskip \noindent
The union of all the principal left edges given by Definition~\ref{def:PrincipalLeft} define a spanning forest in our Squier complex, as previous claimed \cite[Lemma~9.4]{MR1396957}, which allows us to deduce the following statement:

\begin{thm}[\cite{MR1396957}]\label{thm:Presentation}
Let $\mathcal{P}= \langle \Sigma \mid \mathcal{R} \rangle$ be a complete semigroup presentation and $w \in \Sigma^+$ a baseword. The diagram group $D(\mathcal{P},w)$ admits the following presentation. The generators are the triples $(u, \ell \to r, v)$ where $(\ell = r) \in \mathcal{R}$, $urv=w$ mod $\mathcal{P}$, $u$ is a reduced word, $v \in \Sigma^+$, $(u, \ell \to r,v)$ is not a principal left edge. The defining relations are all the relations of the form
$$(u, \ell \to r, vsw) = (u, \ell \to r, vtw)^{( \overline{urv}, s \to t,w)}$$
where $(s=t) \in \mathcal{R}$ and the edge $( \overline{urv}, s \to t,w)$ is not a principal left edge; or of the form
$$(u, \ell \to r, vsw) = (u, \ell \to r, vtw)$$
if $( \overline{urv}, s \to t,w)$ is a principal left edge. 
\end{thm}

\noindent
Here, we use the notation $a^b=b^{-1}ab$. It is worth mentioning that the presentation given by Theorem~\ref{thm:Presentation} is in general not optimal since the second type of relations allows us to decrease the number of generators. It is possible to extract a presentation with a minimal number of generators (which turns out to coincide with the rank of the abelianisation), but at the cost of a more technical statement. We refer the interested reader to \cite[Theorem~9.8]{MR1396957}.

\medskip \noindent
An interesting consequence of Theorem~\ref{thm:Presentation} is that diagram groups over complete semigroup presentations are examples of \emph{LOG groups}.

\begin{definition}
Let $\Gamma$ be a directed graph and $\lambda : E(\Gamma) \to V(\Gamma)$ a labelling map. The \emph{LOG group} $\Gamma(\lambda)$ is given by the presentation
$$\left\langle V(\Gamma) \mid  a=b^{\lambda(a,b)}  , \ (a,b) \in E(\Gamma) \right\rangle.$$
\end{definition}

\noindent
LOG stands for Labelled Oriented Graph. It is worth noticing that not all diagram groups are LOG. Indeed, it is clear that every LOG group surjects onto $\mathbb{Z}$ (it suffices to send all the generators to $1$), but we know from Proposition~\ref{prop:CommutatorF} that the commutator subgroup $F'$ of Thompson's group $F$ is a simple diagram group. Thus, $F'$ is a diagram group but not a LOG group. We also deduce that $F'$ cannot be described as a diagram group over a complete semigroup presentation, so we lost some generality when restricting ourselves to complete semigroup presentation. Nevertheless, as shown by Lemma~\ref{lem:Complete} below, it turns out that every diagram group can be described (more or less explicitly) as a retract inside some diagram group over a complete semigroup presentation, so it is also possible to compute presentations of diagram groups over incomplete semigroup presentations thanks to the techniques developed in this section (see \cite[Lemma~9.11]{MR1396957} and the related discussion).

\begin{ex}
Consider the semigroup presentation $\mathcal{P}=\langle x \mid x^2=x \rangle$, which is complete. Notice that the principal left edges in $S(\mathcal{P},x)$ are the edges of the form $(\emptyset, x^2 \to x, x^n)$, $n \geq 0$. A direct application of Theorem~\ref{thm:Presentation} shows that $D(\mathcal{P},x)$, which is isomorphic to the Thompson group $F$ according to Proposition~\ref{prop:Thompson}, admits a presentation whose generators are the triples $(x,x^2 \to x,x^n)$ for $n \geq 0$, and whose relations are
$$(x,x^2 \to x, x^px^2x^q)= (x,x^2 \to x, x^pxx^q)^{(x,x^2\to x,x^q)}, \ p,q \geq 0.$$
Setting $s_i:= (x,x^2 \to x, x^i)$ for every $i \geq 0$, one gets the infinite presentation
$$\left\langle s_0,s_1, \ldots \mid s_{n+1}= s_n^{s_m}, \ n>m \geq 0 \right\rangle.$$
The presentation can be simplified as a finite presentation; see \cite[Example~9.10]{MR1396957} for details. 
\end{ex}

\noindent
Let us mention a theoretical application of Theorem~\ref{thm:Presentation}, namely:

\begin{cor}\label{cor:Abelianisation}
Diagram groups have free abelianisations.
\end{cor}

\begin{lemma}\label{lem:Complete}
Every diagram group is a retract in a diagram group over a complete semigroup presentation.
\end{lemma}

\noindent
Recall that, given a group $G$, a subgroup $H \leq G$ is a \emph{retract} if there exists a morphism $\rho : G \to H$ that restricts to the identity on $H$.

\begin{proof}[Proof of Lemma~\ref{lem:Complete}.]
Let $\mathcal{P} = \langle \Sigma \mid \mathcal{R} \rangle$ be a semigroup presentation. We fix a total order on $\Sigma$ with no infinite decreasing sequence and order the positive words in $\Sigma^+$ by using the corresponding ShortLex order. Up to replacing some $u=v$ in $\mathcal{R}$ with $v=u$, we assume that, for every $u=v$ in $\mathcal{R}$, $v$ is smaller than $u$ with respect to the ShortLex order. Let $\mathcal{R}'$ denote the collection of all the relations $u=v$ such that $v$ is smaller than $u$ with respect to the ShortLex order and such that $u=v$ mod $\mathcal{P}$. Notice that $\mathcal{R} \subset \mathcal{R}'$. 

\medskip \noindent
Observe that $\mathcal{P}':= \langle \Sigma \mid \mathcal{R}' \rangle$ is complete. Indeed, the fact that our initial order on $\Sigma$ has no infinite decreasing sequence implies that the corresponding ShortLex order has no infinite decreasing sequence either. Moreover, for every $w \in \Sigma^+$, there exists a unique ShortLex-minimal word $\bar{w} \in \Sigma^+$ among all the positive words equal to $w$ mod $\mathcal{P}$; and, for all $u,v \in \Sigma^+$ satisfying $w \overset{\star}{\to} u$ and $w \overset{\star}{\to} v$, we have $u \to \bar{w}$ and $v \to \bar{w}$ by definition of $\mathcal{R}'$. Thus, $\mathcal{P}'$ is terminating and confluent.

\medskip \noindent
Fix a baseword $w \in \Sigma^+$. Because a (reduced) diagram over $\mathcal{P}$ is also a (reduced) diagram over $\mathcal{P}'$, there is an obvious injective morphism $D(\mathcal{P},w) \to D(\mathcal{P}',w)$. From now on, we identify $D(\mathcal{P},w)$ with its image in $D(\mathcal{P}',w)$. For every $u=v$ in $\mathcal{R}'$ but not in $\mathcal{R}$, we know that $u=v$ mod $\mathcal{P}$, so there must exist a diagram $\Delta(u,v)$ with its top path labelled $u$ and its bottom path labelled $v$. Define a map $D(\mathcal{P}',w) \to D(\mathcal{P},w)$ by sending a diagram over $\mathcal{P}'$ to the diagram over $\mathcal{P}$ obtained by replacing each to cell labelled by $u=v$ with $(u=v) \in \mathcal{R}$ (resp. with $(v=u) \in \mathcal{R}$) with a copy of $\Delta(u,v)$ (resp. $\Delta(u,v)^{-1}$). We can verify that this map is a well-defined morphism; see \cite[Theorem~7.7]{MR1396957} for more details. Moreover, this morphism is clearly the identity on $D(\mathcal{P},w)$. 
\end{proof}

\begin{proof}[Proof of Corollary~\ref{cor:Abelianisation}.]
We know that diagram groups over complete semigroup presentations are LOG groups, and clearly LOG groups have free abelianisations. Therefore, it follows from Lemma~\ref{lem:Complete} that the following observation is sufficient to conclude the proof of our corollary:

\begin{fact}
Let $G$ be a group and $H \leq G$ a retract. The abelianisation of $H$ embeds into the abelianisation of $G$.
\end{fact}

\noindent
Let $\alpha : G \to G^{\mathrm{ab}}$ denote the abelianisation map of $G$ and let $\rho : G \to H$ be a retraction. The restriction of $\alpha$ to $H$ yields a morphism $H \to G^{\mathrm{ab}}$, which factors through $\varphi : H^{\mathrm{ab}} \to G^{\mathrm{ab}}$. Let $g \in H^{\mathrm{ab}}$ be an element in the kernel of $\varphi$. Fix a pre-image $\bar{g}$ of $g$ in $H$. By definition of $\varphi$, $\alpha(\bar{g})$ is trivial. In other words, $\bar{g}$ belongs to $[G,G]$. Hence
$$\bar{g} = \rho(\bar{g}) \in [\rho(G), \rho(G)] = [H,H],$$
which implies that $g$ is trivial in $H^{\mathrm{ab}}$. Thus, we have proved that $\varphi : H^{\mathrm{ab}} \to G^{\mathrm{ab}}$ is injective. 
\end{proof}

\noindent
As already mentioned in Section~\ref{section:Properties}, Corollary~\ref{cor:Abelianisation} can be generalised by showing that homology groups of diagram groups are free abelian. The strategy in similar in spirit to the proof of Corollary~\ref{cor:Abelianisation}. We refer to \cite{MR2193190} for more details, and only sketch the argument below (in the language of diagrams instead of directed $2$-complexes). 

\medskip \noindent
Given a semigroup presentation $\mathcal{P}= \langle \Sigma \mid \mathcal{R} \rangle$ and a baseword $w \in \Sigma^+$, we need a classifying space for the diagram group $D(\mathcal{P},w)$. We already know that $D(\mathcal{P},w)$ is the fundamental group of the Squier square complex $S(\mathcal{P},w)$, but this complex is not contractible in general. However, by adding natural higher-dimensional cubes to it, it turns out to be contractible.

\begin{definition}
The \emph{Squier cube complex} $S^+(\mathcal{P})$ is the cube complex whose vertices are the words in $\Sigma^+$; whose (oriented) edges can be written as $(a,u \to v,b)$, where $u=v$ or $v=u$ belongs to $\mathcal{R}$, connecting the vertices $aub$ and $avb$; and whose $n$-cubes similarly can be written as $(a_1,u_1 \to v_1, \ldots, a_n, u_n \to v_n, a_{n+1})$, spanned by the set of vertices $\{a_1w_1 \cdots a_nw_na_{n+1} \mid w_i=v_i \text{ or } u_i\}$. Given a word $w \in \Sigma^+$, we denote by $S^+(\mathcal{P},w)$ the connected component of $S^+(\mathcal{P})$ containing the vertex $w$.  
\end{definition}

\noindent
Observe that $S(\mathcal{P})$ coincides with the two-skeleton of $S^+(\mathcal{P})$, so our diagram group $D(\mathcal{P},w)$ remains isomorphic to the fundamental group of $S^+(\mathcal{P},w)$. Each connected component of the Squier cube complex is contractible. This can be proved directly \cite[Theorem~3.10]{MR1978047} or can be deduced from its nonpositive curvature described in Section~\ref{section:MedianDiag}. Anyway, the key point is that $S^+(\mathcal{P},w)$ yields a classifying space for the corresponding diagram group $D(\mathcal{P},w)$. 

\medskip \noindent
Next, observe that, given a group $G$, if $H \leq G$ is a retract then the homology groups of $H$ are retracts of the homology groups of $G$, so it follows from Lemma~\ref{lem:Complete} that we can focus on diagram groups over complete semigroup presentations. So let $\mathcal{P}=\langle \Sigma \mid \mathcal{R} \rangle$ be a complete semigroup presentation and fix a baseword $w \in \Sigma^+$. Similarly to what we did for $S(\mathcal{P},w)$ above, we can use the completeness of $\mathcal{P}$ in order to extract from $S^+(\mathcal{P},w)$ the essential part that encodes the homology. 
\begin{figure}[h!]
\begin{center}
\includegraphics[width=0.7\linewidth]{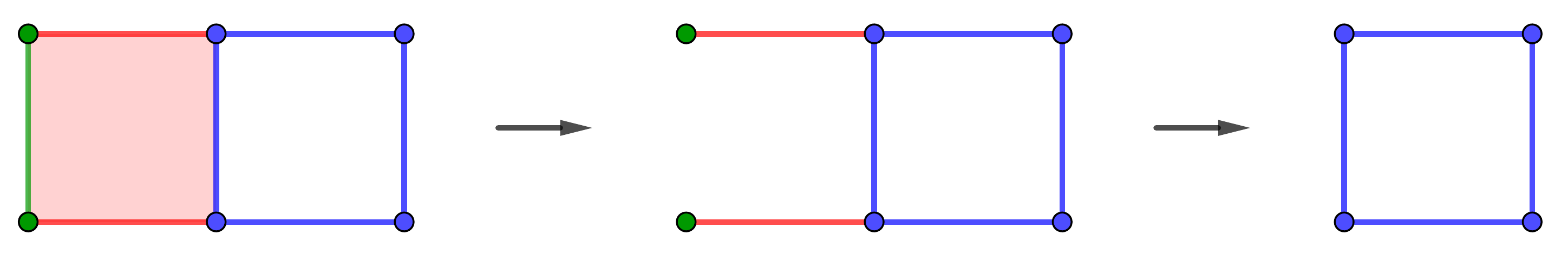}
\caption{A collapsible scheme. Collapsible cells in red, redundant cells in green, essential cells in blue.}
\label{Scheme}
\end{center}
\end{figure}

\noindent
Formally, \cite{MR2193190} defines a \emph{collapsible scheme} on $S^+(\mathcal{P},w)$; see \cite[Section~9]{MR2193190} for a precise statement. Roughly speaking, this is a partition of the cubes of $S^+(\mathcal{P},w)$ into three categories: \emph{contractible}, \emph{redundant}, and \emph{essential} cubes. Under reasonable assumptions, it is possible to successively collapse contractible cubes by pushing inside a codimension-one face that is redundant. At the end of the process, one gets a homotopy equivalent cube complex $E(\mathcal{P},w)$ containing only the initial essential cubes. 

\medskip \noindent
Thanks a well-chose collapsible scheme, it is proved in \cite{MR2193190} that $E(\mathcal{P},w)$ is as small as possible. More precisely, in each dimension $n$, the $n$th homology group of $E(\mathcal{P},w)$ (which is also the $n$th homology group of $D(\mathcal{P},w)$ since $E(\mathcal{P},w)$ and $S^+(\mathcal{P},w)$ are homotopy equivalent) is freely generated by the $n$-cubes.

\subsection{Word and conjugacy problems}\label{section:WordConj}

\noindent
As already mentioned, every diagram over a semigroup presentation admits a unique reduction (i.e.\ a diagram with no dipole). Therefore, in order to determine whether a given diagram represents the neutral element in the corresponding diagram group, it suffices to reduce its dipoles, in whatever order, and to look at if $2$-cells survive at the end of the process. If so, the diagram represents a non-trivial element; otherwise, it represents the neutral element. Moreover, this word problem can be solved very efficiently.

\begin{thm}[\cite{MR1396957}]
Let $\mathcal{P}= \langle \Sigma \mid \mathcal{R} \rangle$ be a semigroup presentation and $w \in \Sigma^+$ a baseword. Fix finitely many $(w,w)$-diagrams $\Delta_1, \ldots, \Delta_r$ over $\mathcal{P}$. There exists an algorithm that decides, given a product of $n$ diagrams $\Delta_i$, whether it represents the trivial element in $D(\mathcal{P},w)$ with time complexity $O(n^2 \log^2n)$. 
\end{thm}

\noindent
In order to solve the conjugacy problem, we need to introduce some terminology. 

\begin{definition}
Let $\mathcal{P}$ be a semigroup presentation and $\Delta$ a diagram. If $\Delta$ decomposes as the concatenation $\Delta_1 \circ \Delta_2$, then the diagram $\Delta_2 \circ \Delta_1$, when well-defined, is a \emph{cyclic shift} of $\Delta$. 
\end{definition}

\noindent
We emphasize that $\Delta_1 \circ \Delta_2$ is a decomposition of $\Delta$ on the nose, not up to dipole reduction. As an illustration, given the two diagrams $\Delta,\Omega$ below, $\Omega$ is a cyclic shift of~$\Delta$.
\begin{figure}[h!]
\begin{center}
\includegraphics[width=0.7\linewidth]{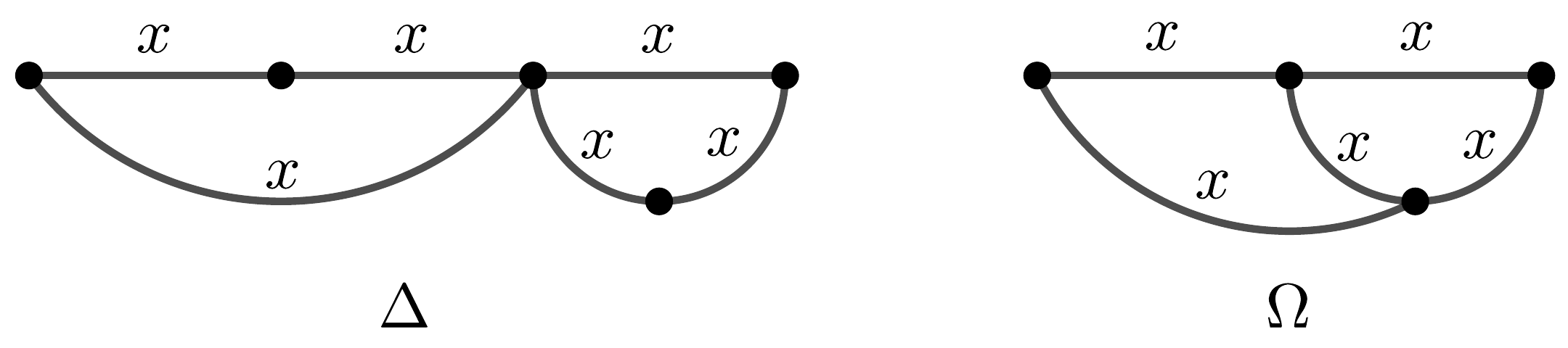}
\caption{A diagram $\Delta$ over $\mathcal{P}:= \langle x \mid x=x^2 \rangle$ and a cyclic shift $\Omega$.}
\label{Shift}
\end{center}
\end{figure}

\noindent
Indeed, $\Delta$ decomposes as $(A^{-1} + \epsilon(x)) \circ (\epsilon(x)+A)$, where $\epsilon(x)$ denotes the $(x,x)$-diagram with no $2$-cell and where $A$ denotes the $(x,x^2)$-diagram with a single $2$-cell; and $\Omega = (\epsilon(x)+A) \circ (A^{-1}+ \epsilon(x))$. 

\begin{definition}\label{def:SumDiag}
Let $\mathcal{P}$ be a semigroup diagram. The \emph{sum} $\Delta_1+ \Delta_2$ of two diagrams $\Delta_1,\Delta_2$ over $\mathcal{P}$ is the diagram obtained by identifying the sink-vertex of $\Delta_1$ with the source-vertex of $\Delta_2$. A diagram is \emph{simple} if it does not decompose as the sum of two diagrams. Clearly, every diagram admits a unique decomposition as a sum of simple diagrams. A spherical diagram is \emph{normal} if all the diagrams in its decomposition as a sum of simple diagrams are spherical. 
\end{definition}

\noindent
For instance, the diagram $\Delta$ above is neither simple nor normal. However, its cyclic conjugate $\Omega$ is simple. 

\begin{thm}[\cite{MR1396957}]
Let $\mathcal{P}= \langle \Sigma \mid \mathcal{R} \rangle$ be a semigroup presentation and $w \in \Sigma^+$ a baseword. If the semigroup $\mathcal{P}$ has a solvable word problem, then the conjugacy problem in the diagram group $D(\mathcal{P},w)$ is solvable. 
\end{thm}

\noindent
It is worth noticing that a diagram group given by a semigroup presentation with unsolvable word problem may have an unsolvable conjugacy problem; see \cite[Example~15.22]{MR1396957}. Let us describe the main ingredients needed in order to solve the conjugacy problem in pratice. The first step is to reduce the problem to simple absolutely reduced diagrams.

\begin{lemma}\label{lem:NormalAbs}\emph{(\cite[Lemma~15.14]{MR1396957})}
Let $\mathcal{P}$ be a semigroup presentation and $\Delta$ a spherical diagram over $\mathcal{P}$. There exist diagrams $\Phi, \Omega$ such that $\Delta = \Phi \circ \Omega \circ \Phi^{-1}$, up to dipole reduction, and with $\Omega$ normal and absolutely reduced. Moreover, $\Omega$ can be chosen with a number of $2$-cells bounded above by the number of $2$-cells of $\Delta$. 
\end{lemma}

\noindent
We emphasize that $\Omega$ must be a spherical diagrams but its top and bottom paths may have distinct labels than the top and bottom paths of $\Delta$. Recall from Section~\ref{section:Properties} that a spherical diagram $\Delta$ is \emph{absolutely reduced} if $\Delta^n$ is reduced for every $n \geq 1$. 

\begin{lemma}\label{lem:ConjNormal}\emph{(\cite[Lemma~15.15]{MR1396957})}
Let $\mathcal{P}$ be a semigroup presentation and $\Delta,\Gamma$ two normal and absolutely reduced diagrams over $\mathcal{P}$. Decompose $\Delta$ (resp. $\Gamma$) as a sum $\Delta_1 + \cdots + \Delta_m$ (resp. $\Gamma_1+ \cdots + \Gamma_n$) of simple diagrams. For every $1 \leq i \leq m$ (resp. $1 \leq j \leq n$), let $x_i$ (resp. $y_j$) be the words labelling the top and bottom paths of $\Delta_i$ (resp. $\Gamma_j$). Assume that there exists some $\Theta$ such that $\Delta = \Theta \circ \Gamma \circ \Theta^{-1}$, up to dipole reduction. Then $m=n$ and $\Gamma$ decomposes as a sum $\Gamma_1 + \cdots + \Gamma_m$ where each $\Gamma_i$ is a $(y_i,x_i)$-diagram satisfying $\Delta_i = \Theta_i \circ \Gamma_i \circ \Theta_i^{-1}$, up to dipole reduction. 
\end{lemma}

\noindent
Next, we can solve the conjugacy problem for simple absolutely reduced diagrams thanks the next lemma. In its statement, we denote by $\mathrm{Shift}(\Delta)$ the set of all the diagrams that can be obtained from $\Delta$ by a sequence of cyclic shifts. (It is worth mentioning that a cyclic shift of a cyclic shift may not be a cyclic shift of the initial diagram.) 

\begin{lemma}\label{lem:ShiftConj}
Let $\mathcal{P}$ be a semigroup presentation. Two simple and absolutely reduced spherical diagrams $\Delta_1, \Delta_2$ are conjugate if and only if $\mathrm{Shift}(\Delta_1)= \mathrm{Shft}(\Delta_2)$. 
\end{lemma}

\noindent
Putting everything together, the conjugacy problem is solved as follows. Let $\mathcal{P}= \langle \Sigma \mid \mathcal{R} \rangle$ be a semigroup presentation and $w \in \Sigma^+$ a baseword. We fix two $(w,w)$-diagrams $\Delta_1,\Delta_2$ over $\mathcal{P}$ and we would like to determine whether or not $\Delta_1$ and $\Delta_2$ are conjugate in $D(\mathcal{P},w)$. 
\begin{enumerate}
	\item For $i=1,2$, write $\Delta_i$ as $\Phi_i \circ \Omega_i \circ \Phi_i^{-1}$ where $\Omega_i$ is normal and absolutely reduced. (Which is possible according to Lemma~\ref{lem:NormalAbs}, and can be done algorithmically.)
	\item Decompose $\Omega_i$ as a sum $\Omega_i^1 + \cdots + \Omega_i^{n_i}$ of simple diagrams.
	\item If $n_1 \neq n_2$, it follows from Lemma~\ref{lem:ConjNormal} that $\Delta_1$ and $\Delta_2$ are not conjugate. 
	\item Otherwise, check whether $\mathrm{Shift}(\Omega_1^j)= \mathrm{Shift}(\Omega_2^j)$ for every $1 \leq j \leq n_1=n_2$. (This can be done algorithmically according to \cite[Lemma~15.21]{MR1396957}.)
	\item According to Lemma~\ref{lem:ShiftConj}, $\Delta_1$ and $\Delta_2$ are conjugate if and only if all the equalities hold. 
\end{enumerate}
As a simple illustration, we can show that the two smallest reduced $(x,x)$-diagrams over $\mathcal{P}=\langle x \mid x=x^2 \rangle$ are not conjugate in $D(\mathcal{P},x)$. See Figure~\ref{Conjugacy}. 

\begin{figure}[h!]
\begin{center}
\includegraphics[width=\linewidth]{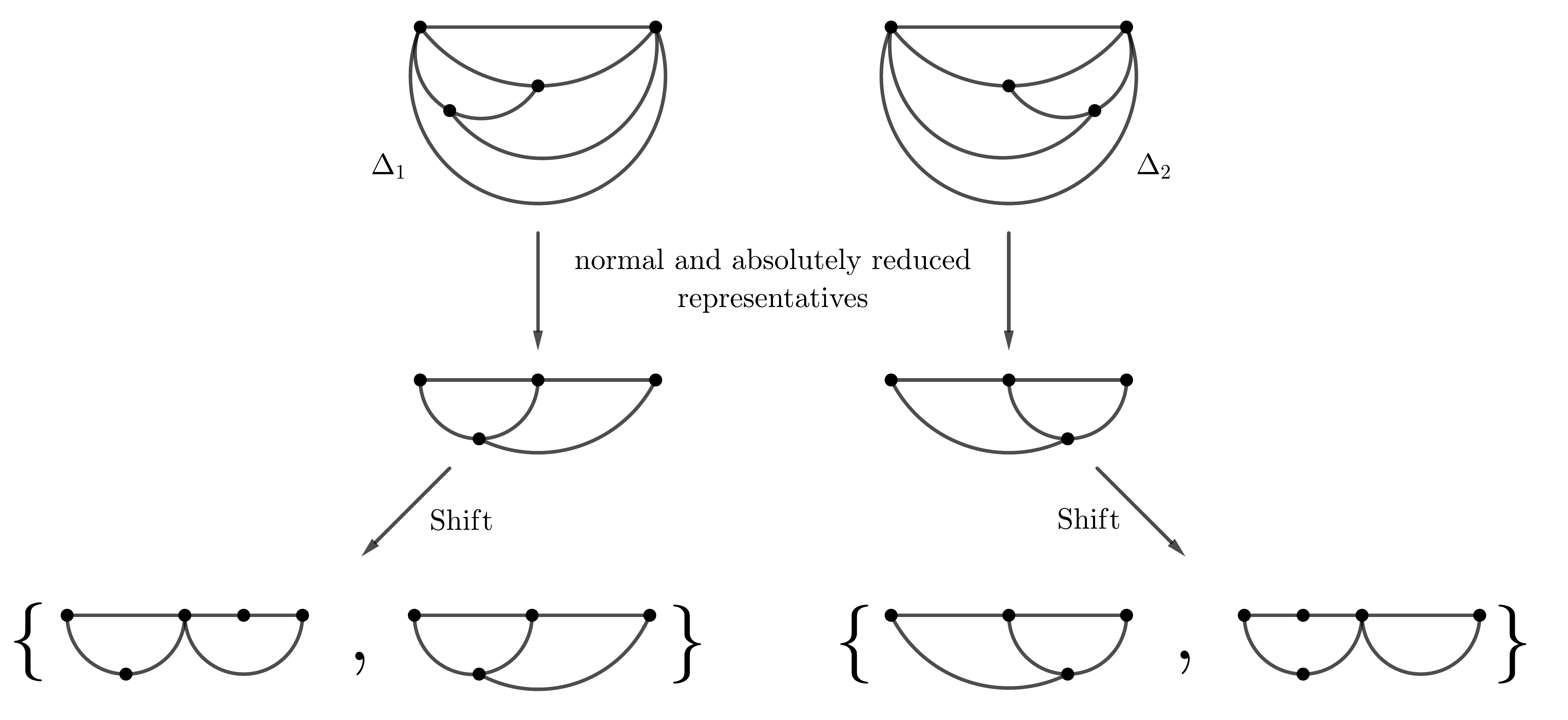}
\caption{Two $(x,x)$-diagrams over $\langle x, \mid x=x^2 \rangle$ that are not conjugate.}
\label{Conjugacy}
\end{center}
\end{figure}

\medskip \noindent
It is worth mentioning that the tools developped in order to solve the conjugacy problem turn out to be also useful to help proving some of the statements mentioned in Section~\ref{section:Properties}. For instance:

\begin{proof}[Proof of Proposition~\ref{prop:Roots}.]
Let $\Delta,\Gamma$ be two spherical reduced diagrams satisfying $\Delta= \Gamma^k$ for some $k \geq 1$. We assume that $\Delta$ and $\Gamma$ contain at least one $2$-cell. According to Lemma~\ref{lem:NormalAbs}, we can write $\Delta$ as $\Phi \Delta_0 \Phi^{-1}$ and $\Gamma$ as $\Psi \Gamma_0 \Psi^{-1}$ for some normal absolutely reduced diagrams $\Delta_0$ and $\Gamma_0$. Moreover, we can choose $\Delta$ with a number of $2$-cells bounded above by the number of $2$-cells of $\Delta$. Let
$$\Delta_0 = \Delta_0^1 + \cdots + \Delta_0^m \text{ and } \Gamma_0 = \Gamma_0^1+ \cdots + \Gamma_0^n$$
be the decompositions of $\Delta_0$ and $\Gamma_0$ as sums of simple diagrams. We have
$$\Delta_0^1+ \cdots + \Delta_0^m = \left( \Phi^{-1} \Psi \right) \circ \left( \Gamma_0^1+ \cdots + \Gamma_0^n \right)^k \circ \left( \Phi^{-1} \Psi \right)^{-1}.$$
As a consequence of Lemma~\ref{lem:ConjNormal}, $n=m$ and $\Phi^{-1} \Phi$ decomposes as a sum $\Omega_1+ \cdots + \Omega_n$ such that $\Delta_0^i = \Omega_i \circ (\Gamma_0^i)^k \circ \Omega_i^{-1}$ for every $1 \leq i \leq n$. Fix an index $1 \leq i \leq n$ such that $\Delta_0^i$ contains at least one $2$-cell. We deduce from Lemma~\ref{lem:ShiftConj} that $\mathrm{Shift}(\Delta_0^i)= \mathrm{Shift} (( \Gamma_0^i)^k)$, which implies that $\Delta_0^i$ and $(\Gamma_0^i)^k$ have the same number of $2$-cells. Because the number of $2$-cells of $(\Gamma_0^i)^k$ coincides with $k$ times the number of $2$-cells of $\Gamma_0^i$, it follows that $k$ is at most half the number of $2$-cells of $\Delta_0^i$, and a fortiori of $\Delta_0$ and finally of $\Delta$. 
\end{proof}

\subsection{Commutation problem}\label{section:Commutation}

\noindent
Centralisers in diagram groups are rather well understood. More precisely:

\begin{thm}[\cite{MR1396957}]\label{thm:Centralisers}
Let $\mathcal{P}= \langle \Sigma \mid \mathcal{R} \rangle$ be a semigroup presentation, $w \in \Sigma^+$ a baseword, and $g \in D(\mathcal{P},w)$ an element. One can represent $g$ as
$$\Gamma \cdot \left( \Delta_1^{n_1} + \cdots + \Delta_r^{n_r} \right) \cdot \Gamma^{-1},$$
where $\Gamma$ is some $(w,u_1 \cdots u_r)$-diagram and where each $\Delta_i$ is a simple absolutely reduced $(u_i,u_i)$-diagram with $\langle \Delta_i \rangle$ either trivial or a maximal cyclic subgroup in $D(\mathcal{P},u_i)$. Then, the centraliser of $g$ coincides with
$$\left\{ \Gamma \cdot \left( \Omega_1 + \cdots + \Omega_r \right) \cdot \Gamma^{-1} \mid \Omega_1 \in C_1, \ldots, \Omega_r \in C_r \right\}$$
where $C_i := \langle \Delta_i \rangle$ if $\Delta_i$ is non-trivial and $C_i:= D(\mathcal{P},u_i)$ otherwise. 
\end{thm}

\noindent
As a consequence, the centraliser of an element in a diagram group is always isomorphic to the the product of a free abelian group of finite rank with finitely many diagram groups given by the same semigroup presentation but possibly different basewords. 

\begin{ex}
Let $\mathcal{P}$ be the semigroup presentation $\langle x \mid x=x^2 \rangle$ and $g \in D(\mathcal{P},x)$ the element represented by the following diagram.

\medskip \noindent
\includegraphics[width=\linewidth]{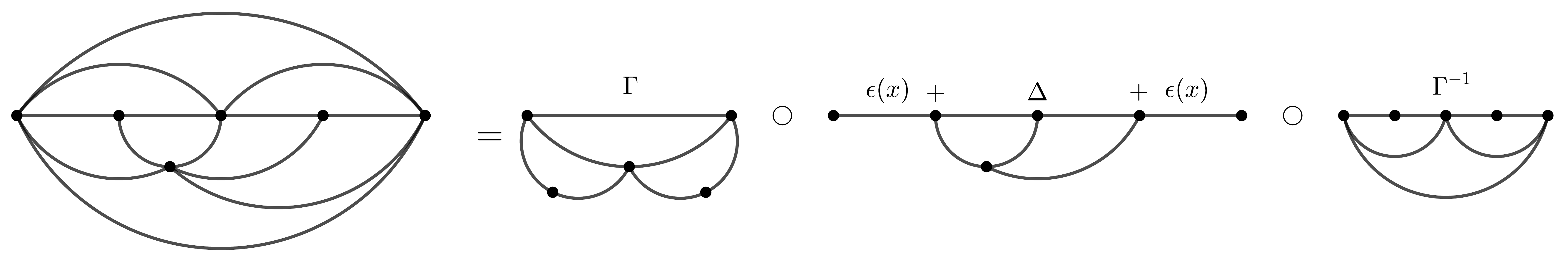}

\medskip \noindent
By applying Theorem~\ref{thm:Centralisers}, we deduce that the centraliser of $g$ in $D(\mathcal{P},x)$ is
$$\left\{ \Gamma \cdot \left( \Phi + \Delta^n + \Psi \right) \cdot \Gamma^{-1} \mid \Phi,\Psi \in D(\mathcal{P},x), n \in \mathbb{Z} \right\}.$$
In particular, the centraliser is isomorphic to $\mathbb{Z} \times F \times F$. Notice that, in $D(\mathcal{P},x)$ and a fortiori in $F$, all the centralisers are isomorphic to $\mathbb{Z}^m \times F^n$ for some $m,n \geq 0$. 
\end{ex}

\noindent
Theorem~\ref{thm:Centralisers} has a lot of interesting applications. The rest of the section is essentially dedicated to them.

\begin{cor}[\cite{MR1725439}]\label{cor:FreeAbelian}
In diagram groups, abelian subgroups are free.
\end{cor}

\begin{proof}
Let $H$ be an abelian subgroup in a diagram group $G$. Our goal is to show that $H$ is \emph{residually infinite cyclic}, i.e.\ for every non-trivial $h \in H$ there exists a morphism $\varphi : H \to \mathbb{Z}$ satisfying $\varphi(h) \neq 0$. This implies that $H$ embeds into a direct product of (infinitely many) cyclic groups. Such a product being free abelian, the desired conclusion follows.

\medskip \noindent
So let $h \in H$ be a non-trivial element. Decompose $h$ as 
$$\Gamma \cdot \left( \Delta_1^{n_1} + \cdots + \Delta_r^{n_r} \right) \cdot \Gamma^{-1},$$
where $\Gamma$ is some $(w,u_1 \cdots u_r)$-diagram and where each $\Delta_i$ is a simple absolutely reduced $(u_i,u_i)$-diagram with $\langle \Delta_i \rangle$ either trivial or a maximal cyclic subgroup in $D(\mathcal{P},u_i)$. Because $h$ is non-trivial, at least one of the $\Delta_i$ is non-trivial, say $\Delta_k$. According to Theorem~\ref{thm:Centralisers}, the centraliser $C_G(h)$ of $h$ is
$$\left\{ \Gamma \cdot \left( \Omega_1 + \cdots + \Omega_r \right) \cdot \Gamma^{-1} \mid \Omega_1 \in C_1, \ldots, \Omega_r \in C_r \right\}$$
where $C_i := \langle \Delta_i \rangle$ if $\Delta_i$ is non-trivial and $C_i:= D(\mathcal{P},u_i)$ otherwise. The restriction to $H$ of the morphism $\varphi : C_G(h) \to \langle \Delta_k \rangle$ defined by
$$\Gamma \cdot \left( \Omega_1+ \cdots + \Omega_r \right) \cdot \Gamma^{-1} \mapsto \Omega_k$$
satisfies $\varphi(h) \neq 1$. This concludes the proof of our corollary. 
\end{proof}

\noindent
We saw with Propositions~\ref{prop:RAAGinterval} and~\ref{prop:RAAGtree} that many right-angled Artin groups are diagram groups. As another application of Theorem~\ref{thm:Centralisers}, we can show that there are also many right-angled Artin groups that are not diagram groups.

\begin{cor}[\cite{MR1725439}]\label{cor:NotRAAG}
Let $\Gamma$ be a triangle-free graph. If $\Gamma$ contains an induced cycle of odd length $\geq 5$, then the right-angled Artin group $A(\Gamma)$ is not a diagram group\footnote{\cite[Theorem~30]{MR1725439} only proves that a right-angled Artin group defined by a cycle of odd length $\geq 5$ is not a diagram group, but the proof can be repeated without any modification to deduce our corollary. In fact, our statement is not optimal either. \cite[Theorem~30]{MR1725439} proves that a diagram group does not contain elements $g_0, \ldots, g_{2n}$ such that $g_i$ commutes with $g_{i+1}$ for every $i \in \mathbb{Z}_{2n}$, such that the centraliser of each $g_i$ has a cyclic centre, and such that $g_i$ and $g_{i+1}$ do no belong to the same cyclic subgroup for every $i \in \mathbb{Z}_{2n}$.}.
\end{cor}

\begin{proof}[Sketch of proof.]
Given a triangle-free graph $\Gamma$ with an induced cycle of odd length $\geq 5$, assume for contradiction that $A(\Gamma)$ is a diagram group, say $D(\mathcal{P},w)$ for some semigroup presentation $\mathcal{P}= \langle \Sigma \mid \mathcal{R} \rangle$ and some baseword $w \in \Sigma^+$. 

\medskip \noindent
Any spherical diagram over $\mathcal{P}$ can be conjugate to a sum of spherical diagrams, and the number of non-trivial components depends only on the element of $D(\mathcal{P},w)$ we are looking at. Say that an element of $D(\mathcal{P},w)$ is \emph{monocomponent} if it has only one such component. Given two monocomponent elements $g_1,g_2 \in D(\mathcal{P},w)$, we write $g_1 \prec g_2$ whenever there exist a $(w,xuyvz)$-diagram $\Gamma$ and simple absolutely reduced $(u,u)$- and $(v,v)$-diagrams $\Psi_1$ and $\Psi_2$ such that
$$g_1 = \Gamma \cdot \left( \epsilon(x) + \Psi_1 + \epsilon(yvz) \right) \cdot \Gamma^{-1} \text{ and } g_2 = \Gamma \cdot \left( \epsilon(xuy) + \Psi_2+ \epsilon(z) \right) \cdot \Gamma^{-1}.$$
It can be verified that $\prec$ is transitive and antisymmetric, and that any two commuting monocomponent elements that do not belong to the same cyclic subgroup are $\prec$-comparable. See \cite{MR1725439} for details. 

\medskip \noindent
Now, let $s_0, \ldots, s_{2n} \in D(\mathcal{P},w)$ denote the generators of $A(\Gamma)$ given by an induced cycle of odd length $\geq 5$. It follows from the description of centralisers in right-angled Artin groups \cite{ServatiusCent} that the centraliser of each $s_i$ is $\langle s_{i-1},s_i,s_{i+1} \rangle \simeq \mathbb{Z} \times \mathbb{F}_2$; in particular, its centre is infinite cyclic. We deduce from Theorem~\ref{thm:Centralisers} that the $s_i$ are monocomponent. 

\medskip \noindent
Assume without loss of generality that $s_0 \prec s_1$. We cannot have $s_1 \prec s_2$ since otherwise this would imply that $s_0 \prec s_2$, which is impossible since $s_0$ and $s_2$ do not commute. Therefore, we have $s_2 \prec s_1$. Similarly, we cannot have $s_3 \prec s_2$ since otherwise we would have $s_3 \prec s_1$ and $s_3$ would commute with $s_1$. Hence $s_2 \prec s_3$. Iterating the argument, we find that:
$$s_0 \prec s_1, s_1 \prec s_2, s_2 \prec s_3, \ldots, s_{2n-1} \prec s_{2n}, s_{2n} \prec s_0.$$
It follows that $s_{2n}$ commutes with $s_1$ since $s_{2n} \prec s_1$, a contradiction.
\end{proof}

\noindent
An interesting consequence of Theorem~\ref{thm:Centralisers} is that nilpotent subgroups in diagram groups must be abelian (and a fortiori free abelian according to Corollary~\ref{cor:FreeAbelian}). This assertion will follow from the next statement:

\begin{thm}[\cite{MR1725439}]\label{thm:ForNilpotent}
For subgroup in a diagram group, the centre and the commutator subgroup intersect trivially. 
\end{thm}

\begin{proof}
Let $\mathcal{P}= \langle \Sigma \mid \mathcal{R} \rangle$, be a semigroup presentation, $w\in \Sigma^+$ a baseword, and $H$  a subgroup of $D(\mathcal{P},w)$. Assume that the centre of $H$ contains a non-trivial element $h \in H$, otherwise there is nothing to prove. Decompose $h$ as 
$$\Gamma \cdot \left( \Delta_1^{n_1} + \cdots + \Delta_r^{n_r} \right) \cdot \Gamma^{-1},$$
where $\Gamma$ is some $(w,u_1 \cdots u_r)$-diagram and where each $\Delta_i$ is a simple absolutely reduced $(u_i,u_i)$-diagram with $\langle \Delta_i \rangle$ either trivial or a maximal cyclic subgroup in $D(\mathcal{P},u_i)$. For any two $a,b \in H$, it follows from Theorem~\ref{thm:Centralisers} that we can write
$$a = \Gamma \cdot \left( A_1 + \cdots + A_r \right) \cdot \Gamma^{-1} \text{ and } b= \Gamma \cdot \left( B_1+ \cdots + B_r \right) \cdot \Gamma^{-1}$$
where, for each $1 \leq i \leq r$, $A_i$ and $B_i$ belong to $\langle \Delta_i \rangle$ if $\Delta_i$ is non-trivial and to $D(\mathcal{P},u_i)$ otherwise. Then
$$[a,b] = \Gamma \cdot \left( [A_1,B_1] + \cdots + [A_r,B_r] \right) \cdot \Gamma^{-1}.$$
Observe that, for every $1 \leq i \leq r$ such that $\Delta_i$ is non-trivial, $[A_i,B_i]$ must be trivial. Consequently, a product of commutators of elements in $H$ decomposes as 
$$\Gamma \cdot ( C_1+ \cdots + C_r) \cdot \Gamma^{-1}$$ 
where each $C_i$ is trivial if $\Delta_i$ is non-trivial. This implies that $h$ cannot be written as a product of commutators in $H$. 
\end{proof}

\begin{cor}[\cite{MR1725439}]
Nilpotent subgroups in diagram groups are abelian.
\end{cor}

\begin{proof}
For every $k \geq 2$, a $k$-step nilpotent group always contains a $2$-step nilpotent group. But being a $2$-step nilpotent group precisely means that its commutator subgroup is non-trivial and contains in its centre. Thus, it follows from Theorem~\ref{thm:ForNilpotent} that nilpotent subgroups in diagram groups are $1$-step nilpotent, i.e.\ abelian. 
\end{proof}

\noindent
It is worth noticing that, even though diagram groups do not contain interesting nilpotent groups, they may contain interesting solvable groups. For instance, we saw in Section~\ref{section:Examples} that the lamplighter group $\mathbb{Z} \wr \mathbb{Z}$, which is metabelian, is a diagram group; as well as the iterated wreath product $((\mathbb{Z} \wr \mathbb{Z}) \wr \cdots) \wr \mathbb{Z}$. In fact, it is contained in many diagram groups. It is proved in \cite{MR1725439} that a diagram group $D(\mathcal{P},w)$ contains a subgroup isomorphic to $\mathbb{Z} \wr \mathbb{Z}$ if and only if there exist words $x,y,z$ such that $D(\mathcal{P},z) \neq \{1\}$ and the equalities $w=xy$, $x=xz$, and $y=zy$ hold modulo $\mathcal{P}$. This includes for instance Thompson's group $F$. 

\medskip \noindent
We saw with Corollary~\ref{cor:FreeAbelian} that abelian subgroups in diagram groups are necessarily free. In fact, we can be more precise and describe the structures of abelian subgroups in terms of diagrams.

\begin{thm}[\cite{MR1725439}]\label{thm:FormAbelian}
Let $\mathcal{P}= \langle \Sigma \mid \mathcal{R} \rangle$ be a semigroup presentation and $w \in \Sigma^+$ a baseword. For all pairwise commuting elements $g_1, \ldots, g_m \in D(\mathcal{P},w)$, there exist an $(w,v_1 \cdots v_n)$-diagram $\Gamma$ and $(v_j,v_j)$-diagrams $\Delta_j$ such that
$$g_i = \Gamma \cdot \left( \Delta_1^{p_{i1}} + \cdots + \Delta_n^{p_{in}} \right) \cdot \Gamma^{-1} \text{ for some } p_{i1}, \ldots, p_{in} \in \mathbb{Z}$$
for every $1 \leq i \leq m$. Moreover, the diagrams $\Delta_1, \ldots, \Delta_n$ can be chosen either trivial or simple and absolutely reduced. 
\end{thm}

\noindent
As a consequence, one can show that abelian subgroups in diagram groups are undistorted.

\begin{cor}[\cite{MR1725439}]
In diagram groups, finitely generated abelian subgroups are undistorted.
\end{cor}

\begin{proof}
Given a diagram group, let $A$ be a finitely generated abelian subgroup. Fix a finite generating set $\{g_1, \ldots, g_m\}$ in $A$. According to Theorem~\ref{thm:FormAbelian}, there exist an $(w,v_1 \cdots v_n)$-diagram $\Gamma$ and $(v_j,v_j)$-diagrams $\Delta_j$ such that
$$g_i = \Gamma \cdot \left( \Delta_1^{p_{i1}} + \cdots + \Delta_n^{p_{in}} \right) \cdot \Gamma^{-1} \text{ for some } p_{i1}, \ldots, p_{in} \in \mathbb{Z}$$
for every $1 \leq i \leq m$. Moreover, we can choose the diagrams $\Delta_1, \ldots, \Delta_n$ either trivial or simple and absolutely reduced. Let $A^+$ denote the abelian subgroup freely generated by
$$\left\{ \Gamma \cdot \left( \epsilon(u_1 \cdots u_{k-1}) + \Delta_k + \epsilon(u_{k+1} \cdots u_n) \right) \cdot \Gamma^{-1} \mid 1 \leq k \leq n \text{ with $\Delta_k$ non-trivial} \right\}.$$
It suffices to show that $A^+$ is undistorted in our diagram group in order to conclude. So let $H$ be a finitely generated subgroup containing $A^+$. On the one hand, if $| \cdot |$ denotes the word length of $A^+$ given by our free basis, then it is clear that
$$|a| - 2 \# (\Gamma)  \leq \#(a) \leq |a| \cdot \max \{ \#(\Delta_i), 1 \leq i \leq n\} +2 \# (\Gamma)$$
for every $a \in A^+$. And, on the other hand, if we fix a generating set $\{s_1, \ldots, s_r \}$ of $H$ and a shortest representative $m=m(s_1, \ldots,s_r)$ of a given element $a \in A^+$, then
$$\#(a) = \#(m) \leq \mathrm{length}(m) \cdot \max \{ \#(s_i), 1 \leq i \leq r\}.$$
Thus, the length in $H$ of an element in $A^+$ is bounded below by its length in $A^+$. This concludes the proof of our corollary. 
\end{proof}

\subsection{Membership problem for the commutator subgroup}\label{section:ComSub}

\noindent
Let $\mathcal{P}= \langle \Sigma \mid \mathcal{R} \rangle$ be a semigroup presentation and $w \in \Sigma^+$ a baseword. Our goal in this section is to describe a natural morphism $\alpha : D(\mathcal{P},w) \to \mathscr{A}(\mathcal{P})$ from our diagram group $D(\mathcal{P},w)$ to some free abelian group $\mathscr{A}(\mathcal{P})$ whose kernel coincides with commutator subgroup of $D(\mathcal{P},w)$. The motivation being to solve efficiently the membership problem for the commutator subgroup of $D(\mathcal{P},w)$ in terms of diagrams, as soon as the word problem is solvable in $\mathcal{P}$. 

\medskip \noindent
So let $\mathscr{A}(\mathcal{P})$ denote the free abelian group generated by $\mathcal{S} \times \mathcal{R} \times \mathcal{S}$, where $\mathcal{S}$ denotes the semigroup given by $\mathcal{P}$. Given a diagram $\Delta$ over $\mathcal{P}$, we associate to each cell $\pi$ of $\Delta$ an element $\eta(\pi)$ of $\mathcal{S} \times \mathcal{R} \times \mathcal{S}$ as follows. The top and bottom paths of $\pi$ meet at two points: a source $\iota(\pi)$ and a sink $\tau(\pi)$. All the oriented paths in $\Delta$ from the source of $\Delta$ to $\iota(\pi)$ are labelled by positive words representing the same element in the semigroup $\mathcal{S}$; let $\ell(\pi)$ denote this element. Similarly, all the oriented paths in $\Delta$ from $\tau(\pi)$ to the sink of $\Delta$ are labelled by positive words representing the same element in the semigroup $\mathcal{S}$; let $r(\pi)$ denote this element. If $\pi$ is labelled by the relation $u=v$ with $(u=v) \in \mathcal{R}$ (resp. $(v=u) \in \mathcal{R}$), then we set $\eta(\pi):= (\ell(\pi), u=v, r(\pi))$ (resp. $-(\ell(\pi), v=u,r(\pi))$). Finally, we define
$$\alpha(\Delta):= \sum\limits_{\pi \text{ cell of } \Delta} \eta(\pi) \in \mathscr{A}(\mathcal{P}).$$
Let us consider an explicit example in order to illustrate our definition. Take the presentation
$$\mathcal{P}:= \left\langle x, a_i,b_i \ (i \geq 1) \mid x=x^2, a_i=a_{i+1}x, b_i=xb_{i+1} \ (i \geq 1) \right\rangle$$ 
from Proposition~\ref{prop:CommutatorF} and the following diagram $\Delta$:
\begin{center}
\includegraphics[width=0.7\linewidth]{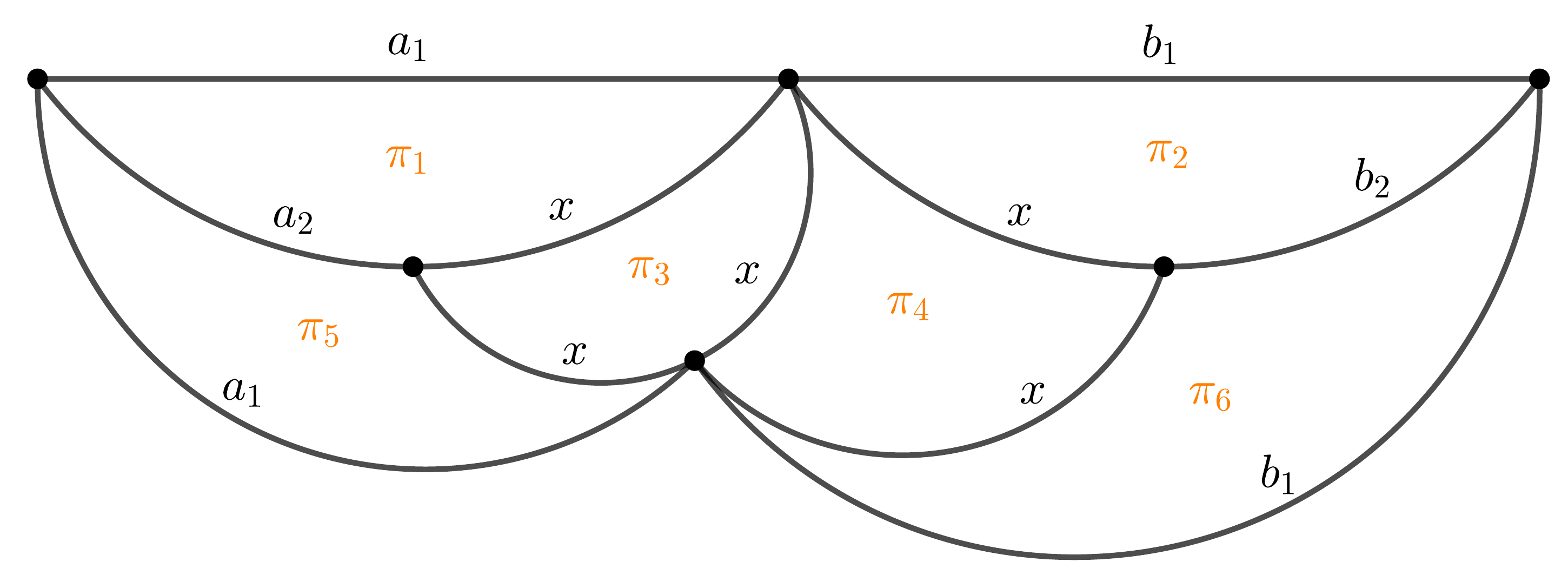}
\end{center}

\noindent
First of all, observe that $a_ix^n=a_1$ and $x^nb_i=b_1$ mod $\mathcal{P}$ for all $i,n \geq 1$. Indeed, we have
$$a_ix^n=a_{i+1}x^{n+1} = a_{i+1}x = a_i \text{ mod } \mathcal{P}$$
for all $i,n \geq 1$, and
$$a_1=a_2x = a_2=a_3x= a_3 = \cdots \text{ mod } \mathcal{P}.$$
We argue similarly for the equality mod $\mathcal{P}$ with $b_1$. This allows us to deduce that
$$\eta(\pi_5) = -(1, a_1=a_2x, xb_1) = -(1,a_1=a_2x,b_1) = - \eta(\pi_1),$$
$$\eta(\pi_6)= -(a_2x, b_1=xb_2, 1) = -(a_1, b_1=xb_2, 1) = -\eta(\pi_2),$$
and 
$$\eta(\pi_4)= - (a_2x, x=x^2, b_2) = -(a_2,x=x^2, b_1) = - \eta(\pi_3).$$
Therefore,
$$\alpha(\Delta)= \eta(\pi_1) + \cdots + \eta(\pi_6) = 0.$$
Coming back to the general case, the main result of interest for us is:

\begin{prop}[\cite{MR1396957}]\label{prop:Commutator}
Let $\mathcal{P}= \langle \Sigma \mid \mathcal{R} \rangle$ be a semigroup presentation and $w \in \Sigma^+$ a baseword. The kernel of $\alpha : D(\mathcal{P},w) \to \mathscr{A}(\mathcal{P})$ coincides with the commutator subgroup of $D(\mathcal{P},w)$.
\end{prop}

\noindent
We emphasize that $\alpha$ is not surjective in general, so we cannot identify $\mathscr{A}(\mathcal{P})$ with the abelianisation of $D(\mathcal{P},w)$. In fact, $\mathscr{A}(\mathcal{P})$ is usually much bigger than the abelianisation. 

\medskip \noindent
An idea to prove Proposition~\ref{prop:Commutator} is the following. If $\mathcal{P}$ is a complete semigroup presentation, then Theorem~\ref{thm:Presentation} provides a presentation of $D(\mathcal{P},w)$. Given an element $\Delta$ in the kernel of $\alpha$, we write it as a product of generators from our presentation, say $\Delta = \Phi_1 \cdots \Phi_n$. Using the fact that $0=\alpha(\Delta)=\alpha(\Phi_1)+ \cdots+ \alpha(\Phi_n)$, we deduce that $\Delta$ can be written as a product of commutators. This argument shows that Proposition~\ref{prop:Commutator} holds for complete semigroup presentation. The general case is then deduced thanks to Lemma~\ref{lem:Complete}. See \cite[Theorem~11.3]{MR1396957} for a detailed proof. 

\begin{ex}
Thanks to Proposition~\ref{prop:Commutator}, we deduce that the diagram $\Delta$ from our previous example belongs to the commutator subgroup of the corresponding diagram group. This is not surprising since the diagram group coincides with $F'$, a simple group. Nevertheless, it can easily verified thanks to Propositions~\ref{prop:Commutator} and~\ref{prop:CommutatorF} that $F'$ is perfect. 
\end{ex}

\begin{ex}\label{ex:CommutatorF}
Let $\mathcal{P}=\langle x \mid x=x^2 \rangle$ be the semigroup presentation naturally associated to Thompson's group $F$. Because $\mathcal{P}$ has a single relation and because the semigroup defined by $\mathcal{P}$ contains exactly two elements, namely $1$ and $x$, let us say that the abelian group $\mathscr{A}(\mathcal{P})$ is freely generated by $\{1,x\}^2$. Given a reduced diagram $\Delta$ over $\mathcal{P}$, observe that the top cell is sent to $(1,1)$ under $\eta$, that the bottom cell is sent to $-(1,1)$ under $\eta$, and that no other cell is sent to $\pm (1,1)$. We refer to a cell sent to $\pm (1,x)$ as left, to a cell sent to $\pm (x,1)$ as right, and to a cell sent to $\pm (x,x)$ as internal. Remember that the cells of $\Delta$ labelled by $x=x^2$ (resp. $x^2=x$) are called positive (resp. negative); this convention is compatible with the sign of the image of a cell under $\eta$. We have
$$\alpha(\Delta)= \begin{array}{c} \text{\# of pos. left cells} \\ - \text{\# of neg. left cells} \end{array} (1,x) + \begin{array}{c} \text{\# of pos. right cells} \\ - \text{\# of neg. right cells} \end{array} (x,1) $$
$$+ \begin{array}{c} \text{\# of pos. internal cells} \\ - \text{\# of neg. internal cells} \end{array} (x,x).$$
Thus, $\Delta$ belongs to the commutator subgroup of $D(\mathcal{P},w)$ if and only if $\Delta$ has the same number of positive and negative left cells and the same number of positive and negative right cells. (The number of positive and negative cells being the same, these two equalities automatically implies that $\Delta$ has the same number of positive and negative internal cells as well.)
\end{ex}

\subsection{Foldings and closed subgroups}\label{section:Fold}

\noindent
The method of \emph{folding} graphs \cite{MR695906} has been quite influential in the study of free groups and has been generalised in various directions. Interestingly, thinking of a diagram groups as second fundamental groups of directed $2$-complexes, foldings appear naturally in a similar way. This allows us, in particular, to solve the membership problem for some subgroups. 

\medskip \noindent
First of all, let us recall how to solve the membership problem in free groups thanks to folding. In other words, given a free group $F$ and elements $g,h_1, \ldots, h_n \in F$, we want to determine whether or not $g$ belongs to the subgroup $H:= \langle h_1, \ldots, h_n$. We fix a free basis $s_1, \ldots, s_r$ of $F$ and we think of $F$ as the fundamental group of a bouquet $\Omega$ of $r$ circles, each oriented and labelled by a generator $s_i$. To $H$, we associate a bouquet $\Gamma_0$ of $n$ circles, each subdivided and labelled by a generator $h_i$. Observe that both $\Omega$ and $\Gamma_0$ have a distinguished vertex.
\begin{figure}[h!]
\begin{center}
\includegraphics[width=0.6\linewidth]{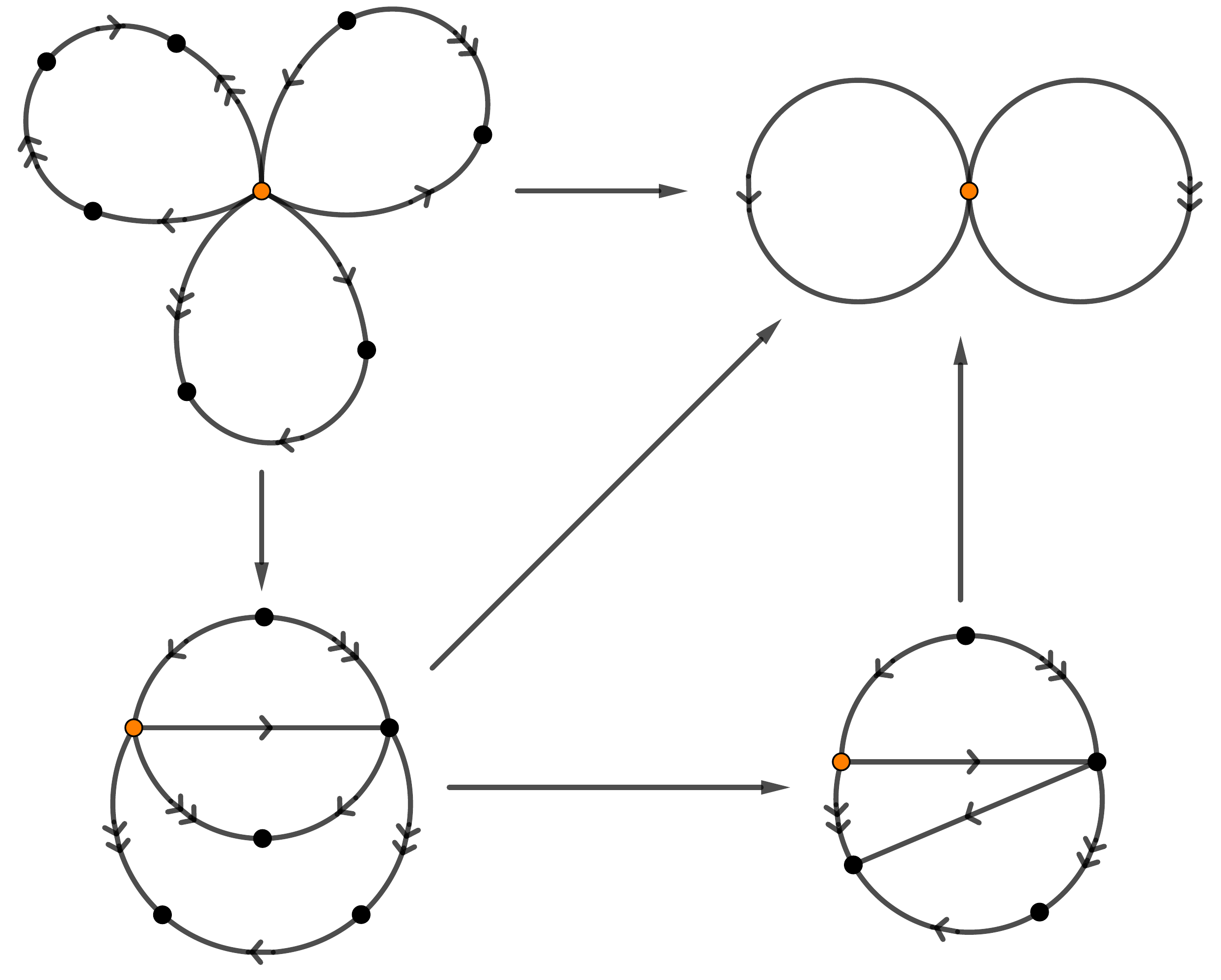}
\caption{Foldings for $\langle ab^{-1}a, abab^{-1}, a^2b^{-1} \rangle$ in $\langle a,b \mid \ \rangle$.}
\label{Fold}
\end{center}
\end{figure}

\medskip \noindent
Now, we fold $\Gamma_0$. More precisely, if two edges of $\Gamma_0$ share an endpoint and have the same image in $\Omega$, then we identify it. After finitely many folds $\Gamma_0 \to \cdots \to \Gamma_k$, one gets a graph $\Gamma_k$ with a locally injective map $\iota : \Gamma_k \to \Omega$ such that $H$ coincides with the image under $\iota$ of the fundamental group of $\Gamma_k$ in the fundamental group of $\Omega$. (The basepoints defining the fundamental groups being the distinguished vertices.) The key property is that an element of $F$ belongs to $H$ if and only if the reduced word $w$ representing it is \emph{accepted} by $\Gamma_k$, i.e.\ there is a loop in $\Gamma_k$ based at the distinguished vertex which is labelled by $w$. As an illustration, one deduces from Figure~\ref{Fold} that $ba^{-1}b^{-2}aba^{-2}$ belongs to $\langle ab^{-1}a, abab^{-1}, a^2b^{-1} \rangle$, but $ba^{-3}b^2$ does not. 

\medskip \noindent
This method for finite graphs can be naturally adapted to finite directed $2$-complexes. This has been done by V. Guba and M. Sapir in the late 1990s (unpublished), but it first appeared in print in \cite{MR3710646}. 

\medskip \noindent
Let $\mathcal{P}=\langle \Sigma \mid \mathcal{R} \rangle$ be a semigroup presentation and $w \in \Sigma^+$ a baseword. We think of the diagram group $D(\mathcal{P},w)$ as the second fundamental group of the directed $2$-complex $\mathcal{X}(\mathcal{P})$ based at the $1$-path labelled by $w$. See Section~\ref{section:Free} for the relevant definitions. We fix some spherical diagrams $\Delta, \Delta_1, \ldots, \Delta_n \in D(\mathcal{P},w)$ and we would like know whether $\Delta$ belongs to the subgroup $H:=\langle \Delta_1, \ldots, \Delta_n \rangle$. 

\medskip \noindent
To $H$ we associate the directed $2$-complex $\mathcal{X}_0$ obtained by identifying all the top and bottom paths of the $\Delta_i$ to a single segment. Topologically, $\mathcal{X}_0$ can be thought of as a bouquet of $2$-spheres pairwise intersecting along a fixed segment. Observe that there is a natural morphism of directed $2$-complexes $\mathcal{X}_0 \to \mathcal{X}(\mathcal{P})$ and that both $\mathcal{X}(\mathcal{P})$ and $\mathcal{X}_0$ have a distinguished $1$-path. A \emph{folding} refers to the following operation: if one sees two cells in $\mathcal{X}_0$ sharing their top or bottom paths and having the same image in $\mathcal{X}(\mathcal{P})$, then we identify them. After finitely many folds $\mathcal{X}_0 \to \cdots \to \mathcal{X}_k$, one gets a directed $2$-complex $\mathcal{X}_k$ with a locally injective morphism $\mathcal{X}_k \to \mathcal{X}(\mathcal{P})$. 

\medskip \noindent
Define the closure $\mathrm{Cl}(H)$ of $H$ as the subgroup of $D(\mathcal{P},w)$ given by the diagrams \emph{accepted} by $\mathcal{X}_k$, i.e.\ the image of the second fundamental group of $\mathcal{X}_k$ based at its distinguished $1$-path under the morphism $\mathcal{X}_k \to \mathcal{X}(\mathcal{P})$. (It can be shown that the closure of a subgroup does not depend on a particular choice of generators.) By construction, the closure of $H$ contains $H$, but, contrary to the case of free groups, it can be bigger. One can say that the closure of a subgroup is itself a diagram group, but subgroups of diagram groups may not be diagram groups themselves, so there must exist subgroups properly contained in their closure. In fact, it is easy to construct subgroups distinct from their closures. Indeed, if $\Delta$ is a diagram decomposing as a sum of two spherical diagrams $\Delta_1+ \Delta_2$, then the closure of $\langle \Delta \rangle$ contains $\Delta_1+ \epsilon(u_2)$ and $\epsilon(u_1)+\Delta_2$, where each $\epsilon(u_i)$ denotes the trivial diagram having the same label as the top and bottom paths of $\Delta_i$. Indeed, when identifying the top and bottom paths of $\Delta$, we do not get a topological sphere, but a pointed sum of two spheres due to the cut-point separating $\Delta_1$ and $\Delta_2$ in $\Delta$. It is conjectured that this phenomenon is essentially the only obstruction for a subgroup to coincide with its closure. See \cite[Conjecture~3.11]{Golan} for a precise statement. 

\medskip \noindent
In conclusion, for every diagram group, there exists an algorithm that decides, given spherical diagrams $\Delta, \Delta_1, \ldots, \Delta_n$, whether or not $\Delta$ belongs to the closure of the subgroup $\langle \Delta_1, \ldots, \Delta_n \rangle$. In particular, the membership problem is solvable for \emph{closed subgroups} (i.e.\ subgroups that coincide with their closures).

\section{Median geometry}\label{section:Median}

\subsection{Crash course on median geometry}\label{section:Crash}

\noindent
In this preliminary section, we record some basic definitions and results related to \emph{median graphs} that will be used in the next sections.

\begin{definition}
A connected graph $X$ is \emph{median} if, for all vertices $x_1,x_2,x_3 \in X$, there exists a unique vertex $m \in X$ satisfying
$$d(x_i,x_j)=d(x_i,m)+d(m,x_j) \text{ for all } i \neq j.$$
We refer to this vertex $m$ as the \emph{median point} of $x_1,x_2,x_3$. 
\end{definition}
\begin{figure}[h!]
\begin{center}
\includegraphics[scale=0.4]{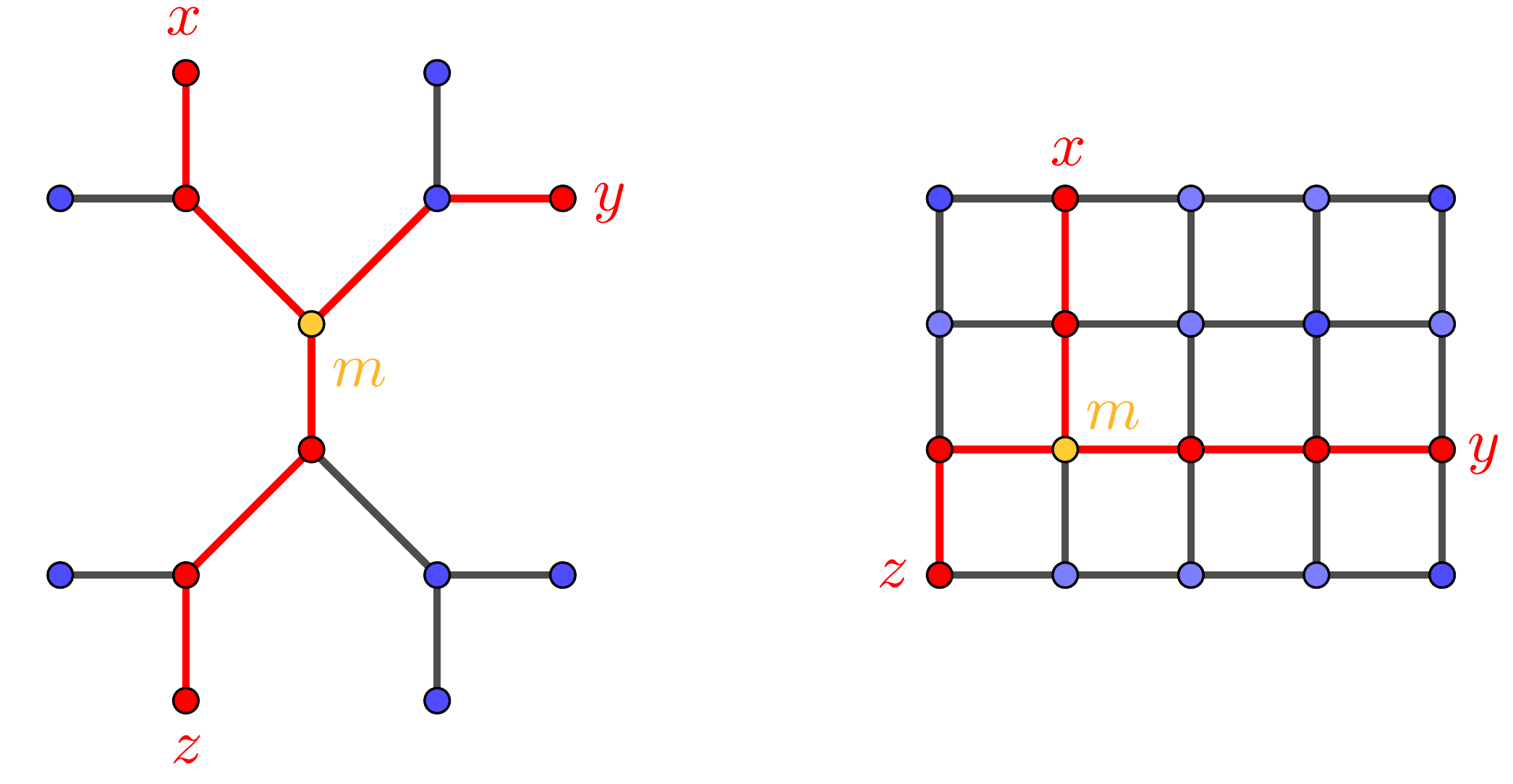}
\caption{Examples of median graphs and median points.}
\label{ExMedian}
\end{center}
\end{figure}

\noindent
Examples include of course simplicial trees, the median point of a triple of vertices corresponding to the centre of the tripod they delimit. A product of median graphs is still median, so product of trees are also median graphs. This includes in particular (one-skeleta of) cubes of arbitrary dimensions. 

\medskip \noindent
A fundamental idea is that the geometry of a median graph is essentially encoded in the combinatorics of its \emph{hyperplanes}. 

\begin{definition}
Let $X$ be a median graph. A(n \emph{oriented}) \emph{hyperplane} $J$ is an equivalence class of (oriented) edges with respect to the transitive closure of the relation that identifies two (oriented) edges when they are opposite sides of a $4$-cycle. 
\begin{itemize}
	\item If $X \backslash \backslash J$ denotes the graph obtained from $X$ by removing all the edges of $J$, then a connected component of $X \backslash \backslash J$ is a \emph{halfspace}. Two subsets $A,B \subset X$ are \emph{separated} by $J$ if they lie in distinct halfspaces delimited by $J$.
	\item Two hyperplanes $J_1$ and $J_2$ are \emph{transverse} if there exist two intersecting edges $e_1 \subset J_1$ and $e_2 \subset J_2$ that span a $4$-cycle. 
	\item They are \emph{tangent} if exist two intersecting edges $e_1 \subset J_1$ and $e_2 \subset J_2$ that do not span a $4$-cycle.
\end{itemize}
See Figure~\ref{Hyp} for a few examples.
\end{definition}
\begin{figure}[h!]
\begin{center}
\includegraphics[scale=0.4]{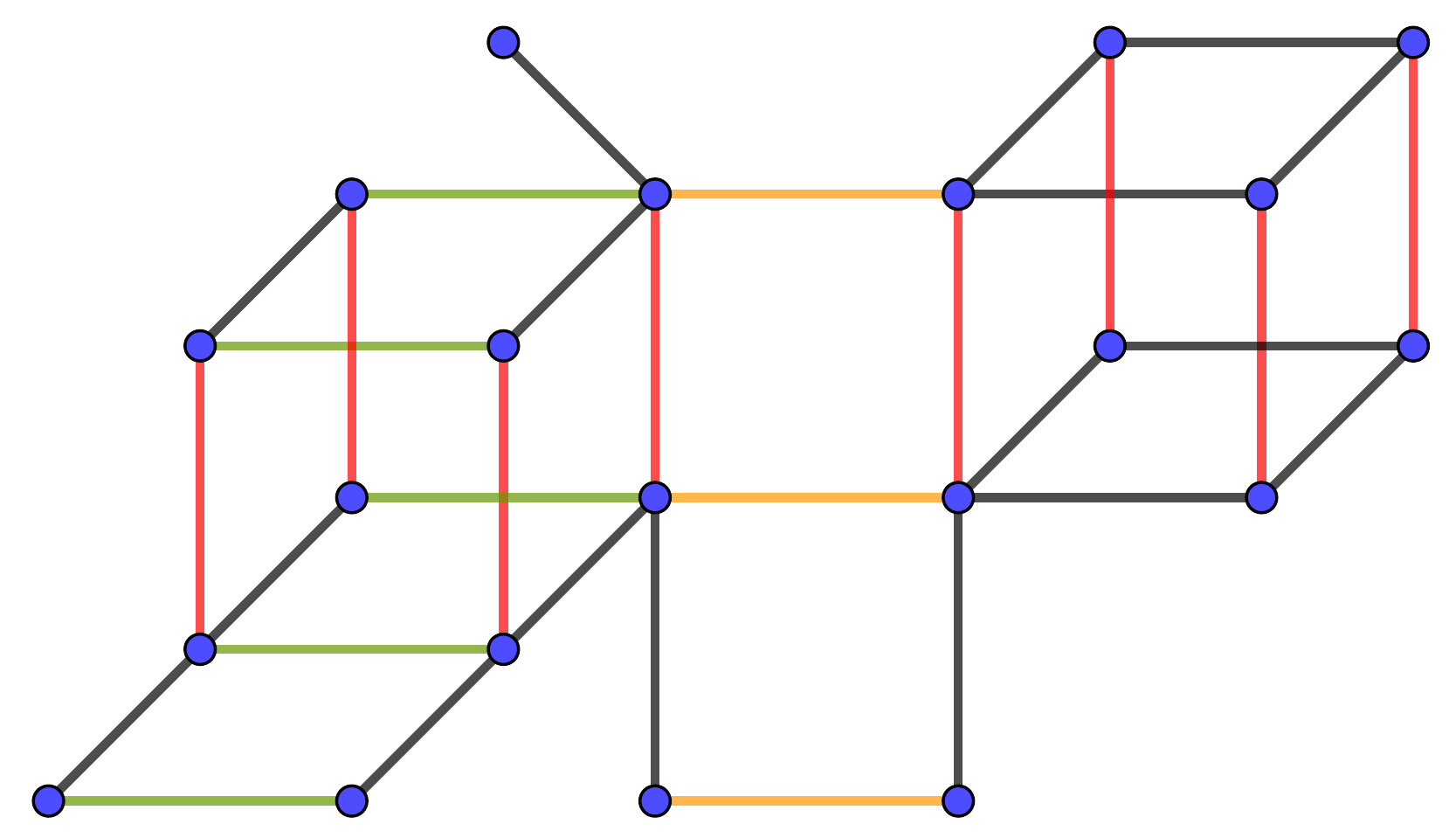}
\includegraphics[width=0.45\linewidth]{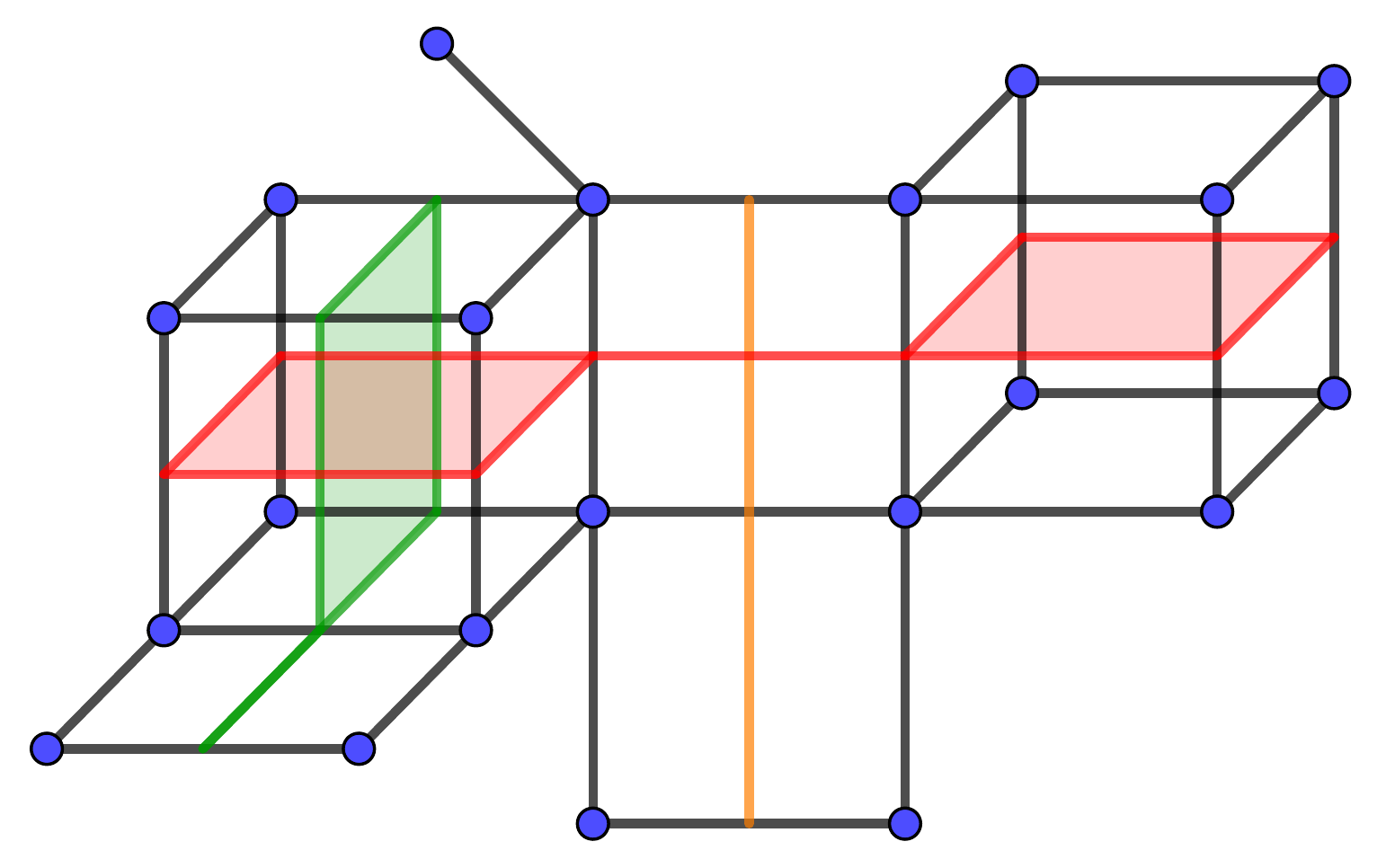}
\caption{Hyperplanes in a median graph and its cube-completion. The red hyperplane is transverse to the green and yellow hyperplanes. Green and yellow hyperplanes are tangent.}
\label{Hyp}
\end{center}
\end{figure}

\noindent
The previous claim is mainly motivated by the following statement:

\begin{thm}[\cite{MR1347406}]
Let $X$ be a median graph. The following assertions hold.
\begin{itemize}
	\item Every hyperplane $J$ separates. More precisely, $X \backslash \backslash J$ has exactly two connected components, which turn out to be convex.
	\item A path in $X$ is a geodesic if and only if it crosses each hyperplane at most once.
	\item The distance between two vertices $x,y \in X$ coincides with the number of hyperplanes separating $x,y$. 
\end{itemize}
\end{thm}

\noindent
In geometric group theory, median graphs are better known as \emph{CAT(0) cube complexes}. 

\begin{definition}
A cube complex is \emph{nonpositively curved} if the links of its vertices are flag simplicial complexes.
\end{definition}

\noindent
Recall that the \emph{link} of a vertex $x$ is the complex whose vertices are the (half)edges starting from $x$ and whose simplices are given by edges than span a cell. Typically, the link of a vertex should be thought of as a small ball around the vertex, endowed with the cellular structure induced by the whole complex. A complex is \emph{flag} if, for all $n \geq 1$, any $n+1$ vertices span an $n$-simplex if and only if they are pairwise adjacent. 
\begin{center}
\includegraphics[width=\linewidth]{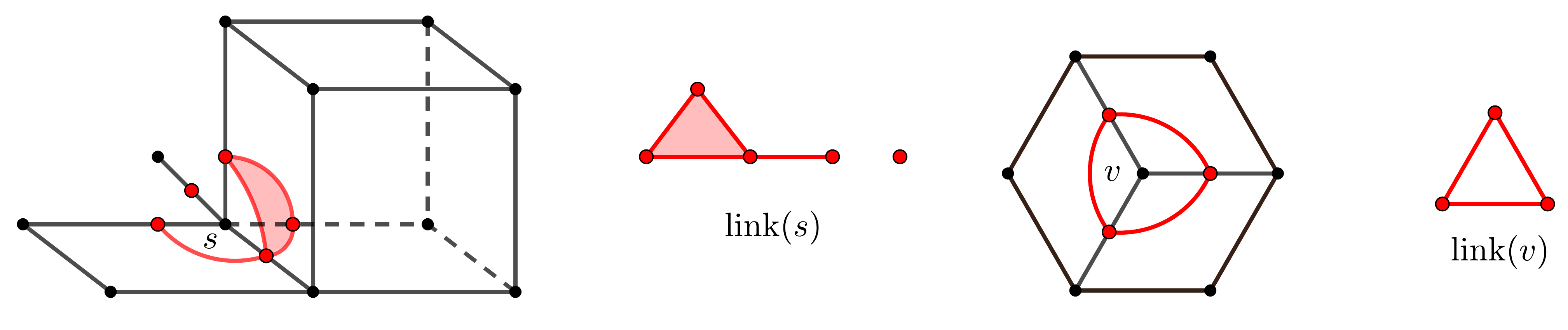}
\end{center}

\noindent
The terminology is justified by the fact that a cube complex, when endowed by the length metric extending the Euclidean metrics of its cubes, is \emph{locally CAT(0)} exactly when it is nonpositively curved \cite{GromovHyp, Leary}. As a consequence, a cube complex endowed with the same metric is \emph{CAT(0)} if and only if it is simply connected and nonpositively curved. We refer to \cite{MR1744486} for more information on CAT(0) geometry. It turns out that CAT(0) cube complexes and median graphs define essentially the same objects. More precisely:

\begin{thm}[\cite{MR1663779, mediangraphs, Roller}]\label{thm:MedianVScube}
A graph is median if and only if its cube-completion is simply connected and nonpositively curved. 
\end{thm}

\noindent
Here, given a graph $X$, its \emph{cube-completion} refers to the cube complex $X_\square$ obtained by filling with cubes all the subgraphs isomorphic to one-skeleta of cubes. 

\medskip \noindent
Usually, hyperplanes are thought of slightly differently in cube-completions of median graphs: they are thought of as unions of \emph{midcubes}. See Figure~\ref{Hyp}. The terminology related to hyperplanes in median graphs can be easily adapted for this alternative point of view.

\medskip \noindent
The following observation can be thought of as a consequence of the CAT(0)ness of the cube complexes under consideration, but it can also be proved directly thanks to the combinatorics of hyperplanes.

\begin{prop}
Cube-completions of median graphs are contractible.
\end{prop}

\noindent
For more information on median graphs, or equivalently CAT(0) cube complexes, we refer to \cite{MR3329724, Book}.

\subsection{Median geometry of diagram groups}\label{section:MedianDiag}

\noindent
In this section, we show that diagram groups act on median graphs. These median graphs can be described in several equivalent ways. For now, we give a direct and explicit definition and postpone more conceptual characterisations.

\begin{definition}\label{def:MedianDiag}
Let $\mathcal{P}$ be a semigroup presentation. Let $M(\mathcal{P})$ denote the graph whose vertices are the diagrams over $\mathcal{P}$ modulo dipole reduction and whose edges connect two diagrams whenever one can be obtained from another by right-concatenation by an \emph{atomic diagram} (i.e.\ a diagram with a single $2$-cell). 
\end{definition}

\noindent
In other words, the vertices are given by reduced diagrams and one passes from a vertex to a neighbour by gluing below a $2$-cell. In the sequel, vertices of $M(\mathcal{P})$ will be often thought of as reduced diagrams. 

\begin{figure}[h!]
\includegraphics[width=\linewidth]{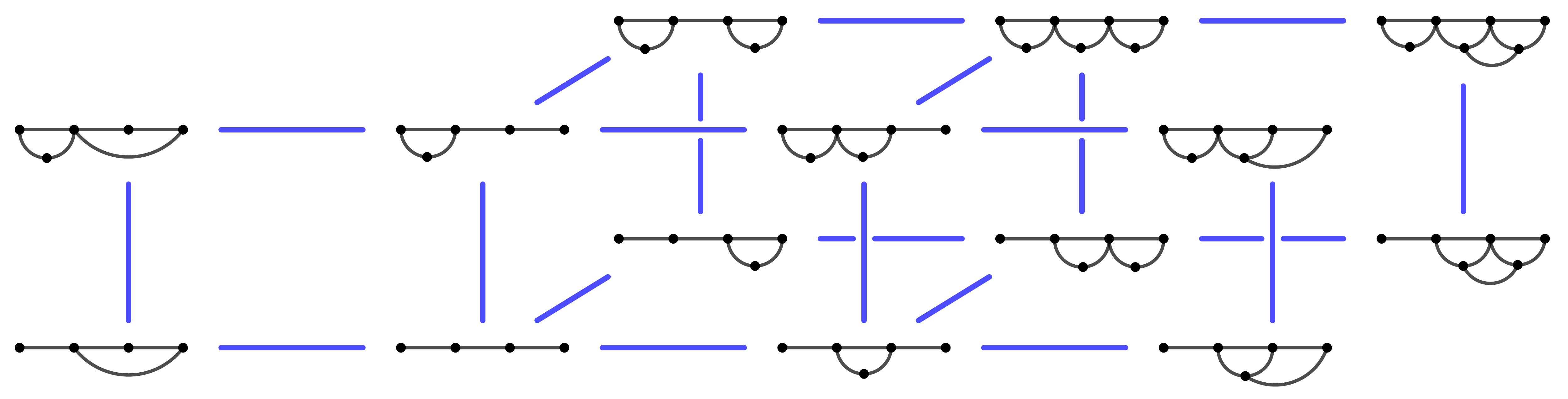}
\caption{A piece of $M(\mathcal{P},x^3)$ for $\mathcal{P}= \langle x \mid x=x^2 \rangle$.}
\label{MedianGraph}
\end{figure}

\noindent
One easily verifies that two diagrams belong to the same connected component of $M(\mathcal{P})$ if and only if their top paths have the same label. Given a word $w \in \Sigma^+$, we denote by $M(\mathcal{P},w)$ the connected component of $M(\mathcal{P})$ containing the diagrams whose top paths are labelled by $w$. 

\medskip \noindent
Notice that the diagram group $D(\mathcal{P},w)$ naturally acts on $M(\mathcal{P},w)$ by left multiplication. 

\medskip \noindent
It is worth noticing that, since edges of $M(\mathcal{P})$ are labelled by atomic diagrams, the paths in $M(\mathcal{P})$ are naturally labelled by the diagrams obtained by concatenating the atomic diagrams given by the edges along the path. Whether such a diagram is reduced characterises the geodesics in $M(\mathcal{P})$. More precisely:

\begin{lemma}\label{lem:Geodesics}
Let $\mathcal{P}$ be a semigroup presentation, $w \in \Sigma^+$ a baseword, and $\Phi,\Psi \in M(\mathcal{P},w)$ two vertices. A path
$$\Phi, \ \Phi \circ \Delta_1, \ \Phi \circ \Delta_1 \circ \Delta_2, \ldots, \ \Phi \circ \Delta_1 \circ \cdots \circ \Delta_n= \Psi$$
is a geodesic if and only if the diagram $\Delta_1 \circ \cdots \circ \Delta_n$ is reduced. As a consequence, the distance between $\Phi$ and $\Psi$ in $M(\mathcal{P},w)$ coincides with the number $\#( \Phi^{-1} \Psi)$ of $2$-cells in the reduction of $\Phi^{-1} \Psi$. 
\end{lemma}

\begin{proof}
It is clear that any path from $\Phi$ to $\Psi$ must have length at least $\# (\Phi^{-1}\Psi)$, and that there exists at least one such path of this length. So the distance between $\Phi$ and $\Psi$ in $M(\mathcal{P},w)$ coincides with $\#(\Phi^{-1} \Psi)$. Therefore, our path is a geodesic if and only if $\Delta_1 \circ \cdots \circ \Delta_n = \Phi^{-1}\Psi$ with $n=\# (\Phi^{-1} \Psi)$. The only possibility is that there is no dipole in $\Delta_1 \circ \cdot \circ \Delta_n$, which amounts to saying that $\Delta_1 \circ \cdots \circ \Delta_n$ is reduced. 
\end{proof}

\noindent
The following theorem is the main result of this section.

\begin{thm}[\cite{MR1978047}]\label{thm:FarleyCC}
Let $\mathcal{P}$ be a semigroup presentation and $w \in \Sigma^+$ a baseword. Then $M(\mathcal{P},w)$ is a median graph, locally of finite cubical dimension, on which $D(\mathcal{P},w)$ acts freely. The action is cocompact if and only if $[w]_\mathcal{P}$ is finite. 
\end{thm}

\noindent
Here, we use the notation $[w]_\mathcal{P}:= \{ m \in \Sigma^+ \mid m=w \text{ mod } \mathcal{P} \}$. We are going to give two proofs of Theorem~\ref{thm:FarleyCC}. The first one verify that every triple of vertices admits a unique median point thanks to the description of the geodesics provided by Lemma~\ref{lem:Geodesics}; and the second proof applies the characterisation of median graphs provided by Theorem~\ref{thm:MedianVScube}. 

\begin{proof}[First proof of Theorem~\ref{thm:FarleyCC}.]
Let $\Delta_0,\Delta_1,\Delta_2$ be three reduced diagrams representing three vertices of $M(\mathcal{P},w)$. Up to conjugating by $\Delta_0$, we can assume without loss of generality that $\Delta_0$ has no cell. (Notice that this operation may modify the baseword $w$.) Let $\Delta$ be the biggest common prefix of $\Delta_1$ and $\Delta_2$. (Given two diagrams $\Phi,\Psi$, one says that $\Phi$ is a prefix of $\Psi$ if $\Psi$ decomposes as a concatenation $\Phi \circ \Gamma$ for some diagram $\Gamma$.) The fact that $\Delta$ is a common prefix of $\Delta_1$ and $\Delta_2$ implies that $\Delta$ belongs to geodesics from $\Delta_0$ to $\Delta_1$ and from $\Delta_0$ to $\Delta_2$. And the maximality of this prefix assures that, if we decompose $\Delta_1$ (resp. $\Delta_2$) as $\Delta \circ \Gamma_1$ (resp. $\Delta \circ \Gamma_2$), then $\Gamma_1 \circ \Gamma_2^{-1}$ is reduced. This implies that $\Delta$ also belongs to a geodesic between $\Delta_1$ and $\Delta_2$. Thus, $\Delta$ provides a median point for $\Delta_0,\Delta_1,\Delta_2$. 

\medskip \noindent
If $\Xi$ is a median point of $\Delta_0,\Delta_1,\Delta_2$, then $\Xi$ must be a common prefix of $\Delta_1$ and $\Delta_2$, and consequently a prefix of $\Delta$. Decompose $\Delta$ as $\Xi \circ \Omega$. Because $\Xi$ belongs to a geodesic between $\Delta_1$ and $\Delta_2$, the concatenation $(\Omega \circ \Gamma_1)^{-1} (\Omega \circ \Gamma_2)$ must be reduced. In other words, $\Omega$ cannot have any cell, hence $\Xi= \Delta$. 
\end{proof}

\begin{proof}[Second proof of Theorem~\ref{thm:FarleyCC}.]
According to Theorem~\ref{thm:MedianVScube}, it suffices to show that the cube-completion $M_\square(\mathcal{P},w)$ is simply connected and nonpositively curved. 

\medskip \noindent
We begin by describing the cubes in $M_\square(\mathcal{P},w)$. A \emph{thin diagram} is a sum of atomic diagrams (see Definition~\ref{def:SumDiag}). Given a diagram $\Delta \in M(\mathcal{P},w)$ and thin diagram $\Psi$ with top label equal to the bottom label of $\Delta$, the vertices $\{ \Delta \circ \Psi_0 \mid \Psi_0 \text{ prefix of } \Psi\}$ span an $n$-cube if $\Psi$ has $n$ cells. One can show that all the cubes in $M_\square(\mathcal{P},w)$ are of this form. 

\medskip \noindent
As a consequence, given a vertex $\Delta \in M_\square(\mathcal{P},w)$, if $w$ denotes the bottom label of $\Delta$ then the link of $\Delta$ is isomorphic to the simplicial complex whose vertices are the subwords of $w$ which are part of a relation from $\mathcal{R}$ and whose simplices are given by pairwise disjoint subwords. Clearly, such a complex is flag, so $M_\square(\mathcal{P},w)$ is nonpositively curved.

\medskip \noindent
Now, let us show that every compact subcomplex in $M_\square(\mathcal{P},w)$ is contractible. More precisely, given a finite set of vertices $S$, we claim that the subcomplex spanned by $\bar{S}:= \{ \Delta \mid \exists \Omega \in S, \text{ $\Delta$ prefix of } \Omega \}$ is contractible. According to Lemma~\ref{lem:Geodesics}, $\bar{S}$ coincides with the union of all the geodesics from the trivial diagram to the vertices in $S$. Fix a reduced diagram $\Delta \in \bar{S}$ with a maximal number of cells (which amounts to saying that the distance from the trivial diagram to $\Delta$ is maximal). Decompose $\Delta$ as the concatenation $\Delta_0 \circ \Gamma$ where $\Gamma$ is a thin diagram containing all the cells of $\Delta$ whose bottom paths lie in the bottom path of $\Delta$. (In other words, $\Gamma$ is the maximal thin suffix of $\Delta$.) If $\Gamma$ has $n$ cells, let $\Delta_1, \ldots, \Delta_n$ denote the diagrams obtained from $\Delta$ by removing a single cell from $\Gamma$. Then
$$\bar{S}= \overline{ \underset{=:R}{\underbrace{S \backslash \{\Delta\} \cup \{\Delta_1, \ldots, \Delta_n\} }} } \cup \underset{\text{$n$-cube}}{\underbrace{ \left\{ \Delta_0 \circ \Gamma_0 \mid \Gamma_0 \text{ prefix of } \Gamma\right\} }}.$$
Thus, $\bar{S}$ deformation retracts onto $\bar{R}$. During the replacement, observe that the quantity
$$\left( \begin{array}{c} \text{maximal number of cells} \\ \text{of a diagram in $S$} \end{array}, \ \begin{array}{c} \text{number of vertices with a} \\ \text{maximal number of cells} \end{array} \right),$$
ordered with respect to the lexicographic order, decreases. Therefore, the process has to stop after finitely many steps. Necessarily, the final step corresponds to the situation where our set of vertices is reduced to the trivial diagram. 
\end{proof}

\noindent
Interestingly, the median geometry of diagram groups can replace the combinatorics of diagrams in some arguments. We record some of examples below.

\begin{proof}[Second proof of Theorem~\ref{thm:TorsionFree}.]
Let $\mathcal{P}= \langle \Sigma \mid \mathcal{R} \rangle$ be a semigroup presentation and $w \in \Sigma^+$ a baseword. A finite-order isometry of a median graph has to stabilise (the one-skeleton of) a cube (for instance, combine \cite[Lemma~2.1.3]{MR1418304} with \cite{MR913179}). Because cube-stabilisers are vertex-stabilisers for the action of $D(\mathcal{P},w)$ on $M(\mathcal{P},w)$, the fact that $D(\mathcal{P},w)$ acts freely on $M(\mathcal{P},w)$ immediately implies that $D(\mathcal{P},w)$ is torsion free. 
\end{proof}

\begin{proof}[Proof of a weak version of Theorem~\ref{thm:Nilpotent}.]
If a group acts properly on a median graph, then its finitely generated abelian subgroups are automatically undistorted \cite[Corollary~6.B.5]{CornulierCommensurated} and its polycyclic subgroups are automatically virtually abelian \cite[Corollary~6.C.8]{CornulierCommensurated}. 
\end{proof}

\begin{proof}[Second proof of Prosition~\ref{prop:Roots}.]
Let $\mathcal{P}= \langle \Sigma \mid \mathcal{R} \rangle$ be a semigroup presentation and $w \in \Sigma^+$ a baseword. Given an isometry of a median graph, if no power inverts a hyperplane (i.e.\ stabilises the hyperplane but swaps the two halfspaces it delimits), then it acts as a translation on a bi-infinite geodesic lines \cite{arXiv:0705.3386}. Because $D(\mathcal{P},w)$ acts on $M(\mathcal{P},w)$ with no inversion, it follows that a non-trivial element $g \in D(\mathcal{P},w)$ does act as a translation on some bi-infinite geodesic. This amounts to saying that there exists a diagram $\Delta$ and a geodesic $[\Delta, g\Delta]$ from $\Delta$ to $g \Delta$ such that the concatenation $\bigcup_{n \in \mathbb{Z}} b^n \cdot [\Delta,g \Delta]$ is a geodesic. Observe that the translation length $\tau(g)$ of $g$ is $d(\Delta, g \Delta) = \# (g)$. 

\medskip \noindent
Now, if $a,b \in D(\mathcal{P},w) \backslash \{1\}$ and $k \neq 0$ are such that $a=b^k$, then 
$$\#(a)= \tau(a) = |k| \cdot \tau(b) = |k| \cdot \#(b),$$
which implies that $|k| \leq \#(a)/2$. 
\end{proof}

\noindent
As already mentioned, there are several equivalent but more conceptual view on the median graph given by Definition~\ref{def:MedianDiag}. Below, we record three of them.

\paragraph{Cayley graph of the diagram groupoid.} Let $\mathcal{P}= \langle \Sigma \mid \mathcal{R} \rangle$ be a semigroup presentation. Recall from Section~\ref{section:Diagrams} that the diagram groupoid $D(\mathcal{P})$ is defined as the set of diagrams over $\mathcal{P}$ up to dipole reduction endowed with the concatenation of diagrams. Because every diagram can be constructed by adding one $2$-cell at a time, the atomic diagrams generate the groupoid $D(\mathcal{P})$. It follows directly from Definition~\ref{def:MedianDiag} that $M(\mathcal{P})$ coincides with the Cayley graph of $D(\mathcal{P})$ with respect to the generating set given by the atomic diagrams, i.e.\ the graph whose vertices are the elements of $D(\mathcal{P})$ and whose edges connect one element to another if one can be obtained from the other by right-multiplying by a generator. 

\medskip \noindent
This observation is compatible with our analogy between diagram groups and free groups. 
\begin{center}
\begin{tabular}{|c|c|} \hline
\textbf{Free groups} & \textbf{Diagram groups} \\ \hline
Letters & $2$-cells \\ \hline
Free generators & atomic diagrams \\ \hline
Words & Semigroup diagrams \\ \hline
Cancellations & Dipole reductions \\ \hline
\end{tabular}
\end{center}

\noindent
From this perspective, Cayley graphs of free groups with respect to free bases correspond to Cayley graphs of diagram groupoids with respect to atomic diagrams. Then, the arboreal geometry of free groups becomes the median geometry of diagram groups. 

\begin{remark}
Pushing further this analogy, one could say that free groups have natural boundaries, which can be described both as infinite words and as geodesic rays in their trees; and ask whether there is a counterpart for diagram groups. It turns out to the case: one can define boundaries of diagram groups both as infinite semigroup diagrams and as geodesic rays in their median graphs. Such a boundary coincides with the usual boundary of median graphs: the \emph{Roller boundary}. See \cite[Proposition~A.3]{MR4071367} for more details. 
\end{remark}

\paragraph{Universal cover of the Squier cube complex.} It turns out that (cube-completions of) our median graphs can be constructed directly from the Squier cube complexes. More precisely:

\begin{prop}
Let $\mathcal{P}= \langle \Sigma \mid \mathcal{R} \rangle$ be a semigroup presentation. The map that sends a semigroup diagram over $\mathcal{P}$ to the labels of its bottom path induces a universal cover $M_\square(\mathcal{P}) \to S^+(\mathcal{P})$. 
\end{prop}

\noindent
The proposition is essentially immediate from the definitions of the spaces involved: it is clear that the map $M_\square(\mathcal{P}) \to S^+(\mathcal{P})$ induces isomorphisms between links of vertices, so it is a covering map, and we already saw that each connected component of $M_\square(\mathcal{P})$ is simply connected.

\paragraph{From directed $2$-complexes to cube complexes.} One can naturally associate to every directed $2$-complex $\mathcal{X}$ a cube complex $\mathrm{Sq}(\mathcal{X})$, namely the cube-completion of the graph whose vertices are the $1$-paths in $\mathcal{X}$ and whose edges connect two $1$-paths whenever one can be obtained from the other by an elementary transformation. Observe that, if $\mathcal{P}$ is a semigroup presentation and if $\mathcal{X}(\mathcal{P})$ is the associated directed $2$-complex, then $\mathrm{Sq}(\mathcal{X}(\mathcal{P}))$ coincides with the Squier cube complex $S^+(\mathcal{P})$. It turns out that the cube-completion $M_\square(\mathcal{P})$ can also be constructed from a directed $2$-complex. 

\begin{definition}
A \emph{rooted $2$-tree} is a directed $2$-complex with a distinguished $1$-path $\alpha$ which can be described as a union of directed $2$-complexes $\mathcal{X}_0 \subset \mathcal{X}_1 \subset \cdots$ such that:
\begin{itemize}
	\item $\mathcal{X}_0$ coincides with the $1$-path $\alpha$;
	\item for every $n \geq 0$, $\mathcal{X}_{n+1}$ is obtained from $\mathcal{X}_n$ by adding new $2$-cells whose bottom paths are pairwise disjoint, except possibly at their endpoints; and such that, for each new cell, its top path lies in $\mathcal{X}_n$ but its bottom path must be a simple arc meeting $\mathcal{X}_n$ only at its endpoints. 
\end{itemize}
\end{definition}

\noindent
Among directed $2$-complexes, there is a natural notion of \emph{covers} and every directed $2$-complexes (with a distinguished $1$-path) admits a (unique) \emph{universal cover}. This cover is a rooted $2$-tree and it is constructed as follows. Let $\mathcal{X}$ be a directed $2$-complex with a distinguished $1$-path $\alpha$. 
\begin{itemize}
	\item Let $\mathcal{X}_0$ denote the directed $2$-complex reduced to a single $1$-path, which we endow with a map $\phi_0 : \mathcal{X}_0 \to \mathcal{X}$ that sends isomorphically $\mathcal{X}_0$ to $\alpha$. 
	\item Assume that $\mathcal{X}_n$ and $\phi_n : \mathcal{X}_n \to \mathcal{X}$ are constructed for some $n \geq 0$. If there exists a $1$-path $\xi \subset \mathcal{X}_n$ and a $2$-cell $\pi \subset \mathcal{X}$ such that the top path of $\pi$ is $\phi_n(\xi)$ and such that no $2$-cell of $\mathcal{X}_n$ is sent to $\pi$ under $\phi_n$, then we constructed $\mathcal{X}_{n+1}$ by gluing a new $2$-cell along $\xi$ and extend $\phi_n$ by sending this new cell to $\pi$. 
\end{itemize}
Possibly after infinitely many steps, one gets a rooted $2$-tree $\widetilde{\mathcal{X}}$ and a (covering) map $\phi : \widetilde{\mathcal{X}} \to \mathcal{X}$. We refer to \cite{MR2193190} for more details. 

\begin{prop}
Let $\mathcal{P}= \langle \Sigma \mid \mathcal{R} \rangle$ be a semigroup presentation and $w \in \Sigma^+$ a baseword. Let $X(\mathcal{P})$ denote the directed $2$-complexes associated to $\mathcal{P}$ with $w$ as its distinguished $1$-path. The cube complex $M_\square(\mathcal{P},w)$ is naturally isomorphic to $\mathrm{Sq}\left( \widetilde{\mathcal{X}(\mathcal{P})} \right)$. 
\end{prop}

\noindent
A possible description of the universal cover $\widetilde{\mathcal{X}(\mathcal{P})}$ is the following. Take all the reduced diagrams over $\mathcal{P}$ with top labels $w$ and glue them along their top paths $\tau$. Each time two diagrams $\Delta_1$ and $\Delta_2$ have a common prefix $\Delta$ (i.e.\ each time $\Delta_1$ and $\Delta_2$ decompose as concatenations $\Delta_1=\Delta\circ \Gamma_1$ and $\Delta_2= \Delta \circ \Gamma_2$), identify the two copies of $\Delta$. The directed $2$-complex one gets is precisely $\widetilde{\mathcal{X}(\mathcal{P})}$, with $\tau$ as its root. Observe that $\tau$ and any $1$-path $\alpha$ in $\widetilde{\mathcal{X}(\mathcal{P})}$ delimits a unique diagram $\Omega(\alpha) \subset \widetilde{\mathcal{X}(\mathcal{P})}$. The map $\alpha \mapsto \Omega(\alpha)$ yields the isomorphism $\mathrm{Sq}\left( \widetilde{\mathcal{X}(\mathcal{P})} \right) \to M_\square(\mathcal{P},w)$ from the proposition.

\subsection{Morse theory and finiteness properties}\label{section:Morse}

\noindent
Recall that, given an $n \geq 0$, a group is \emph{of type $F_n$} if it admits a classifying space with only finitely many $k$-cells for $k \leq n$. It is \emph{of type $F_\infty$} if is of type $F_n$ for every $n \geq 0$; or equivalently if it admits a classifying space with only finitely many cells in each dimension. A group is \emph{of type $F$} if it admits a classifying space with only finitely many cells. It is worth noticing that a group is of type $F_1$ (resp. of type $F_2$) if and only if it is finitely generated (resp. finitely presented). 

\medskip \noindent
Since universal covers of Squier cube complexes are contractible, as cube-completions of median graphs, one gets explicit classifying spaces for diagram groups, which allows us to deduce various finiteness properties. As an immediate application:

\begin{prop}\label{prop:EasyF}
Let $\mathcal{P}= \langle \Sigma \mid \mathcal{R} \rangle$ be a semigroup presentation and $w \in \Sigma^+$ a baseword. If only finitely many words in $\Sigma^+$ equal $w$ modulo $\mathcal{P}$, then the diagram group $D(\mathcal{P},w)$ is of type $F$. 
\end{prop}

\begin{proof}
The Squier cube complex $S(\mathcal{P},w)$ defines a classifying space of $D(\mathcal{P},w)$ with only finitely many cells.
\end{proof}

\noindent
For instance, the proposition applies to all the diagram groups associated to the semigroup presentations
$$\langle x_1, \ldots, x_n \mid x_ix_j = x_jx_i \ (1 \leq i < j \leq n) \rangle, \ n\geq 1,$$
including to the pure planar braid groups given by Proposition~\ref{prop:PlanarBraid}. It is worth noticing that not every diagram group satisfies good finiteness properties. For instance, the groups $\mathbb{Z} \wr \mathbb{Z}$ and $\mathbb{Z} \bullet \mathbb{Z}$ given in Section~\ref{section:Examples} are finitely generated by not finitely presented. As a more elementary example, it is not difficult to see that the diagram group $D(\mathcal{P},x)$, where
$$\mathcal{P} = \langle x,a,b,c \mid x=xa, a=b,b=c,c=a \rangle,$$
is a free abelian group of infinite rank. Of course, this group is not finitely generated. 

\medskip \noindent
Even though the Squier cube complex contains infinitely many cells in each dimension, it is possible to deduce good finiteness properties of the corresponding diagram groups by extracting from the cube complex a compact core. A general result illustrating this assertion is the following:

\begin{thm}[\cite{MR1978047}]\label{thm:Farley}
Diagram groups given by finite presentations of finite semigroups are of type $F_\infty$. 
\end{thm}

\noindent
Let us illustrate the proof in the particular case of the diagram group $D(\mathcal{P},x)$ where $\mathcal{P}:= \langle x \mid x=x^2 \rangle$. So we want to show that $D(\mathcal{P},x)$ is of type $F_\infty$ thanks to its action on the cube-completion $M_\square(\mathcal{P},x)$ of $M(\mathcal{P},x)$. We cannot deduce what we want to prove directly from $M_\square(\mathcal{P},x)$, because there are infinitely many orbits of cells. The trick is to truncate the complex at some height.

\medskip \noindent
More formally, given a polyhedral complex $X$, a map $f : X \to \mathbb{R}$ is a \emph{Morse function} if
\begin{itemize}
	\item $f$ is affine on each polyhedron;
	\item $f$ is never constant on a cell of positive dimension;
	\item and $f(X^{(0)})$ is discrete in $\mathbb{R}$.
\end{itemize}
Typically, one should think of a Morse function as a height function. 

\medskip \noindent
In our cube complex $M_\square(\mathcal{P},x)$, we have a natural Morse function at our disposal: it suffices to extend the map $M(\mathcal{P},x) \to \mathbb{N}$ that sends a vertex represented by an $(x,x^n)$-diagram to $n$. This Morse function $\mathfrak{h} : M_\square(\mathcal{P},x) \to \mathbb{R}_+$ is $D(\mathcal{P},x)$-equivariant, so $D(\mathcal{P},x)$ preserves the sublevels $\mathfrak{h}^{-1}([0,k])$, $k \geq 0$. Moreover, because $\mathcal{P}$ is a finite presentation, all these actions are cocompact. Thus, we would like to replace the Squier cube complex with the quotient under $D(\mathcal{P},x)$ of a well-chosen sublevel. Unfortunately, there is not guarantee that such a sublevel is contractible, contrary to the Squier complex, so it is not clear in general that our new complex will be a classifying space. 

\medskip \noindent
This is where Morse theory enters the game, whose philosophy claims that understanding the topology of a polyhedron complex is equivalent to understanding the topology of the (sub)levels of a Morse function. 

\begin{prop}[\cite{MR1465330}]
Let $X$ be an $n$-connected polyhedron complex and $f: X \to \mathbb{R}$ a Morse function. Assume that there exists some $h_0 \in \mathbb{R}$ such that the descending link of every vertex of height $\geq h_0$ is $n$-connected. Then $f^{-1}( (- \infty,h_0])$ is $n$-connected. 
\end{prop}

\noindent
Here, the \emph{descending link} of a vertex refers to the subcomplex of the link given by the edges connecting our vertex to vertices of lower heights. The proposition can be easily understood for $n=1$. So let $\gamma$ be a loop in our sublevel $f^{-1}((-\infty,h_0])$ and let us show that $\gamma$ is homotopically trivial in $f^{-1}((- \infty,h_0])$. Because $X$ is simply connected by assumption, there must be a disc $\mathbb{D}^2 \to X$ bounded by $\gamma$. However, this disc $D$ may not lie in our sublevel. Nevertheless, pick a vertex $x$ of maximal height $\geq h_0$ on $D$. Because $x$ has maximal height and because Morse functions are not constant on cells of positive dimensions, we can draw on $D$ a small loop $\sigma$ around $x$ lying in the descending link $\mathrm{link}_\downarrow(x)$ of $x$. Because $\mathrm{link}_\downarrow(x)$ is simply connected by assumption, there is a small disc $D'$ in $\mathrm{link}_\downarrow(x)$ bounded by $\sigma$. Replacing with $D'$ the subdisc of $D$ delimited by $\sigma$ allows us to push down $D$. By iterating the process, we eventually get a disc bounded by $\gamma$ inside $f^{-1}((-\infty,h_0])$, proving that $\gamma$ is homotopically trivial in $f^{-1}((-\infty,h_0])$. Thus, $f^{|1}((-\infty,h_0])$ is indeed simply connected.

\begin{center}
\includegraphics[trim=12cm 7cm 22cm 4cm,clip,width=0.32\linewidth]{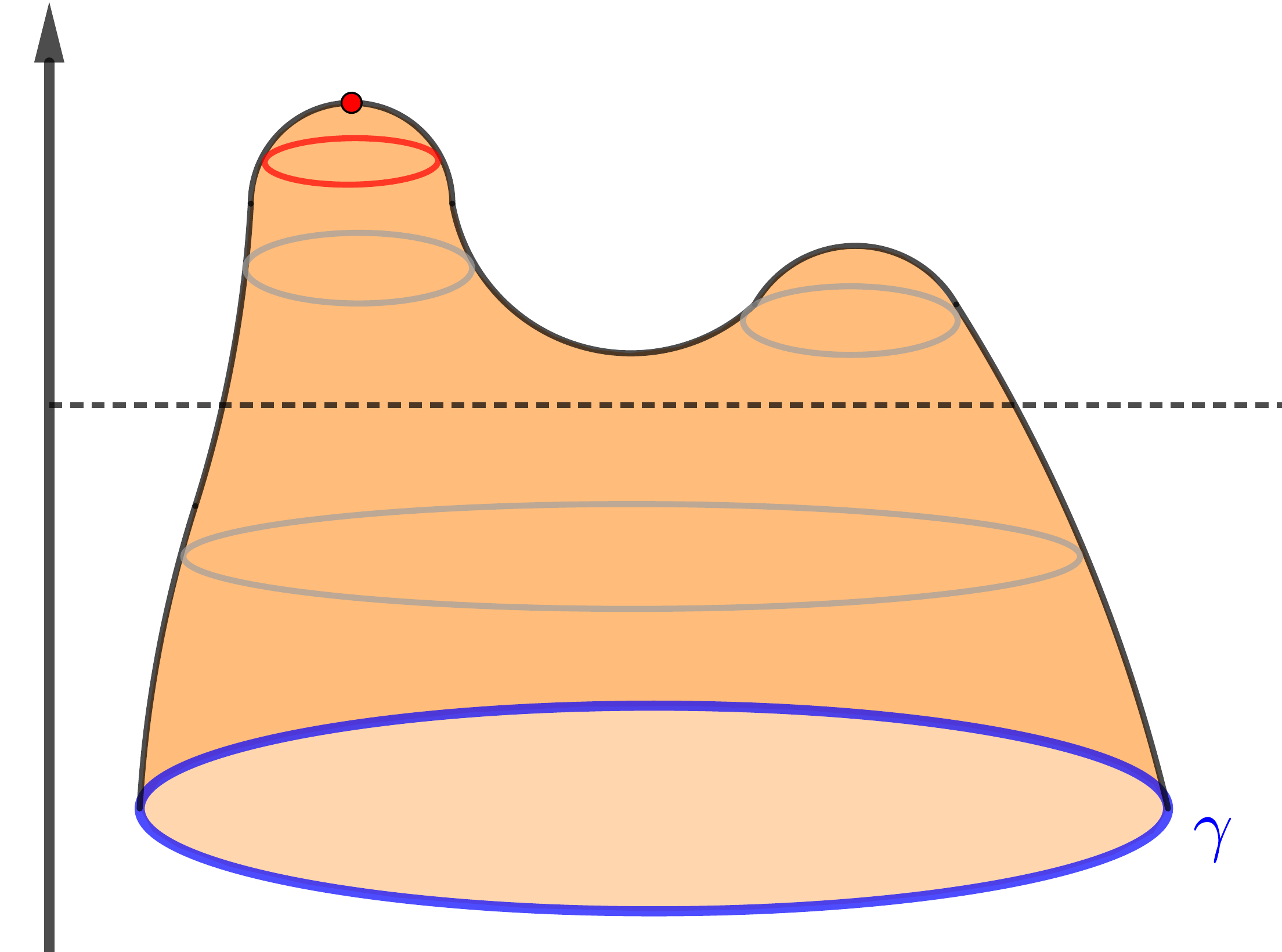}
\includegraphics[trim=12cm 7cm 22cm 4cm,clip,width=0.32\linewidth]{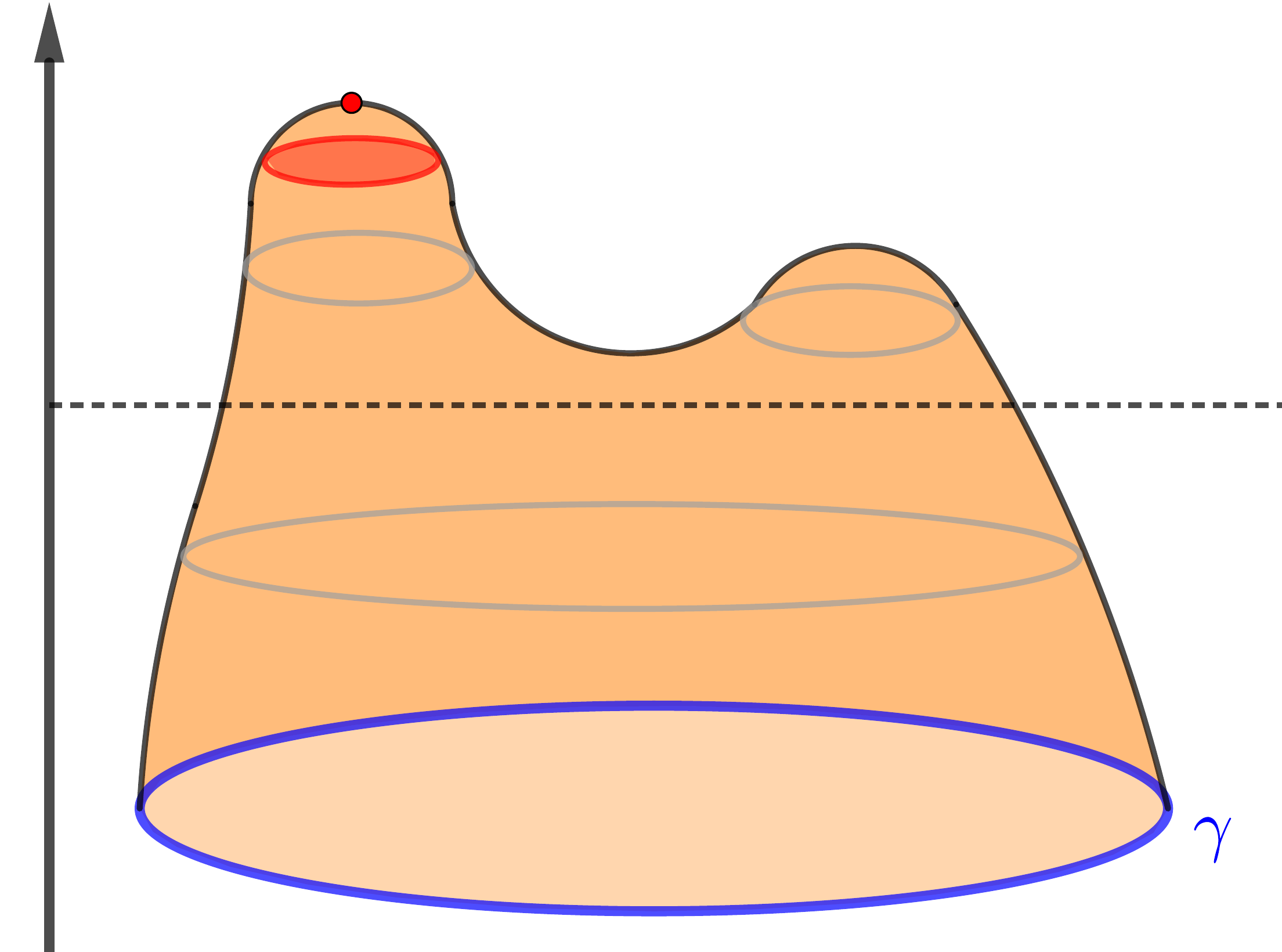}
\includegraphics[trim=12cm 7cm 22cm 4cm,clip,width=0.32\linewidth]{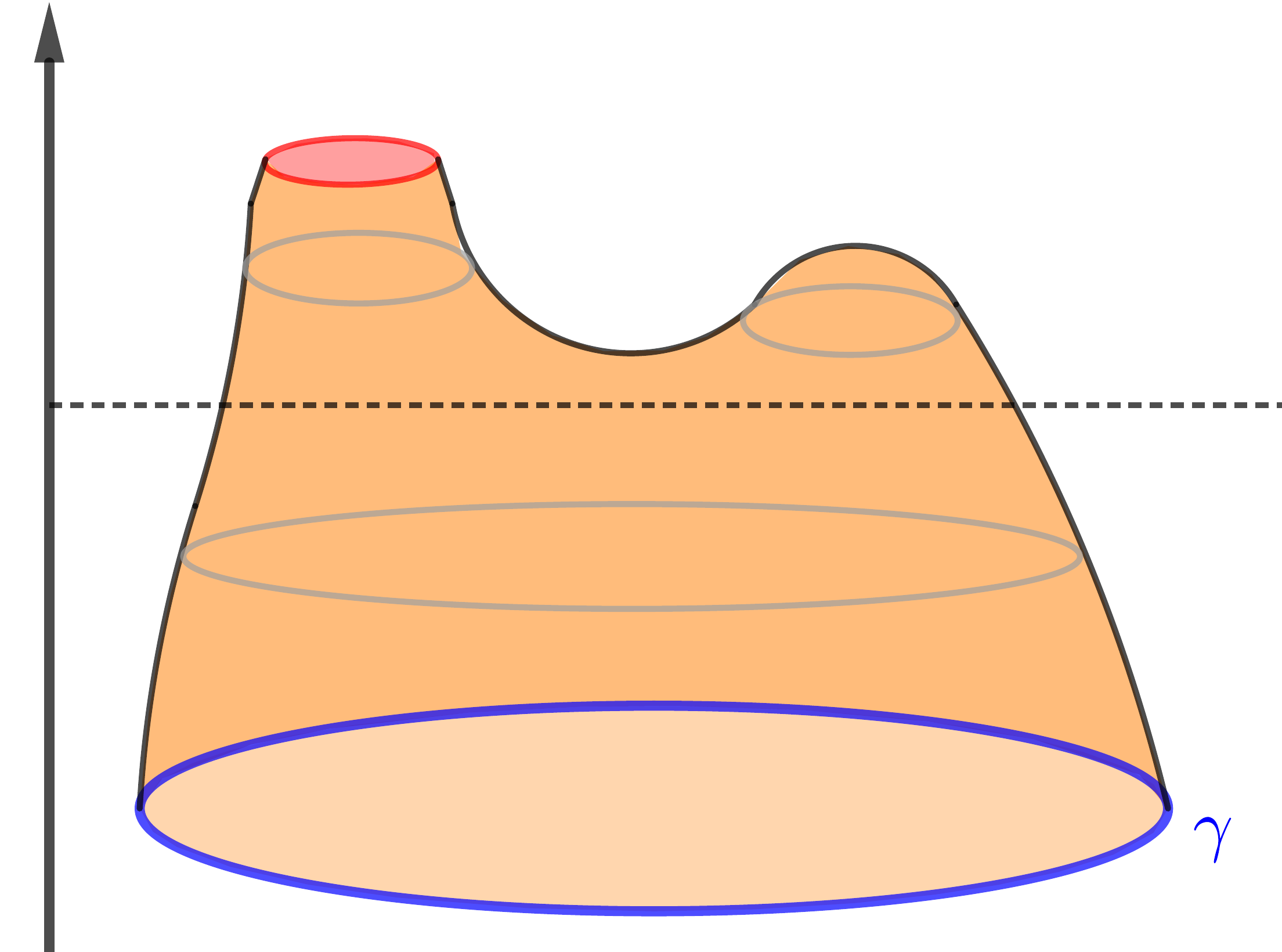}
\end{center}

\noindent
Thus, in order to show that our diagram group $D(\mathcal{P},x)$ is of type $F_n$ for every $n \geq 1$, it suffices to verify that descending links of vertices in $M_\square(\mathcal{P},x)$ gets more and more connected when the height increases. The is usually the delicate point: we need to describe the descending links of vertices and to prove that they are highly connected.

\medskip \noindent
Let $\Delta$ be an $(x,x^k)$-diagram representing a vertex $x$ of height $k$ in $M_\square(\mathcal{P},x)$. In order to connect $x$ to a vertex of lower height, we need to glue on $\mathrm{bot}(\Delta)$ a negative cell (i.e.\ a $2$-cell labelled by the relation $x^2=x$); and a collection of such gluings spans a simplex in the descending link of $x$ exactly when they are performed on pairwise disjoint subsegments of $\mathrm{bot}(\Delta)$. Thus, the descending link of a vertex of height $k$ in $M_\square(\mathcal{P},x)$ is isomorphic to the simplicial complex $\mathfrak{I}_k$
\begin{itemize}
	\item whose vertices are the subintervals $[i,i+1] \subset [1,k]$, $i \in \mathbb{N}$;
	\item and whose simplices are collections of pairwise disjoint subintervals.
\end{itemize}
One can naturally reconstruct $\mathfrak{I}_{k+3}$ from a cone over $\mathfrak{I}_{k+1}$ (whose apex corresponds to the vertex $[k+2,k+3]$ of $\mathfrak{I}_{k+3}$) by gluing a cone over $\mathfrak{I}_k \subset \mathfrak{I}_{k+1}$ (whose apex corresponds to the vertex $[k+1,k+2]$ of $\mathfrak{I}_{k+3}$). This observation allows us to prove by induction that $\mathfrak{I}_k$ is $n$-connected for every $k \geq 3(n+2)$. 

\medskip \noindent
Thus, given an $n \geq 0$, we know that $D(\mathcal{P},x)$ acts freely, properly, and cocompactly on the sublevel $\mathfrak{h}^{-1}([1,3(n+2)])$ in $M_\square(\mathcal{P},x)$, which is $n$-connected. We conclude that $D(\mathcal{P},x)$ is of type $F_n$, as desired. The argument easily generalises to the semigroup presentation $\langle x \mid x=x^r \rangle$, $r \geq 2$. We leave the details as an exercise for the interested reader.

\medskip \noindent
We conclude this section by mentioning the following consequence of Theorem~\ref{thm:Farley}:

\begin{cor}
Every diagram group over a finite semigroup presentation embeds into a diagram group of type $F_\infty$. 
\end{cor}

\begin{proof}
Let $\mathcal{P}= \langle \Sigma \mid \mathcal{R} \rangle$ be a finite semigroup presentation and $w \in \Sigma^+$ a baseword. Set
$$\mathcal{Q}:= \left\langle \Sigma \sqcup \{ x\} \mid \mathcal{R} \sqcup \{ x=x^2, a=b \ (a,b \in \Sigma) \} \right\rangle.$$
Every (reduced) diagram over $\mathcal{P}$ also defines a (reduced) diagram over $\mathcal{Q}$, so there is an obvious injective morphism $D(\mathcal{P},w) \hookrightarrow D(\mathcal{Q},w)$. Clearly, the semigroup defined by $\mathcal{Q}$ is finite, so Theorem~\ref{thm:Farley} implies that $D(\mathcal{Q},w)$ is of type $F_\infty$. 
\end{proof}

\subsection{Hilbert space compression}\label{section:Hilbert}

\noindent
Roughly speaking, the \emph{Hilbert space compression} of a metric space is a number between zero and one which quantifies the compatibility between the geometry of the space with the geometric of Hilbert spaces. The smallest the compression is, the highest the distortion of an embedding into a Hilbert space must be. 

\begin{definition}
Let $f : X \to Y$ be a Lipschitz map between two metric spaces. The \emph{compression} of $f$ is the least $\alpha>0$ for which there exists a constant $C>0$ such that
$$C \cdot d(x,y)^\alpha \leq d(f(x),f(y)) \text{ for all } x,y \in X.$$
The \emph{Hilbert space compression} of $X$ is the least compression of a Lipschitz map from $X$ to a Hilbert space. By extension, the Hilbert space compression of a finitely generated group is the \emph{Hilbert space compression} of its Cayley graphs (constructed from finite generating sets).
\end{definition}

\noindent
Interestingly, median graphs always admit good Lipschitz embeddings into Hilbert spaces with compression bounded below. 

\begin{thm}[\cite{Haagerup, MR2132393}]\label{thm:MedianComp}
The Hilbert space compression of a median graph is always $\geq 1/2$. Moreover, it is equal to $1$ when the cubical dimension is finite. 
\end{thm}

\noindent
Let us justify the first assertion. Given a median graph $X$, let $\mathcal{H}$ denote the Hilbert space given by the $\ell^2$-summable maps $\{ \text{hyperplanes of } X\} \to \mathbb{R}$ endowed with the $\ell^2$-norm. For every hyperplane $J$ in $X$, we denote by $\delta_J$ the map that sends $J$ to $1$ and all the other hyperplanes to $0$. Given a basepoint $o \in X$, we claim that the map $\Phi : X \to \mathcal{H}$ defined by
$$x \mapsto \sum\limits_{J \text{ separates $o$ and $x$}} \delta_J$$
induces a Lipschitz embedding with compression $\geq 1/2$. Indeed, for all vertices $x,y \in X$, if we denote by $m$ the median point of $\{x,y,o\}$ and by $\mathcal{W}(\cdot | \cdot)$ the set of the hyperplanes separating two given vertices, we have
$$\begin{array}{lcl} \Phi(x)- \Phi(y) & = & \displaystyle \sum\limits_{J \in \mathcal{W}(o|m)} \delta_J + \sum\limits_{J \in \mathcal{W}(m|x)} \delta_J - \sum\limits_{J \in \mathcal{W}(o|m)} \delta_J - \sum\limits_{J \in \mathcal{W}(m|y)} \delta_J \\ \\ & = & \displaystyle \sum\limits_{J \in \mathcal{W}(m|x)} \delta_J - \sum\limits_{J \in \mathcal{W}(m|y)} \delta_J \end{array}$$
hence
$$ \| \Phi(x)- \Phi(y)\|_2^2 = \# \mathcal{W}(m|x) + \# \mathcal{W}(m|y) = \# \mathcal{W}(x|y) = d(x,y).$$
This concludes the proof of our claim. 

\medskip \noindent
As an immediate consequence of Theorem~\ref{thm:MedianComp}, it follows that the Hilbert space compression of a finitely generated diagram group $D(\mathcal{P},w)$ is $\geq 1/2$ if its orbits in the median graph $M(\mathcal{P},w)$ are quasi-isometrically embedded; and is even equal to $1$ if in addition the cubical dimension of $M(\mathcal{P},w)$ is finite. For instance:

\begin{prop}\label{prop:EasyB}
Let $\mathcal{P}= \langle \Sigma \mid \mathcal{R} \rangle$ be a semigroup presentation and $w \in \Sigma^+$ a baseword. If only finitely many words in $\Sigma^+$ are equal to $w$ modulo $\mathcal{P}$, then the Hilbert space compression of $D(\mathcal{P},w)$ is equal to $1$.
\end{prop}

\begin{proof}
Our assumptions imply that $D(\mathcal{P},w)$ acts on $M(\mathcal{P},w)$ properly and cocompactly. Therefore, the Hilbert space compression of $D(\mathcal{P},w)$ coincides with the Hilbert space compression of $M(\mathcal{P},w)$, which is equal to $1$ according to Theorem~\ref{thm:MedianComp}. 
\end{proof}

\noindent
As a consequence of the description of the graph metric of $M(\mathcal{P},w)$ given in Section~\ref{section:MedianDiag}, it is clear that the orbits of a finitely generated diagram group $D(\mathcal{P},w)$ are quasi-isometrically embedded if and only if the following property is satisfied:

\begin{definition}
Let $\mathcal{P}= \langle \Sigma \mid \mathcal{R} \rangle$ be a semigroup presentation and $w \in \Sigma^+$ a baseword. The diagram group $D(\mathcal{P},w)$ satisfies the \emph{Burillo property} if it is finitely generated and if every word-metric given by a finite generating set is biLipischtz equivalent to the \emph{diagram metric} $\#(\cdot)$, i.e.\ the map that sends each element of $D(\mathcal{P},w)$ to the number of $2$-cell in a reduced diagram representing it. 
\end{definition}

\noindent
Thus, we deduce the following statement:

\begin{prop}[\cite{MR2271228}]\label{prop:BHilbert}
A diagram group satisfying the Burillo property has Hilbert space compression $\geq 1/2$. 
\end{prop}

\noindent
Examples of diagram groups satisfying the Burillo property include the pure planar braid groups and various right-angled Artin groups (according to Proposition~\ref{prop:EasyB}), Thompson's group $F$ \cite{MR1670622}, the wreath product $\mathbb{Z} \wr \mathbb{Z}$ \cite{MR2271228}, and $\mathbb{Z} \bullet \mathbb{Z}$ \cite{MR3868219}. (All these examples are thought of as diagram groups by taking the semigroup presentations given in Section~\ref{section:Examples}.) So far, no example of a finitely generated diagram group that does not satisfy the Burillo property appears in the literature.

\medskip \noindent
Proposition~\ref{prop:BHilbert} only gives lower bounds on Hilbert space compression. Computing the exact value of the Hilbert space compression may be subtle and goes beyond the scope of this survey. Nevertheless, let us mention that Thompson's group $F$ has Hilbert space compression $1/2$ \cite{MR2271228} and that the lamplighter group $\mathbb{Z} \wr \mathbb{Z}$ has Hilbert space compression $2/3$ \cite{MR2439428}. 

\medskip \noindent
The proof of Proposition~\ref{prop:BHilbert} given in \cite{MR2271228} does not use median geometry, but it turns out to be the exact translation of the argument presented above through the correspondance between median graphs and rooted $2$-trees described in Section~\ref{section:MedianDiag}. Let us describe the argument as presented in \cite{MR2271228}.

\medskip \noindent
Let $\mathcal{P}= \langle \Sigma \mid \mathcal{R} \rangle$ be a semigroup presentation and $w \in \Sigma^+$ a baseword. We glue together all the reduced diagrams over $\mathcal{P}$ with top label $w$ along their top paths $\tau$; and, for any two diagrams $\Delta_1,\Delta_2$, if $\Delta$ is a common prefix of $\Delta_1,\Delta_2$, then we identify the two copies of $\Delta$ in $\Delta_1$ and $\Delta_2$. Let $T(\mathcal{P},w)$ denote the resulting object. Let $\mathcal{H}$ denote the Hilbert space given by the formal sums of $2$-cells in $T(\mathcal{P},w)$ endowed with the $\ell^2$-norm. For every $g \in D(\mathcal{P},w)$, there is a unique copy of a reduced diagram representing $g$ having $\tau$ as its top path. Let $\Xi(g)$ denote the formal sum of the $2$-cells contained in this copy. Then
$$\| \Xi(g) - \Xi(h) \|_2^2 = \#(g^{-1}h) \text{ for all } g,h \in D(\mathcal{P},w).$$
Thus, if $D(\mathcal{P},w)$ satisfies the Burillo property, then $g \mapsto \Xi(g)$ defines a Lipschitz map $D(\mathcal{P},w) \to \mathcal{H}$ with compression $1/2$. 

\medskip \noindent
An explicit bijection between the $2$-cells in $T(\mathcal{P},w)$ and the hyperplanes in $M(\mathcal{P},w)$ can be described as follows. A diagram $\Delta$ over $\mathcal{P}$ is \emph{minimal} if it contains a single $2$-cell having its bottom path in the bottom path of $\Delta$. If $\Delta^-$ denotes the diagram obtained from $\Delta$ by removing this $2$-cell, then $[\Delta,\Delta^-]$ defines an edge in $M(\mathcal{P})$, and we denote by $J(\Delta)$ the hyperplane containing this edge. It turns out that the map $\Delta \mapsto J(\Delta)$ induces a bijection from the reduced minimal diagrams over $\mathcal{P}$ to the hyperplanes in $M(\mathcal{P})$ \cite[Proposition~2]{HypDiag}. For every $2$-cell $\pi$ in $T(\mathcal{P},w)$, let $\Delta(\pi)$ denote the smallest diagram in $T(\mathcal{P},w)$ containing $\pi$ and having $\tau$ as its top path. The minimality assures that $\Delta(\pi)$ is minimal, and the map $\pi \mapsto J(\Delta(\pi))$ yields the desired bijection from the $2$-cells of $T(\mathcal{P},w)$ to the hyperplanes of $M(\mathcal{P},w)$.

\begin{remark}
In this section, we have focused on Hilbert space compressions of groups, but one can define the \emph{equivariant Hilbert space compression} of a finitely generated group $G$ as the least compression of a Lipschitz map from $G$ (endowed with the word-length given by a finite generating set) to a Hilbert space on which $G$ acts by affine isometries. Everything that has been said about Hilbert space compressions in this section also holds for equivariant compressions. 
\end{remark}

\subsection{Hyperplanes in Squier's cube complex}\label{section:Hyperplanes}

\noindent
In this section, we focus on the combinatorics of hyperplanes in Squier cube complexes. In particular, this will allow us to prove that some diagram groups are linear over $\mathbb{Z}$, have finite asymptotic dimension, or satisfy some Tits alternative depending on their underlying semigroup presentations. We begin by recalling some terminology and results from \cite{MR2377497}. 

\begin{definition}
Let $X$ be a nonpositively curved cube complex. 
\begin{itemize}
	\item A hyperplane is \emph{one-sided} if its two orientations coincide. Otherwise, it is \emph{two-sided}.
	\item A hyperplane $J$ \emph{self-intersects} if there exist a vertex $o \in X$ and two neighbours $a,b \in X$ such that the edges $[o,a], [o,b]$ span a square and are both transverse to $J$.
	\item A hyperplane $J$ \emph{self-osculates} if there exist a vertex $o \in X$ and two neighbours $a,b \in X$ such that the edges $[o,a], [o,b]$ do not span a square and are both transverse to $J$.
	\item Two hyperplanes \emph{inter-osculates} if they are both transverse and tangent. 
\end{itemize}
The cube complex $X$ is \emph{conspicial}\footnote{This property is called \emph{$A$-special} in \cite{MR2377497}, shortened as \emph{special} in the recent literature. In order to minimise conflicts with the every-day language, we propose to use the Latin word \emph{conspicial} which is close to \emph{special} both semantically and phonetically} if all its hyperplanes are two-sided and if it does not contain self-intersections, self-osculations, nor inter-osculations.
\end{definition}

\noindent
Let us motivate this definition. Given a nonpositively curved cube complex $X$, let $\Delta X$ denote its \emph{crossing graph}, i.e.\ the graph whose vertices are the hyperplanes of $X$ and whose edges connect two hyperplanes whenever they are transverse. Fix a basepoint $o \in X$. Also, we assume that the hyperplanes of $X$ are two-sided and fix an orientation on the hyperplanes of $X$. A loop in $X^{(1)}$ based at $o$ can be thought of as a sequence of oriented edges, to which we can associate a formal word written over $\{ \text{hyperplanes of } X\}^{\pm 1}$ (an oriented edge $e$ being sent to the hyperplane it crosses or its inverse depending on whether the orientations of $e$ and the hyperplane agree). This correspondence induces a well-defined morphism
$$\mathfrak{A}_X : \pi_1(X,o) \to A(\Delta X)$$
to the right-angled Artin group $A(\Delta X)$. (See Section~\ref{section:Examples} for a definition a right-angled Artin groups.) In full generality, we do not know anything about the morphism $\mathfrak{A}_X$. It may be even trivial. However, it turns out to be injective as soon as $X$ is conspicial \cite[Theorem~4.2]{MR2377497}. Thus, fundamental groups of conspicial cube complexes inherit all the properties satisfied by right-angled Artin groups and stable under taking subgroups. This includes for instance:
\begin{itemize}
	\item being linear over $\mathbb{Z}$ (and, as a consequence, being residually finite);
	\item satisfying a strong Tits alternative, namely every subgroup either is free abelian or contains a non-abelian free subgroup.
\end{itemize}
If moreover the conspicial cube complex is compact, then its fundamental group is nicely embedded into the corresponding right-angled Artin groups: it defines a \emph{convex-cocompact} subgroup. This allows us to strengthen the residual finiteness previously mentioned and to deduce that convex-cocompact subgroups in our fundamental group are \emph{separable} \cite[Corollary~7.9]{MR2377497}. 

\medskip \noindent
A natural question to ask now is: when are Squier cube complexes conspicial? 

\begin{prop}[\cite{MR3868219}]\label{prop:Special}
Let $\mathcal{P}= \langle \Sigma \mid \mathcal{R} \rangle$ be a semigroup presentation and $w \in \Sigma^+$ a baseword. The Squier cube complex $S^+(\mathcal{P},w)$ is conspicial if and only if the following conditions are satisfied:
\begin{itemize}
	\item there are no words $a,b,p \in \Sigma^+$ such that $w=ab$, $a=ap$, and $b=pb$ modulo $\mathcal{P}$ with $[p]_\mathcal{P} \neq \{p\}$;
	\item there are no words $a,u,v,w,b,p,q, \xi \in \Sigma^+$ such that $w=auvwb$, $au=au(v\xi)$, and $wb=(\xi v)wb$ modulo $\mathcal{P}$ with $uv=p, vw=q$ in $\mathcal{R}$. 
\end{itemize}
\end{prop}

\noindent
Here, $[p]_\mathcal{P}$ denotes the set of all the words in $\Sigma^+$ equal to $p$ modulo $\mathcal{P}$. It is worth noticing that, if the condition from the first item fails, then $\{ ap^nb \mid n \geq 0 \} \subset [w]_\mathcal{P}$; and, if the condition from the second item fails, then $\{ au (v\xi)^m v (\xi v)^n wb \mid n,m \geq 0 \} \subset [w]_\mathcal{P}$. Consequently, if we know that $[w]_\mathcal{P}$ is finite, then it follows immediately from the proposition that our Squier cube complex is conspicial. 

\begin{cor}\label{cor:Special}
Let $\mathcal{P}= \langle \Sigma \mid \mathcal{R} \rangle$ be a semigroup presentation and $w \in \Sigma^+$ a baseword. If $[w]_\mathcal{P}$ is finite, then $S^+(\mathcal{P},w)$ is conspicial. 
\end{cor}

\noindent
Corollary~\ref{cor:Special} applies to the right-angled Artin groups and pure planar braid groups from Section~\ref{section:Examples}. It does not apply to the group $\mathbb{Z} \bullet \mathbb{Z}$, but Proposition~\ref{prop:Special} holds anyway. It is worth mentioning that not all diagram groups can be described as fundamental groups of conspicial cube complexes. Thompson's group $F$ is such an example since it is not even residually finite. 

\medskip \noindent
In Corollary~\ref{cor:Special}, the Squier cube complex is conspicial but also compact. Therefore, this allows us to deduce that convex-cocompact subgroups in the corresponding diagram group are separable. One instance of this assertion is given by:

\begin{cor}[\cite{MR3868219}]
Let $\mathcal{P}= \langle \Sigma \mid \mathcal{R} \rangle$ be a semigroup presentation and $w \in \Sigma^+$ a baseword with $[w]_\mathcal{P}$ finite. Fix a $(w,x_1u_1 \cdots x_nu_nx_{n+1})$-diagram $\Delta$. The morphism
$$\left\{ \begin{array}{ccc} D(\mathcal{P},u_1) \times \cdots \times D(\mathcal{P},u_n) & \to & D(\mathcal{P},w) \\ (\Gamma_1, \ldots, \Gamma_n) & \mapsto & \Delta \cdot (\Gamma_1+ \cdots + \Gamma_n) \cdot \Delta^{-1} \end{array} \right.$$
is injective and its image is a separable subgroup of $D(\mathcal{P},w)$. 
\end{cor}

\medskip \noindent
Let us describe explicitly the morphism $\mathfrak{A}_X$ above in the specific case of our Squier cube complexes. So we fix a semigroup presentation $\mathcal{P}= \langle \Sigma \mid \mathcal{R} \rangle$ and a baseword $w \in \Sigma^+$. The oriented edges of $S^+(\mathcal{P},w)$ can be written as $(a, u \to v, b)$ when they connect $aub$ to $avb$; we denote by $[a, u\to v,b]$ the corresponding (oriented) hyperplane. It can be shown that two oriented hyperplanes $[a, u \to v, b]$ and $[c, p \to q, d]$
\begin{itemize}
	\item coincide if and only if $a=c$, $b=d$ modulo $\mathcal{P}$ and $u=p$, $v=q$ in $\Sigma^+$ \cite[Lemma~3.1]{MR3868219};
	\item are transverse if and only if there exists some $y \in \Sigma^+$ such that either $c=auy$ and $b=ypd$ modulo $\mathcal{P}$ or $d=yub$ and $a=cpy$ modulo $\mathcal{P}$ \cite[Lemma~3.2]{MR3868219}. 
\end{itemize}
Let $A(\mathcal{P},w)$ be the right-angled Artin groups given by the crossing graph of $S^+(\mathcal{P},w)$. The generators of $A(\mathcal{P},w)$ are the oriented hyperplanes $[a, u \to v,b]$ where $u=v$ belongs to $\mathcal{R}$. The map that sends each oriented edge $(a, u \to v,b)$ to the element $[a,u \to v,b]$ if $u=v$ belongs to $\mathcal{R}$ and to the element $[a,v \to u,b]^{-1}$ if $v=u$ belongs to $\mathcal{R}$ induces a morphism
$$\mathfrak{A}(\mathcal{P},w) : D(\mathcal{P},w) \to A(\mathcal{P},w).$$
It is injective as soon as the conditions of Proposition~\ref{prop:Special} are satisfied. This happens for instance when $[w]_\mathcal{P}$ is finite. 

\begin{ex}\label{ZbulletZspecial}
Let us consider the following semigroup presentation:
$$\mathcal{P} := \left\langle \begin{array}{l} a_1,a_2,a_3 \\ b_1,b_2,b_3 \end{array}, p \left| \begin{array}{l} a_1=a_2, a_2=a_3, a_3=a_1 \\ b_1=b_2, b_2=b_3, b_3=b_1 \end{array}, \begin{array}{l} a_1=a_1p \\ b_1=pb_1 \end{array} \right. \right\rangle.$$
The diagram group $D(\mathcal{P},a_1b_1)$ is the group $\mathbb{Z} \bullet \mathbb{Z}$ described in Section~\ref{section:Examples}. According to Proposition~\ref{prop:Special}, the Squier cube complex $S^+(\mathcal{P},a_1b_1)$ is conspicial, so our morphism $\mathfrak{A}(\mathcal{P},w_1b_1) \to A(\mathcal{P},a_1b_1)$ is injective.

\medskip \noindent
Using our previous description of hyperplanes, we find that $S^+(\mathcal{P},a_1b_1)$ has eight hyperplanes: $A_i:=[1,a_i \to a_{i+1},b_1]$, $B_i:=[a_1,b_i \to b_{i+1},1]$, $C:= [ 1,a_1 \to a_1p,b_1]$, and $D:=[a_1, b_1 \to pb_1,1]$. Moreover, we find that the crossing graph of $S^+(\mathcal{P},a_1b_1)$ is a complete bipartite graph $K_{4,4}$, where each vertex of $\{ A_1,A_2,A_3,C \}$ is connected by an edge to each vertex of $\{ B_1, B_2, B_3, D \}$. In particular, $A(\mathcal{P},a_1b_1) \simeq \mathbb{F}_4 \times \mathbb{F}_4$. 

\medskip \noindent
Then, by using Theorem~\ref{thm:Presentation} or Proposition~\ref{prop:GraphOfSpaces} below, we find that $\mathbb{Z} \bullet \mathbb{Z}$ is generated by the three diagrams given by Figure~\ref{figure39}.
\begin{figure}[h!]
\begin{center}
\includegraphics[width=0.7\linewidth]{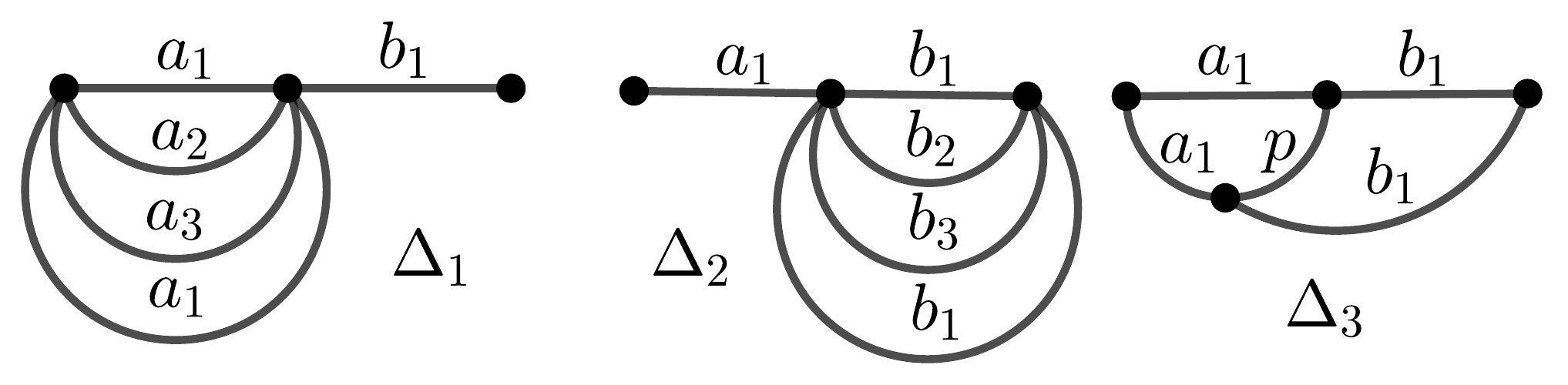}
\caption{Generators of $\mathbb{Z} \bullet \mathbb{Z}$.}
\label{figure39}
\end{center}
\end{figure}

\noindent
For instance, the first diagram corresponds to the loop of edges:
$$(1,a_1 \to a_2, b_1), \ (1,a_2 \to a_3, b_1), \ (1,a_3 \to a_1, b_1),$$
which implies that our morphism $\mathfrak{A}= \mathfrak{A}(\mathcal{P},a_1b_1)$ sends $\Delta_1$ to $A_1A_2A_3$. Similarly, we find that $\mathfrak{A}(\Delta_2)= B_1B_2B_3$ and $\mathfrak{A}(\Delta_3)=CD^{-1}$.

\medskip \noindent
We conclude that the subgroup $\langle A_1A_2A_3, B_1B_2B_3, CD^{-1} \rangle$ of 
\begin{center}
$\mathbb{F}_4 \times \mathbb{F}_4= \langle A_1,A_2,A_3,C \mid \ \rangle \times \langle B_1,B_2,B_3, D \mid \ \rangle$
\end{center}
is isomorphic to $\mathbb{Z} \bullet \mathbb{Z}$. The embedding can be simplified to an embedding into $\mathbb{F}_2 \times \mathbb{F}_2$; we refer to \cite[Example~3.17]{MR3868219} for more details. 
\end{ex}

\noindent
From a more geometric perspective, it is possible to use the structure of hyperplanes in the Squier cube complexes in order to construct quasi-isometric embeddings from diagram groups into product of trees. More precisely:

\begin{prop}\label{prop:ProductTrees}
Let $\mathcal{P}= \langle \Sigma \mid \mathcal{R} \rangle$ be a semigroup presentation and $w \in \Sigma^+$ a baseword. If $D(\mathcal{P},w)$ satisfies the Burillo property and if $S^+(\mathcal{P},w)$ has finite dimension $d$, then $D(\mathcal{P},w)$ quasi-isometrically embeds into a product of $d$ trees. As a consequence, $D(\mathcal{P},w)$ has finite asymptotic dimension $\leq d$ and its Hilbert space compression is $1$.
\end{prop}

\noindent
The last assertion of the proposition follows from the first part but also from \cite{MR2916293} and \cite{MR2160829}. One easily observes that $S^+(\mathcal{P},w)$ has dimension $\leq d$ if and only if $d$ is the smallest integer $k$ for which there exist $u_1, \ldots, u_k \in \Sigma^+$ such that $w=u_1 \cdots u_k$ modulo $\mathcal{P}$ and such that $[u_i]_\mathcal{P} \neq \{u_i\}$ for every $1 \leq i \leq k$. For instance, $S^+(\mathcal{P},w)$ is automatically finite-dimensional when $[w]_\mathcal{P}$ is finite. 

\begin{proof}[Proof of Proposition~\ref{prop:ProductTrees}.]
First of all, we need to introduce some notation and vocabulary. 

\medskip \noindent
Let $J_1$ and $J_2$ be two hyperplanes in  $S^+(\mathcal{P},w)$. If they meet inside a square $(a,u \to v, b,p \to q,c)$ with $[a,u \to v,bpc]=J_1$ and $[aub,p \to q,c]=J_2$, then we write $J_1 \prec J_2$. This relation allows us to define the \emph{rank} of a hyperplane $J$ in $S^+(\mathcal{P},w)$ as
$$\mathrm{rank}(J) := \max \{ n \geq 0 \mid \text{there exist} \ J_1, \ldots, J_n \ \text{with} \ J_1 \prec \cdots \prec J_n \prec J \}.$$
Because any two $\prec$-related hyperplanes must be transverse, the rank of a hyperplane lies between $0$ and $d-1$. By extension, we define the \textit{rank} of a hyperplane in the median graph $M(\mathcal{P},w)$ as the rank of its image by the covering map $M_\square(\mathcal{P},w) \to S^+(\mathcal{P},w)$.

\medskip \noindent
Thus, we get a colouring of the hyperplanes of $M(\mathcal{P},w)$ with $d$ colours given by the rank such that any two hyperplanes of the same colour cannot be transverse. This implies that $M(\mathcal{P},w)$ embeds isometrically into a product of $d$ trees (endowed with the $\ell^1$-metric), which allows us to conclude since $D(\mathcal{P},w)$ embeds isometrically into $M(\mathcal{P},w)$ as a consequence of the Burillo property.

\medskip \noindent
Let us describe how $M(\mathcal{P},w)$ can be embedded into a product of $d$ trees. For every $0 \leq i \leq d-1$, let $T_i$ denote the graph whose vertices are the connected components of the graph obtained by cutting $M(\mathcal{P},w)$ along all the hyperplanes of rank $i$ and whose edges connect two components whenever they are separated by a single hyperplane of rank $i$; because no two hyperplanes of rank $i$ are transverse, $T_i$ is a tree; moreover, there is a natural projection $M(\mathcal{P},w) \to T_i$ sending a vertex to the connected component it belongs to, and the distance in $T_i$ between the projections of two vertices $x,y \in M(\mathcal{P},w)$ coincides with the number of hyperplanes of rank $i$ separating $x$ and $y$. Thus, the product of these projections yields a map $M(\mathcal{P},w) \to T_0 \times \cdots \times T_{d-1}$ such that the distance between the images of two vertices $x,y \in M(\mathcal{P},w)$ coincides with the sum over $i$ of the number of hyperplanes of rank $i$ separating $x$ and $y$, i.e.\ the total number of hyperplanes separating $x$ and $y$, which is exactly the distance between $x$ and $y$ in $M(\mathcal{P},w)$. 
\end{proof}

\noindent
We conclude this section by explaining how one can use hyperplanes in order to decompose Squier cube complexes as graphs of spaces and to compute presentations of diagram groups.  Before stating our decomposition theorem, we need to introduce some vocabulary.

\medskip \noindent
So let $\mathcal{P}= \langle \Sigma \mid \mathcal{R} \rangle$ be a semigroup presentation and $w\in \Sigma^+$ a baseword. A hyperplane $J= [a,u \to v,b]$ in $S^+(\mathcal{P},w)$ is \textit{left} if $D(\mathcal{P},a)= \{1 \}$ but $D(\mathcal{P},au) \neq \{1 \}$. If so, let $p_J,s_J \in \Sigma^+$ be two words and $\ell_J \in \Sigma$ a letter satisfying: $u=p_J \ell_J s_J$ in $\Sigma^+$, $D(\mathcal{P},ap_J)= \{1 \}$, and $D(\mathcal{P},ap_J \ell_J) \neq \{1 \}$. Notice that $p_J$ is just the maximal prefix of $u$ satisfying $D(\mathcal{P},ap_J)= \{1 \}$, so that $p_J,s_J, \ell_J$ are uniquely determined. We define similarly $v=q_Jm_Jr_J$, where $q_J,r_J \in \Sigma^+$ and $m_J \in \Sigma$, so that $D(\mathcal{P},q_J)= \{1 \}$ and $D(\mathcal{P},q_Jm_J) \neq \{1 \}$.

\begin{prop}\label{prop:GraphOfSpaces}
Let $\mathcal{P}= \langle \Sigma \mid \mathcal{R} \rangle$ be a semigroup presentation and $w\in \Sigma^+$ a baseword. Let $\mathfrak{J}$ denote the set of left hyperplanes of $S^+(\mathcal{P},w)$ and let $\mathcal{G}(\mathcal{P},w)$ be the graph of spaces defined by:
\begin{itemize}
	\item the set of vertex-spaces is $$\{ S^+(\mathcal{P},ap_J) \ell_J S^+(\mathcal{P},s_Jb), \ S^+(\mathcal{P},aq_J) m_J S^+(\mathcal{P},r_Jb) \mid J= [a,u \to v,b] \in \mathfrak{J} \};$$
	\item to each left hyperplane $[a, u \to v,b] \in \mathfrak{J}$ is associated the edge-space $S^+(\mathcal{P},a) \times S^+(\mathcal{P},b)$;
	\item the edge-maps are the canonical maps $$S^+(\mathcal{P},a) \times S^+(\mathcal{P},b) \to S^+(\mathcal{P},a)u S^+(\mathcal{P},b)$$ and $$ S^+(\mathcal{P},a) \times S^+(\mathcal{P},b) \to S^+(\mathcal{P},a)v S^+(\mathcal{P},b).$$
\end{itemize}
Then $\mathcal{G}(\mathcal{P},w)$ defines a decomposition of $S^+(\mathcal{P},w)$ as a graph of spaces.
\end{prop}

\noindent
As an illustration of the proposition, let us compute a presentation of some diagram group. 

\begin{ex}\label{ZbulletZgraph}
Let $\mathcal{P}= \left\langle \begin{array}{l} a_1,a_2,a_3 \\ b_1,b_2,b_3 \end{array}, p \left| \begin{array}{l} a_1=a_2, a_2=a_3, a_3=a_1 \\ b_1=b_2, b_2=b_3, b_3=b_1 \end{array}, \begin{array}{l} a_1=a_1p \\ b_1=pb_1 \end{array} \right. \right\rangle$.

\medskip \noindent
Then $S^+(\mathcal{P},a_1b_1)$ contains four left hyperplanes: $[1,a_1 \to a_2, b_1]$, $[1,a_2 \to a_3,b_1]$, $[1,a_3 \to a_1,b_1]$, and $[1,a_1 \to a_1p,b_1]$. Thus, the vertex-spaces of our graph of spaces will be $a_1S(\mathcal{P},b_1)$, $a_2S(\mathcal{P},b_1)$, and $a_3S(\mathcal{P},b_1)$. The graph of spaces given by Proposition~\ref{prop:GraphOfSpaces} is then: 
\begin{center}
\includegraphics[width=0.6\linewidth]{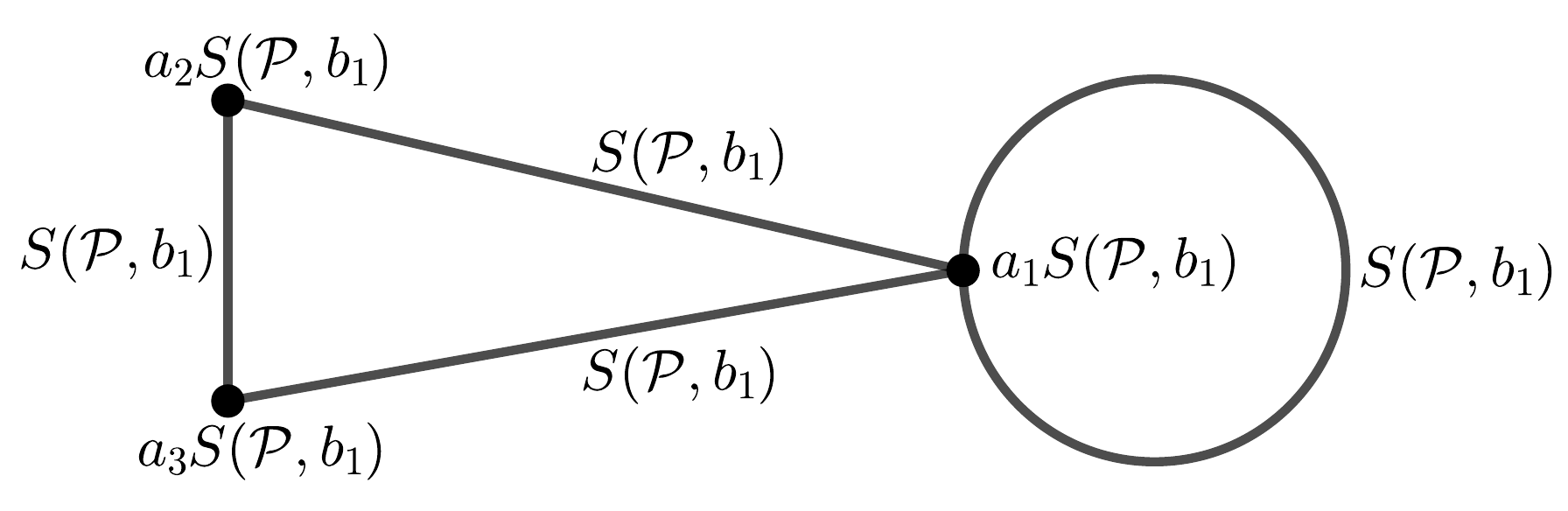}
\end{center}

\noindent
The maps associated to the loop are induced by $w_0 \mapsto a_1w_0$ and $w_0 \mapsto a_1pw_0$. On the other hand, the Squier complex $S(\mathcal{P},b_1)$ is easy to draw:
\begin{center}
\includegraphics[width=0.6\linewidth]{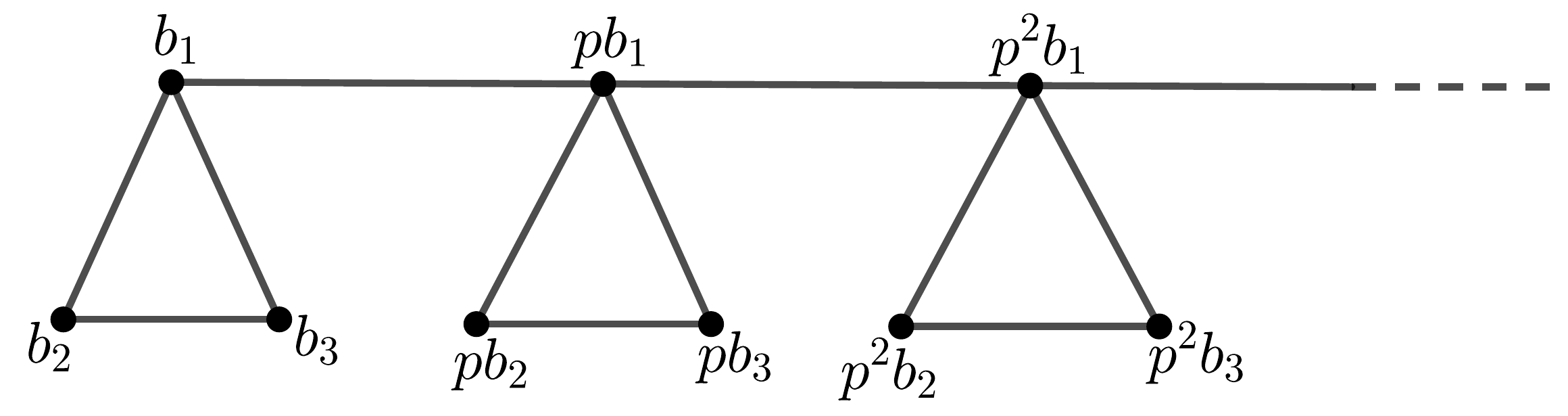}
\end{center}

\noindent
In particular, its fundamental group is isomorphic to $\mathbb{F}_{\infty}= \langle x_1,x_2, \ldots \mid \ \rangle$. 

\medskip \noindent
\begin{minipage}{0.45\linewidth}
\includegraphics[width=0.95\linewidth]{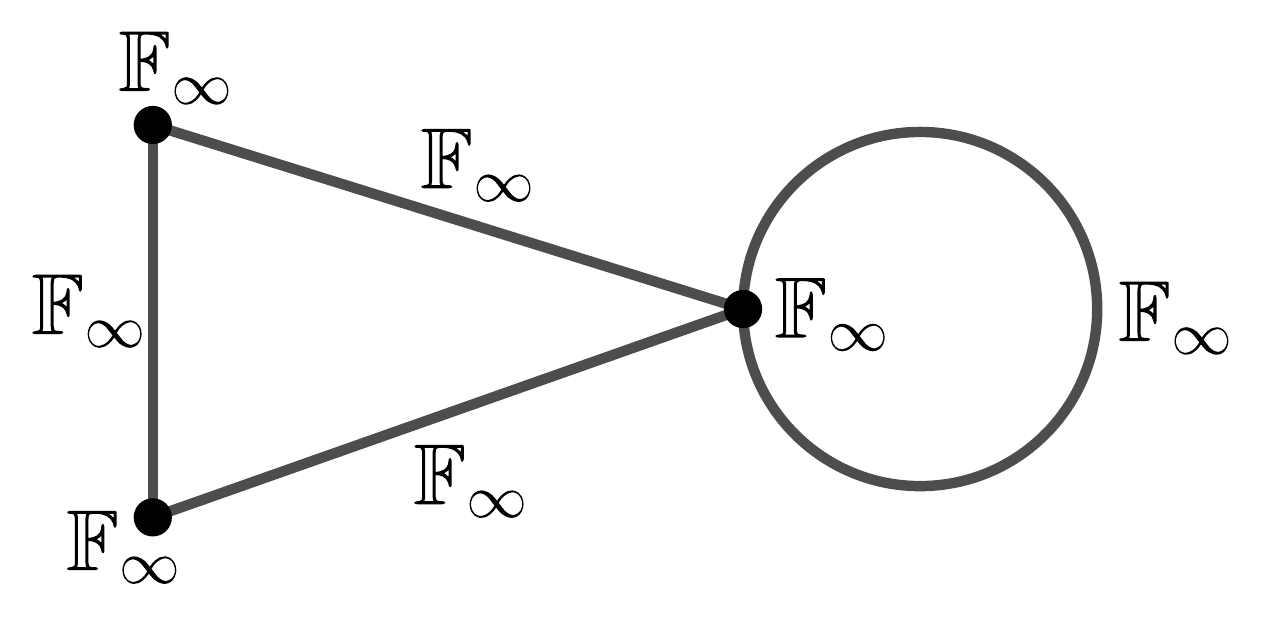}
\end{minipage}
\begin{minipage}{0.53\linewidth}
Thus, $D(\mathcal{P},a_1b_1)$ may be decomposed as the graph of groups illustrated on the left, where the maps associated to the three edges on the left are identities, and where the two maps associated to the loop are the identity and the morphism induced by $x_i \mapsto x_{i+1}$. 
\end{minipage}

\medskip \noindent
\begin{minipage}{0.63\linewidth}
Finally, this graph of groups may be simplified into the graph of groups illustrated on the right, where the maps associated to the left loop are identities and where the two maps associated to the right loop are the identity and the morphism induced by $x_i \mapsto x_{i+1}$. 
\end{minipage}
\begin{minipage}{0.35\linewidth}
\includegraphics[width=\linewidth]{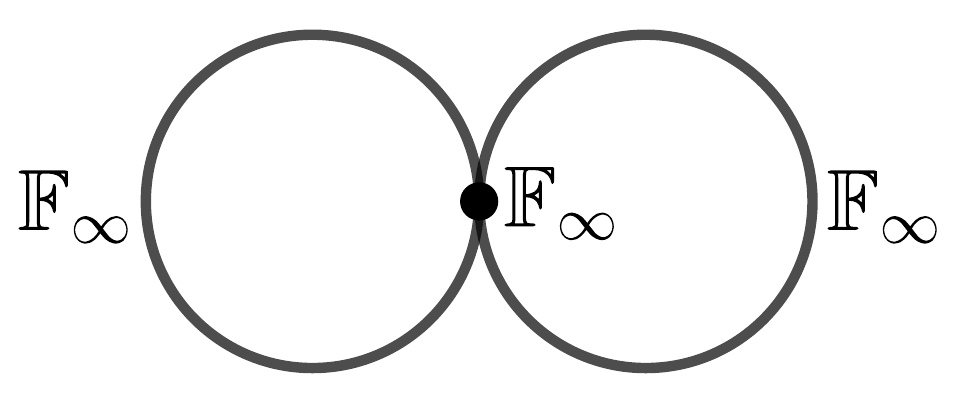}
\end{minipage}

\medskip \noindent
Thus, we deduce the following presentation of our diagram group $D(\mathcal{P},a_1b_1)$:
$$\langle x_1,x_2, \ldots, t,h \mid tx_it^{-1}=x_i, hx_ih^{-1}=x_{i+1} \ (i \geq 1 ) \rangle.$$
Noticing that $x_{i+1}=h^i x_1 h^{-i}$ for every $i \geq 1$, we can simplify the presentation as
$$\mathbb{Z} \bullet \mathbb{Z} = \langle a,t,h \mid [t,a^{h^i}]=1 \ (i \geq 0) \rangle,$$
which is the presentation given in Section~\ref{section:Examples}.
\end{ex}

\subsection{Acylindrical hyperbolicity}\label{section:Acyl}

\noindent
A powerful method to study groups, coming from geometric group theory, is to look for behaviours of negative curvature. We saw with Theorem~\ref{thm:HypDiag} that there is essentially no hyperbolicity in diagram groups, and it can be expected that a similar phenomenon occurs for relative hyperbolicity (see Conjecture~\ref{conj:FreeProduct}). Nevertheless, \emph{acylindrical hyperbolicity} is sufficiently weak to encompass interesting examples of diagram groups and sufficiently restrictive to provide valuable information.

\medskip \noindent
Formally, a group $G$ is \emph{acylindrically hyperbolic} if it admits an action on some (Gromov-)hyperbolic space $X$ which is both \emph{non-elementary} (i.e.\ with an infinite limit set) and \emph{acylindrical} (i.e.\ for every $d \geq 0$, there exist $L,N \geq 0$ such that $\# \{ g \in G \mid d(x,gx),d(y,gy) \leq d\} \leq N$ for all $x,y \in X$ satisfying $d(x,y) \geq L$). Acylindrical hyperbolicity can be thought of as a weak phenomenon of negative curvature, but which nevertheless provides strong restrictions on the structure of the group under consideration. For instance:
\begin{itemize}
	\item Acylindrically hyperbolic groups admit many quotients. This is a consequence of small cancellation \cite{MR3589159} and Dehn fillings \cite{MR3605028}. For instance, every acylindrically hyperbolic group $G$ is \emph{SQ-universal} (i.e.\ every countable group embeds into a quotient of $G$);
	\item Acylindrically hyperbolic groups contain many free subgroups. This can also be seen as a consequence of small cancellation \cite{MR3589159} since the normal subgroups produced are free. But we can also justify this assertion by using the Property $P_{\mathrm{naive}}$ \cite{MR3936081} or the fact that random random subgroups are free \cite{MR4023758}.
	\item Acylindrically hyperbolic groups admit many quasimorphisms \cite{MR1914565}. In particular, this implies that the second bounded cohomology group has uncountable dimension. 
	\item An acylindrically hyperbolic group $G$ does not contain an infinite abelian subgroup $H$ that is \emph{$s$-normal} (i.e.\ $H \cap gHg^{-1}$ is infinite for every $g \in G$) \cite{OsinAcyl}. In particular, centres of acylindrically hyperbolic groups are finite. 
	\item An acylindrically hyperbolic group does not decompose as the direct sum of two infinite groups \cite{OsinAcyl}.
\end{itemize}
Diagram groups may or may not be acylindrically hyperbolic. Obvious examples of acylindrically hyperbolic diagram groups are non-abelian free groups. More interesting examples will be given in this section. On the other hand, free abelian groups, Thompson's group $F$, and the lamplighter group $\mathbb{Z} \wr \mathbb{Z}$ are not acylindrically hyperbolic (since they do not contain any non-abelian free subgroup). Is it possible to recognise whether or not a given diagram group is acylindrically hyperbolic? 

\medskip \noindent
Even though the median graphs constructed in Section~\ref{section:MedianDiag} are usually not hyperbolic, it turns out that there exist efficient methods that allow us to deduce that a group is acylindrically hyperbolic thanks to an action on a median graph; see \cite{MR4057355} and references therein. One of these methods is to look for specific isometries that behave like isometries in hyperbolic spaces. More precisely:

\begin{definition}
Let $X$ be a metric space. An isometry $g \in \mathrm{Isom}(X)$ is \emph{contracting} if there exists a basepoint $o \in X$ such that the map $n \mapsto g^n \cdot o$ induces a quasi-isometric embedding $\mathbb{Z} \to X$; and if there exists some $D \geq 0$ such that the nearest-point projection on $\langle g \rangle \cdot o$ of every ball disjoint from $\langle g \rangle \cdot o$ has diamter at most $D$. 
\end{definition}

\noindent
It turns out that a group acting properly on a metric space with at least one contracting isometry must be either virtually infinite cyclic or acylindrically hyperbolic. Therefore, a natural problem is to determine when an element of a diagram group induces a contracting isometry on the corresponding median graph. 

\begin{thm}[\cite{MR4071367}]\label{thm:Contracting}
Let $\mathcal{P}= \langle \Sigma \mid \mathcal{R} \rangle$ be a semigroup presentation, $w \in \Sigma^+$ a baseword, and $g \in D(\mathcal{P},w)$ a non-trivial absolutely reduced element. Then $g$ is a contracting isometry of $M(\mathcal{P},w)$ if and only if the following conditions are satisfied:
\begin{itemize}
	\item $g^\infty$ does not contain any infinite proper prefix (i.e.\ $g^\infty$ does not decompose as a concatenation $\Phi \circ \Psi$ with $\Phi$ infinite and $\Psi$ non-trivial);
	\item if $\Delta$ is are reduced diagram with $g^\infty$ as a prefix, then $\Delta$ differs from $g^\infty$ by only finitely many $2$-cells. 
\end{itemize}
\end{thm}

\noindent
Roughly speaking, an infinite diagram refers to the object $\Delta_1 \circ \Delta_2 \circ \cdots$ obtained by concatenating infinitely many finite diagrams $\Delta_1,\Delta_2,\ldots$ over a given semigroup presentation. In particular, $g^\infty$ denotes the infinite diagram obtained by concatenating infinitely many copies of the diagram $g$. We refer to \cite{MR4071367} for a more formal description. 

\begin{ex}\label{ex1}
Let $\mathcal{P}= \langle a,b,p \mid a=ap,b=pb \rangle$ and let $g \in D(\mathcal{P},ab)$ be the spherical diagram illustrated by Figure \ref{figure15}.
\begin{figure}[h!]
\begin{minipage}{0.35\linewidth}
\includegraphics[width=0.95\linewidth]{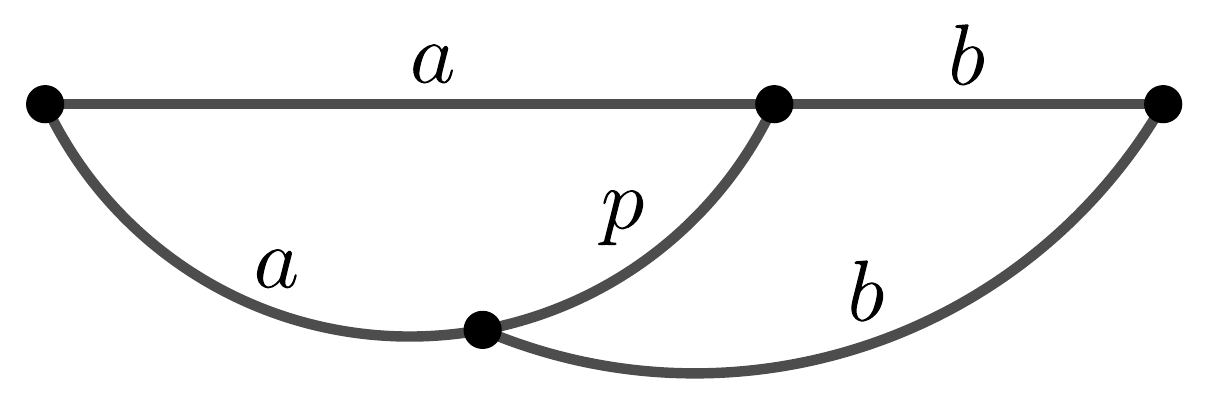}
\end{minipage}
\begin{minipage}{0.63\linewidth}
\includegraphics[width=\linewidth]{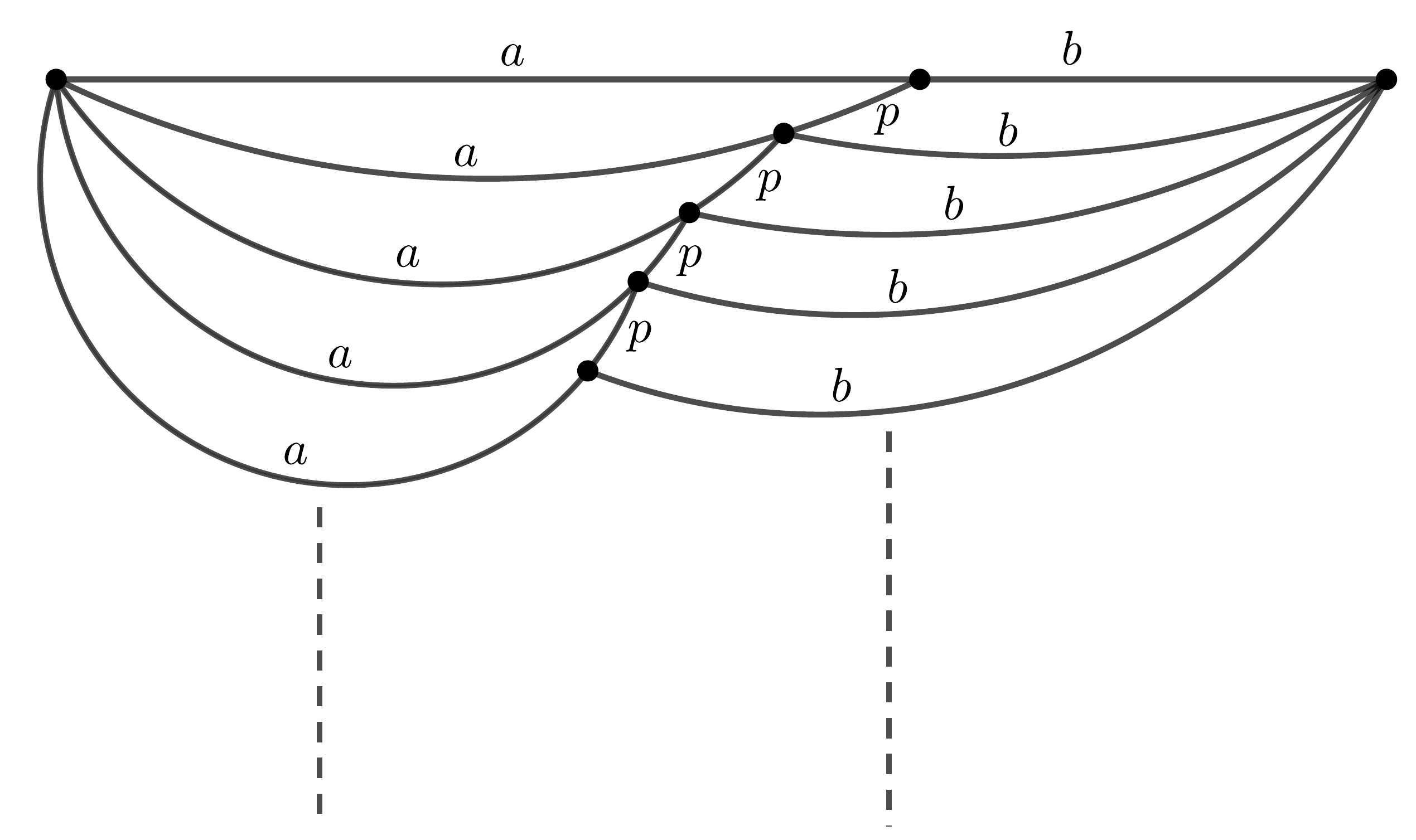}
\end{minipage}
\caption{Diagrams $g$ and $g^\infty$ from Example~\ref{ex1}.}
\label{figure15}
\end{figure}

\noindent
Then $g^{\infty}$ clearly contains a proper infinite prefix. Therefore, $g$ is not a contracting isometry of $M(\mathcal{P},ab)$.
\end{ex}

\begin{ex}\label{ex3}
Let $\mathcal{P}= \langle a,b,c \mid a=b,b=c,c=a \rangle$ and let $g \in D(\mathcal{P},a^2)$ be the spherical diagram illustrated by Figure~\ref{figure19}.
\begin{figure}[h!]
\begin{minipage}{0.4\linewidth}
\includegraphics[width=0.95\linewidth]{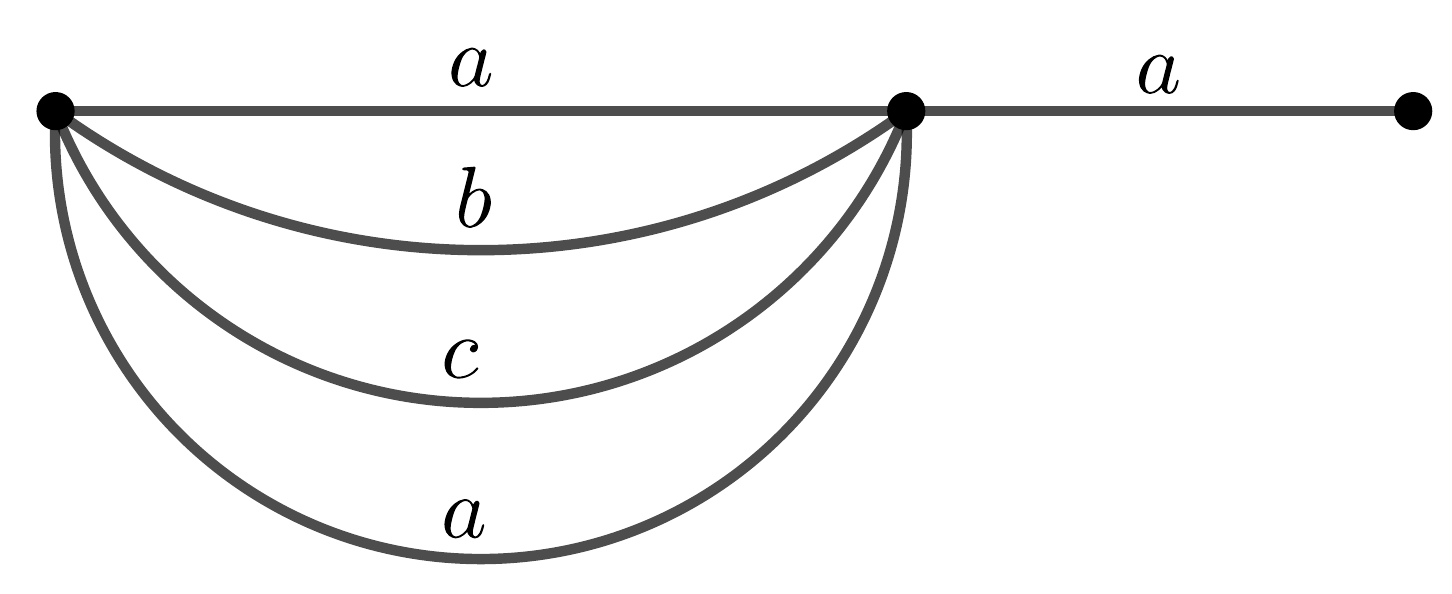}
\end{minipage}
\begin{minipage}{0.58\linewidth}
\includegraphics[width=\linewidth]{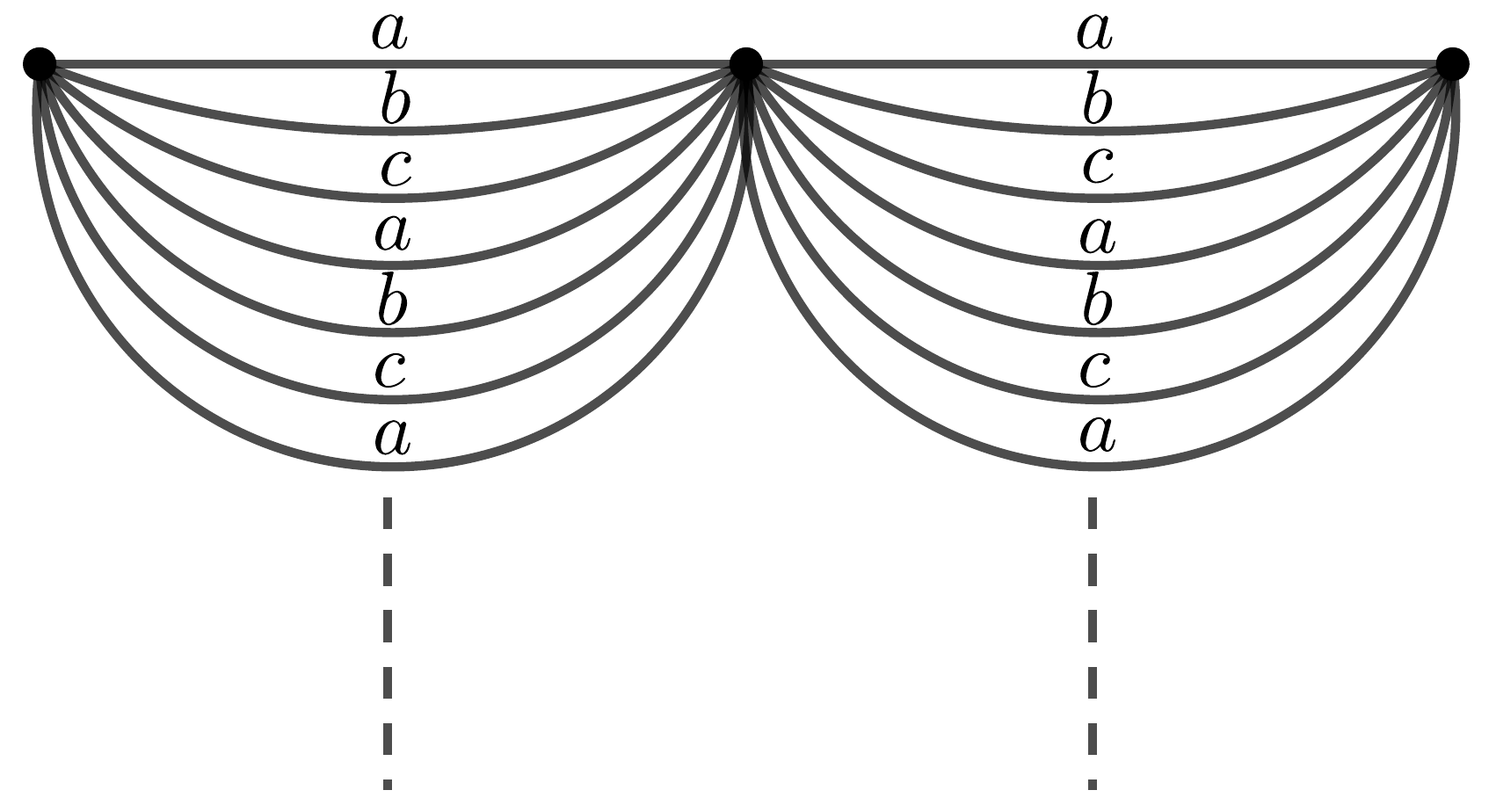}
\end{minipage}
\caption{Diagrams $g$ and $g^\infty$ from Example~\ref{ex3}.}
\label{figure19}
\end{figure}

\noindent
Then $g^{\infty}$ is a prefix of the diagram $\Delta$ given by Figure~\ref{figure19}, but $\Delta$ contains infinitely many $2$-cells not in $g^\infty$. Therefore, $g$ is not a contracting isometry of $M(\mathcal{P},a^2)$. 
\end{ex}

\begin{ex}\label{ex2}
Let $\mathcal{P}= \langle a,b,c,d \mid ab=ac, cd=bd \rangle$ and let $g \in D(\mathcal{P},ab)$ be the spherical diagram illustrated by Figure \ref{figure22}. Then the infinite diagram $g^{\infty}$ is given by Figure \ref{figure22}.
\begin{figure}[h!]
\begin{center}
\includegraphics[width=0.7\linewidth]{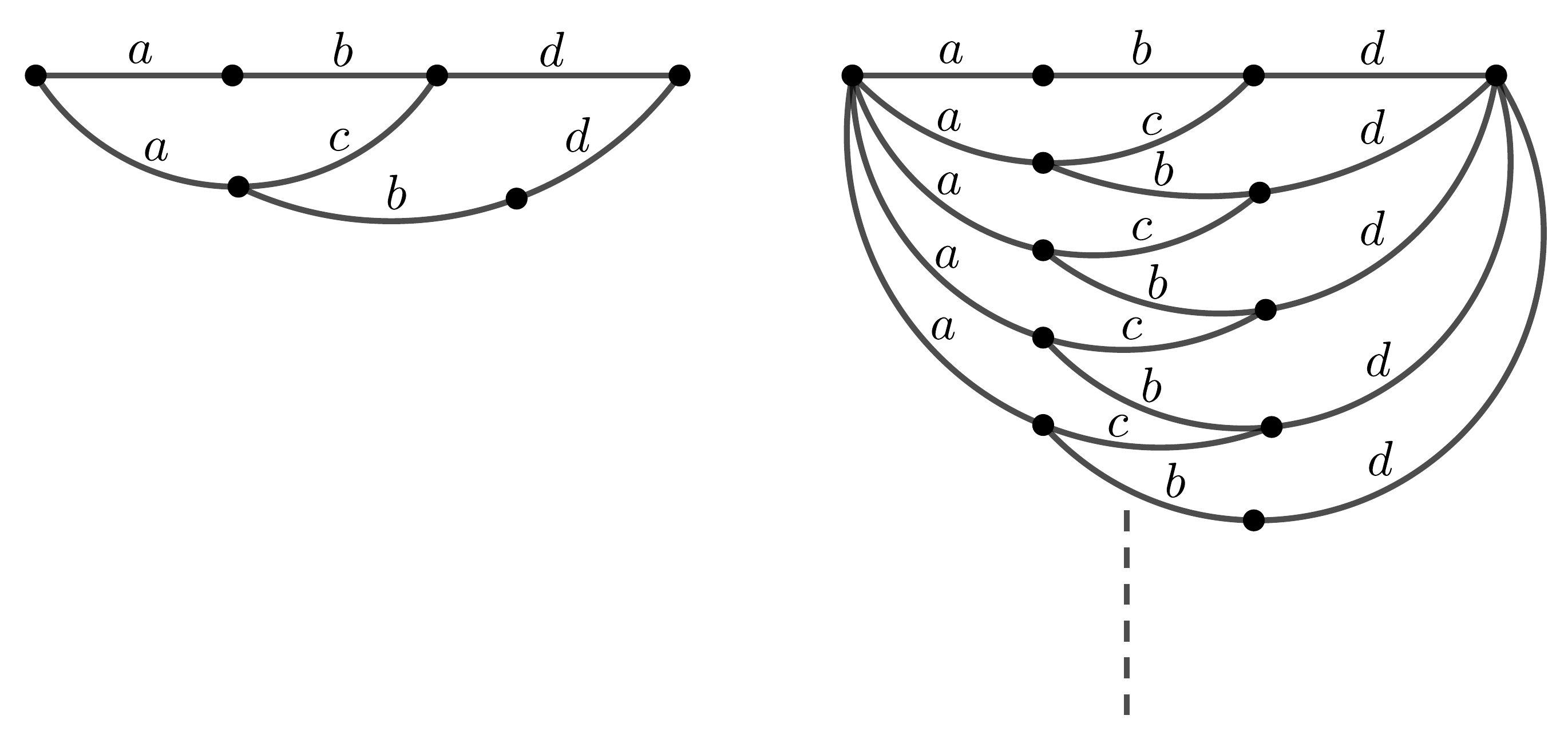}
\caption{The diagram $g$ and its associated infinite diagram $g^{\infty}$ from Example \ref{ex2}.}
\label{figure22}
\end{center}
\end{figure}

\noindent
We can notice that $g^{\infty}$ does not contain a proper infinite prefix and that any diagram containing $g^{\infty}$ as a prefix is necessarily equal to $g^{\infty}$. Therefore, $g$ yields a contracting isometry of $M(\mathcal{P},ab)$. One easily verifies that $D(\mathcal{P},ab)$ is not infinite cyclic, so $D(\mathcal{P},ab)$ is acylindrically hyperbolic.
\end{ex}

\noindent
Theorem~\ref{thm:Contracting} may be considered as unsatisfying because infinite diagrams cannot be drawn. For diagrams with only a few cells, the previous examples show that the theorem can be applied fairly easily, but for bigger diagrams it gets more delicate. For instance, consider the semigroup presentation
$$\left\langle a_1,a_2,a_3,b_1,b_2,b_3,p \mid \begin{array}{c} a_1=a_2, a_2=a_3, a_3=a_1 \\ b_1=b_2, b_2=b_3,b_3=b_1 \end{array}, a_1=a_1p, b_1=pb_1 \right\rangle$$
corresponding to the diagram group $\mathbb{Z} \bullet \mathbb{Z}$, and let $g$ be the element given by the following spherical diagram:
\begin{center}
\includegraphics[width=0.6\linewidth]{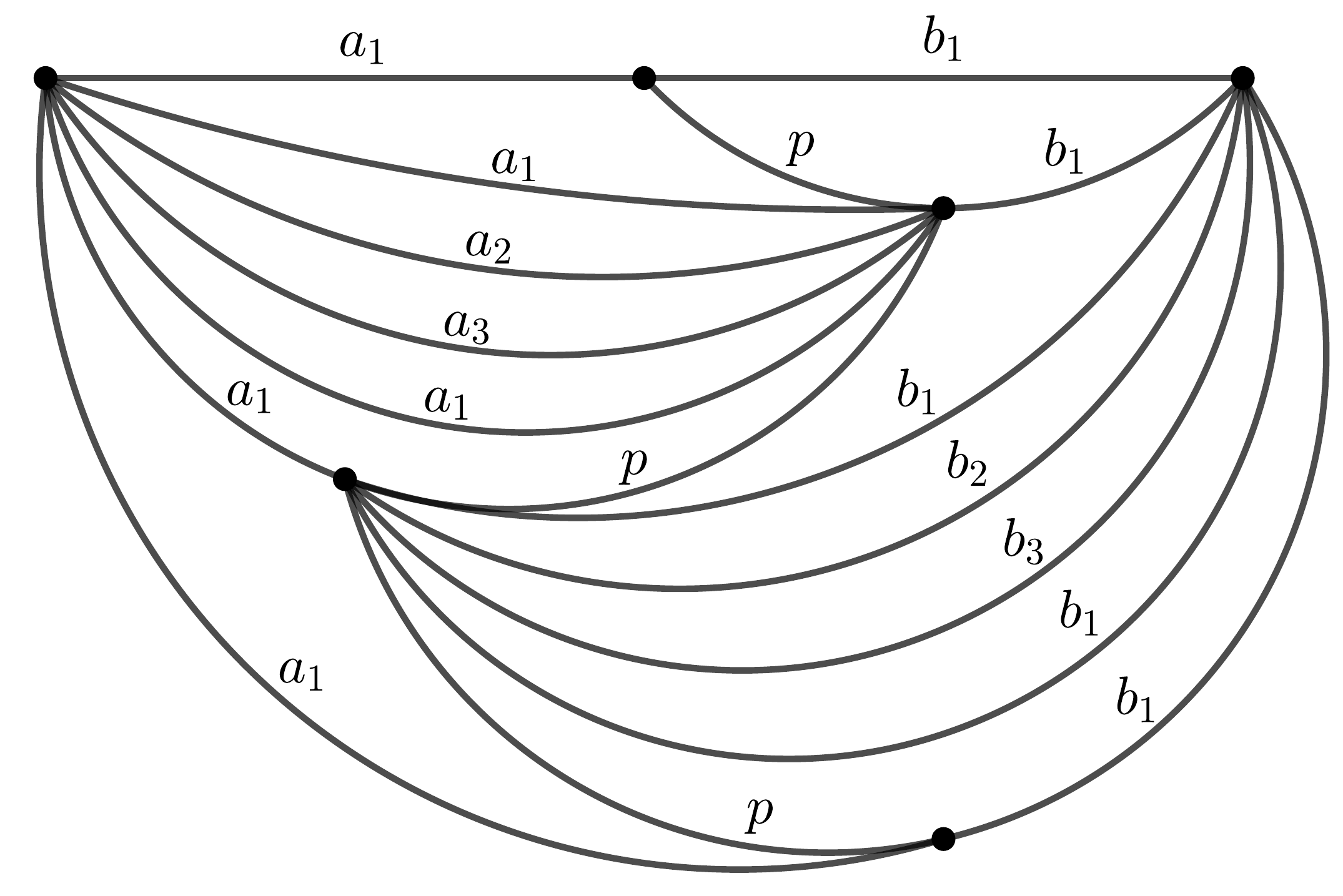}
\end{center}

\noindent
Then it can be verified that $g$ is a contracting isometry of $M(\mathcal{P},a_1b_1)$, proving that $\mathbb{Z} \bullet \mathbb{Z}$ is acylindrically hyperbolic. (In fact, it can be proved that $A \bullet B$ is always acylindrically hyperbolic when the groups $A$ and $B$ are both non-trivial; see \cite[Example~5.44]{MR4071367} for more details.) However, applying Theorem~\ref{thm:Contracting} becomes painful. We refer to \cite[Proposition~5.39]{MR4071367} for an efficient method that allows us verify the conditions given by Theorem~\ref{thm:Contracting} more easily. 

\medskip \noindent
Theorem~\ref{thm:Contracting} only provides a sufficient condition to be acylindrically hyperbolic. In fact, the diagram group may be acylindrically hyperbolic but with no contracting isometries in the associated median graph. However, when the action on the median graph is cocompact, the criterion turns out to provide a necessary and sufficient condition. 

\begin{thm}\label{thm:CocompactAH}
Let $\mathcal{P}= \langle \Sigma \mid \mathcal{R} \rangle$ be a semigroup presentation and $w \in \Sigma^+$ a baseword with $[w]_\mathcal{P}$ finite. There exist words $u_0,u_1, \ldots, u_m \in \Sigma^+$ and a $(w,u_0u_1 \cdots u_m)$-diagram $\Gamma$ such that
$$D(\mathcal{P},w) = \Gamma \cdot \left( D(\mathcal{P},u_0) \times D(\mathcal{P},u_1) \times \cdots \times D(\mathcal{P},u_m) \right) \cdot \Gamma^{-1},$$
where each $D(\mathcal{P},u_i)$ is trivial, infinite cyclic, or acylindrically hyperbolic.
\end{thm}

\noindent
We already mentioned that acylindrically hyperbolic groups do not split as product of infinite groups. Theorem~\ref{thm:CocompactAH} shows that this is essentially the only obstruction for a cocompact diagram group to be acylindrically hyperbolic.

\subsection{Diagram products and quasi-median geometry}\label{section:DiagramProducts}

\noindent
Diagram products, first introduced in \cite{MR1725439}, can be thought of as \emph{diagram groups with coefficients}, where the coefficients come from some collection of groups (the \emph{factors} of the product). Below, we follow the definition given in \cite{QM}.

\medskip \noindent
Let $\mathcal{P}= \langle \Sigma \mid \mathcal{R} \rangle$ be a semigroup presentation and $\mathcal{G}= \{ G_s \mid s \in \Sigma \}$ a collection of groups indexed by the letters of our alphabet $\Sigma$. A \emph{diagram over $(\mathcal{P},\mathcal{G})$} is a diagram $\Delta$ over $\mathcal{P}$ such that each edge of $\Delta$ labelled by a letter $s \in \Sigma$ is also labelled by an element of $G_s$. In other words, the edges of a diagram are labelled by $\{ (s,g) \mid s \in \Sigma, g \in G_s\}$. Most of the vocabulary introduced in Section~\ref{section:Diagrams} extends naturally to such diagrams. Let us record the two major differences.
\begin{itemize}
	\item In a diagram $\Delta$ over $(\mathcal{P},\mathcal{G})$, a \emph{dipole} refers to two cells $\pi_1,\pi_2$ satisfying $\mathrm{bot}(\pi_1)= \mathrm{top}(\pi_2)$ and labelled as follows: the top path of $\pi_1$ is labelled by $(u_1,g_1) \cdots (u_m,g_m)$, its bottom path is labelled by $(v_1,1) \cdots (v_n,1)$, and the bottom path of $\pi_2$ is labelled by $(u_1,h_1) \cdots (u_m,h_m)$ where $u_1 \cdots _m= v_1 \cdots v_n$ or its inverse belongs to $\mathcal{R}$. One \emph{reduces} the dipole by removing $\pi_1,\pi_2$, identifying $\mathrm{top}(\pi_1)$ with $\mathrm{bot}(\pi_2)$, and labelling this path by $(u_1,g_1h_1) \cdots (u_m, g_mh_m)$. 
	\item Let $\Delta_1,\Delta_2$ be two diagrams over $(\mathcal{P},\mathcal{G})$. If $\mathrm{bot}(\Delta_1)$ is labelled by $(u_1,g_1) \cdots (u_m,g_m)$ and $\mathrm{top}(\Delta_2)$ by $(u_1,h_1) \cdots (u_m,h_m)$, the \emph{concatenation} $\Delta_1 \circ \Delta_2$ of $\Delta_1$ and $\Delta_2$ is the diagram obtained by identifying $\mathrm{bot}(\Delta_1)$ with $\mathrm{top}(\Delta_2)$ and by labelling this common path by $(u_1,g_1h_1) \cdots (u_m,g_mh_m)$. 
\end{itemize}

\begin{definition}
Let $\mathcal{P}= \langle \Sigma \mid \mathcal{R} \rangle$ be a semigroup presentation, $\mathcal{G}= \{ G_s \mid s \in \Sigma \}$ a collection of groups, and $w \in \Sigma^+$ a baseword. The \emph{diagram product} $D(\mathcal{P},\mathcal{G},w)$ is the group whose elements are the diagrams over $(\mathcal{P}, \mathcal{G})$ modulo dipole reductions and whose product is given by the concatenation.
\end{definition}

\noindent
As an illustration, consider the semigroup presentation $\mathcal{P}= \langle a,b,c \mid ab=ba, ac=ca, bc=cb \rangle$ and the factors $\mathcal{G}= \{G_a=G_b=G_c= \mathbb{Z} \}$. Here is a product of two diagrams over $(\mathcal{P}, \mathcal{G})$:

\noindent
\includegraphics[width=\linewidth]{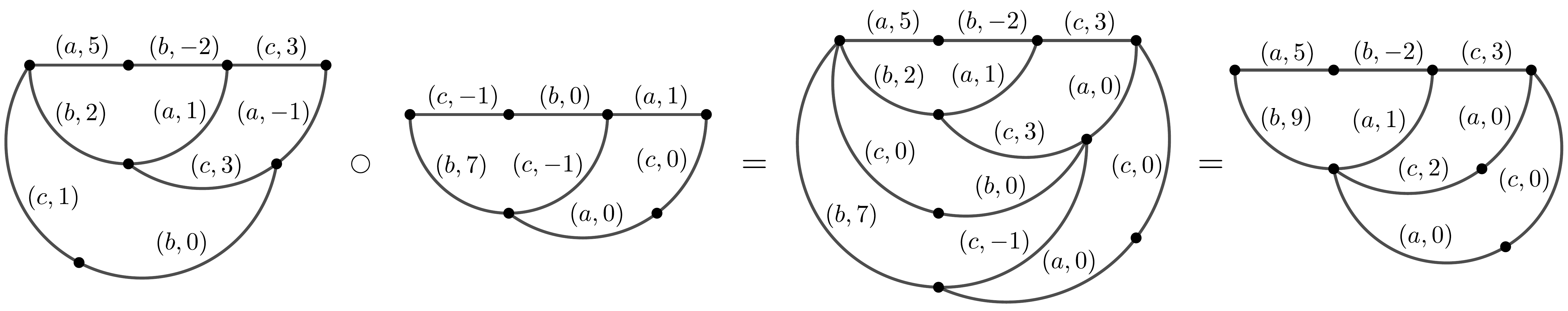}

\noindent
We emphasize that the rightmost diagram is reduced (even though there is a dipole when we remove the labels coming from the factors). 

\medskip \noindent
In the same way that diagram groups can be described both in terms of diagrams and as fundamental groups of Squier complexes, it is also possible to describe diagram products as fundamental groups of specific complexes, which we describe now.

\begin{definition}
Let $\mathcal{P}= \langle \Sigma \mid \mathcal{R} \rangle$ be a semigroup presentation and $\mathcal{G}= \{G_s \mid s \in \Sigma \}$ a collection of groups indexed by $\Sigma$. The \emph{Squier complex} $S(\mathcal{P},\mathcal{G})$ is the square complex whose vertex-set is $\Sigma^+$ with
\begin{itemize}
	\item an edge $(a,u \to v,b)$ between the words $aub$ and $avb$ whenever $(u=v) \in \mathcal{R}$;	
	\item a loop $b(g)$ based at $w= \ell_1 \cdots \ell_n \in \Sigma^+$ for every $g \in G(w) := G_{\ell_1} \times \cdots \times G_{\ell_n}$;
	\item a square delimited by the four vertices $aubpc$, $avbpc$, $aubqc$, and $avbqc$ whenever $(u=v),(p=q) \in \mathcal{R}$;
	\item a triangle delimited by the three loops based at $w$ labelled by $g,h,gh \in G(w)$;
	\item a cylinder between the loops labelled by $g \in G(a) \times G(b) \leq G(aub)$ and $h \in G(a) \times G(b) \leq G(avb)$ following the edge $(a,u \to v,b)$, where $(u=v) \in \mathcal{R}$. 
\end{itemize}
\end{definition}

\noindent
The reader familiar with complexes of groups should recognise the topological realisation of a simple complex of groups with the Squier square complex $S(\mathcal{P})$ as the underlying complex; we refer to \cite[Chapter~II.12]{MR1744486} for more information on simple complexes of groups. In \cite{MR1725439}, diagram products are defined as fundamental groups of such complexes. This point of view is equivalent to the definition presented here. 

\begin{prop}[\cite{QM}]\label{prop:DiagProdSquier}
Let $\mathcal{P}= \langle \Sigma \mid \mathcal{R} \rangle$ be a semigroup presentation, $\mathcal{G}= \{G_s \mid s \in \Sigma \}$ a collection of groups, and $w \in \Sigma^+$ a baseword. The diagram product $D(\mathcal{P},\mathcal{G},w)$ is naturally isomorphic to the fundamental group of $S(\mathcal{P},\mathcal{G})$ based at $w$.
\end{prop}

\noindent
The proof of Proposition~\ref{prop:DiagProdSquier} is basically the same as the proof of Corollary~\ref{cor:PathDiag}. It can be shown that a diagram product of diagram groups yields a diagram group. More precisely:

\begin{thm}[\cite{MR1725439}]
Let $\mathcal{P}= \langle \Sigma \mid \mathcal{R} \rangle$ be a semigroup presentation and $w \in \Sigma^+$ a baseword. For every $s \in \Sigma$, fix a semigroup presentation $\mathcal{P}_s = \langle \Sigma_s \mid \mathcal{R}_s \rangle$ and a baseword $w_s \in \Sigma^+$. Assume that the alphabets $\Sigma$, $\Sigma_s$ ($s \in \Sigma$) are pairwise disjoint. The diagram product $D(\mathcal{P},\mathcal{G},w)$, where $\mathcal{G}:= \{ D(\mathcal{P}_s,w_s) \mid s \in \Sigma\}$, is isomorphic to the diagram group given by the semigroup presentation
$$\left\langle \Sigma \sqcup \{x_s \mid s \in \Sigma\} \sqcup \bigsqcup\limits_{s \in \Sigma} \Sigma_s \mid \bigsqcup\limits_{s \in \Sigma} \mathcal{R}_s \sqcup \mathcal{R} \sqcup \{ s = x_s w_s x_s, s \in \Sigma\} \right\rangle$$
and the baseword $w$. 
\end{thm}

\noindent
It is worth noticing that some of the examples given in Section~\ref{section:Examples} arise in this way.  

\medskip \noindent
Because diagram groups admit a natural median geometry, it is reasonable to expect diagram products to admit a median geometry relative to their factors. This idea is formalised in \cite{QM} through \emph{quasi-median} geometry. (Compare with Remark~\ref{remark:QMvsM} below.)

\begin{definition}
Let $X$ be a connected graph. Given three vertices $x_1,x_2,x_3 \in X$, a \emph{median triangle} is the data of three vertices $y_1,y_2,y_3 \in X$ such that
$$d(x_i,x_j)=d(x_i,y_i)+d(y_i,y_j)+d(y_j,x_j), \ \forall i \neq j.$$
Its \emph{size} if $d(y_1,y_2)+d(y_2,y_3)+d(y_3,y_1)$. The graph $X$ is \emph{quasi-median} if any three vertices admit a unique median triangle of minimal size and if the gated hull of any such median triangle is a product of complete graphs.
\end{definition}
\begin{center}
\includegraphics[width=\linewidth]{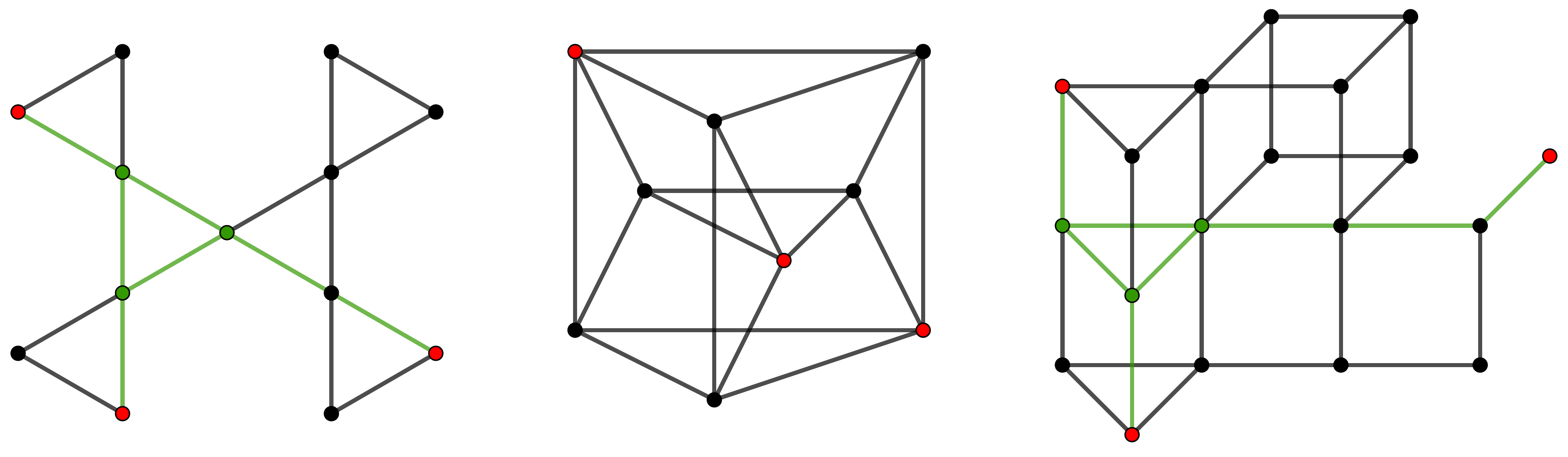}
\end{center}

\noindent
Recall that, given a graph $X$, a subgraph $Y \subset X$ is \emph{gated} if, for every $x \in X$, there exists some $z \in Y$ (referred to as the \emph{gate} or \emph{projection}) that belongs to at least one geodesic from $x$ to $y$ for every $y \in Y$. Notice that, when it exists, the gate is unique and it minimises the distance to $x$ from $Y$. Gatedness should be thought of as a strong convexity. 

\medskip \noindent
In the same way that median graphs can be described as one-skeleta of nonpositively curved cellular complexes, namely CAT(0) cube complexes, quasi-median graphs can be described as one-skeleta of CAT(0) prism complexes (where a \emph{prism} refers to a product of simplices) \cite{QM}. 

\begin{definition}
Let $\mathcal{P}= \langle \Sigma \mid \mathcal{R} \rangle$ be a semigroup presentation and $\mathcal{G}= \{ G_s \mid s \in \Sigma \}$ a collection of groups. Let $\mathrm{QM}(\mathcal{P},\mathcal{G})$ denote the graph whose vertices are the diagrams over $(\mathcal{P},\mathcal{G})$ modulo dipole reductions and whose edges connect two diagrams if one can be obtained from the other by right-multiplying by a \emph{unitary} diagram (i.e.\ a diagram that either contains a single cell and all of whose edges have trivial second coordinates in their labels, or has no cell and a single edge with a non-trivial second coordinate in its label). 
\end{definition}

\noindent
One easily sees that two diagrams belong to the same connected component of $\mathrm{QM}(\mathcal{P},\mathcal{G})$ if and only if the labels in $\Sigma^+$ (i.e.\ when the second coordinates are forgotten) of their top paths are identical. Given a baseword $w \in \Sigma^+$, we denote by $\mathrm{QM}(\mathcal{P},\mathcal{G},w)$ the connected component containing the diagrams whose top paths are labelled by $w$ in $\Sigma^+$. The diagram product $D(\mathcal{P},\mathcal{G},w)$ naturally acts on $\mathrm{QM}(\mathcal{P},\mathcal{G},w)$ by left-multiplication. 

\begin{figure}[h!]
\begin{center}
\includegraphics[width=0.8\linewidth]{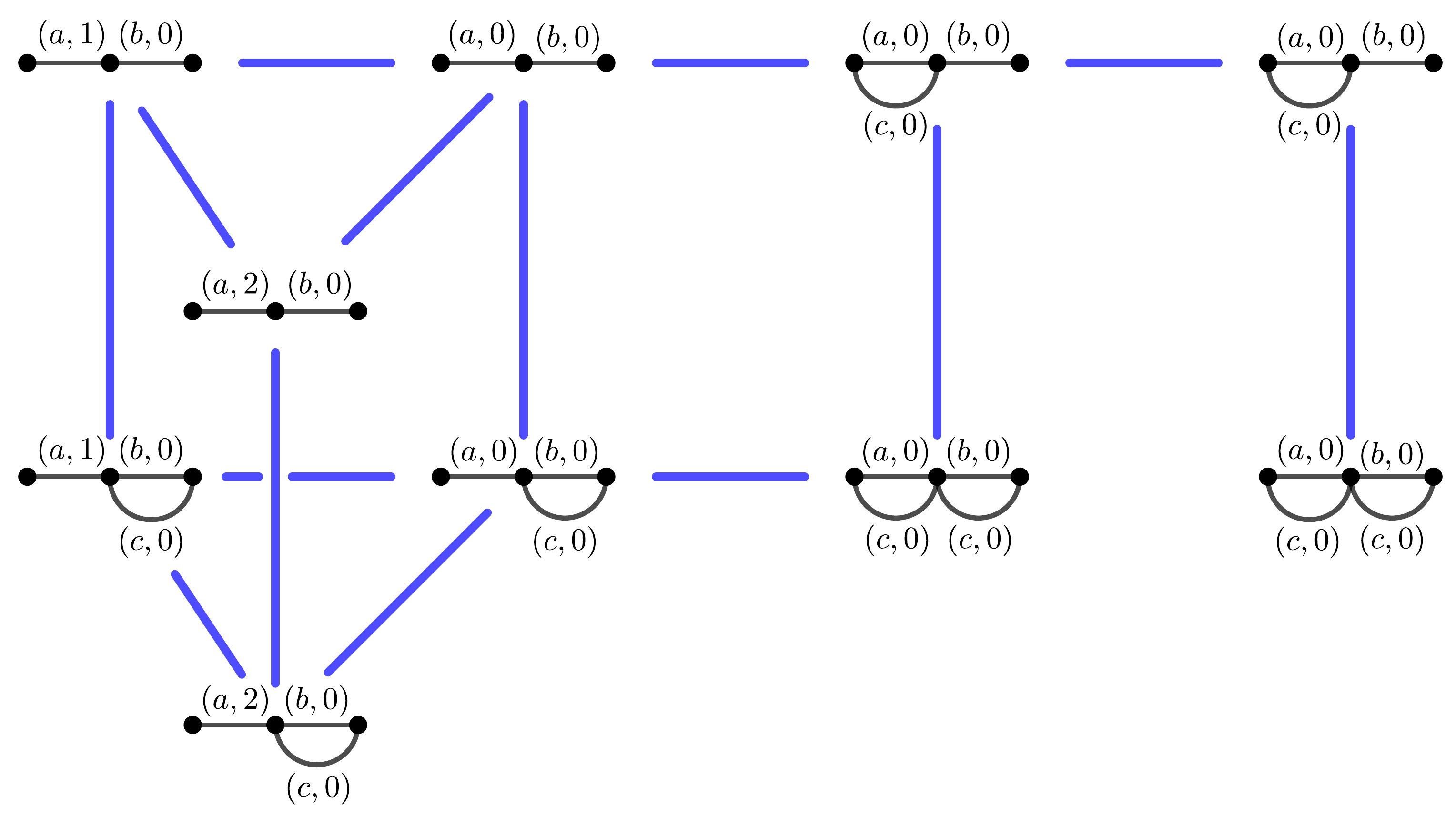}
\caption{A piece of $\mathrm{QM}(\mathcal{P},\mathcal{G},ab)$ for $\mathcal{P}=\langle a,b,c \mid a=c,b=c \rangle$ and $G_a=G_b=G_c= \mathbb{Z}$.}
\label{MedProd}
\end{center}
\end{figure}

\begin{thm}[\cite{QM}]
Let $\mathcal{P}= \langle \Sigma \mid \mathcal{R} \rangle$ be a semigroup presentation, $\mathcal{G}= \{ G_s \mid s \in \Sigma \}$ a collection of groups, and $w \in \Sigma^+$ a baseword. Then $\mathrm{QM}(\mathcal{P},\mathcal{G},w)$ is a quasi-median graph on which the diagram product $D(\mathcal{P},\mathcal{G},w)$ acts freely. The action is cocompact if and only if $[w]_\mathcal{P}$ is finite and all the groups in $\mathcal{G}$ are finite. 
\end{thm}

\noindent
The combination of the combinatorics of diagrams and the quasi-median geometry allows us to deduce a lot of information about diagram products. For instance:
\begin{itemize}
	\item A finitely generated diagram product has a solvable word problem if so do its factor.
	\item A diagram product is torsion-free if and only if so are its factors.
	\item Distorted elements in a diagram product belong to products of factors.
	\item A diagram product of groups acting nicely on median graphs acts nicely on median graphs \cite[Theorems~10.36--10.38]{QM}.
	\item The Hilbert space compression of a diagram product is bounded below by the compressions of its factors \cite[Theorem~10.39]{QM}.
\end{itemize}
It is worth noticing that every diagram product $D(\mathcal{P},\mathcal{G},w)$ surjects onto its underlying diagram group $D(\mathcal{P},w)$ via the map that ``forgets'' the labels coming from the factors. Conversely, it is clear that $D(\mathcal{P},w)$ embeds into $D(\mathcal{P},\mathcal{G},w)$: it corresponds to the diagrams all of whose edges have labels with trivial second coordinates (referred to as diagrams over $\mathcal{P}$ for simplicity). Therefore, diagram products split as semidirect products. More precisely:

\begin{thm}[\cite{QM}]\label{thm:SemiProduct}
Let $\mathcal{P}= \langle \Sigma \mid \mathcal{R} \rangle$ be a semigroup presentation, $\mathcal{G}= \{ G_s \mid s \in \Sigma \}$ a collection of groups, and $w \in \Sigma^+$ a baseword. The diagram product decomposes~as 
$$D(\mathcal{P},\mathcal{G},w)= A(\mathcal{P}, \mathcal{G},w) \rtimes D(\mathcal{P},w)$$
where $A(\mathcal{P},\mathcal{G},w)$ is a graph product whose vertex-groups are isomorphic to group in $\mathcal{G}$. 
\end{thm}

\noindent
We refer to \cite[Theorem~10.58]{QM} for an even more precise statement. Recall that, given a graph $\Gamma$ and a collection of groups $\mathcal{G}= \{ G_u \mid u \in V(\Gamma)\}$ indexed by the vertices of $\Gamma$, the \emph{graph product} $\Gamma \mathcal{G}$ is
$$\langle G_u, \ u \in V(\Gamma) \mid [G_u,G_v]=1, \ \{u,v\} \in E(\Gamma) \rangle$$
where $[G_u,G_v]=1$ is a shorthand for: $[a,b]=1$ for all $a \in G_u$ and $b \in G_b$. Usually, graph products are described as an interpolation between free products (when $\Gamma$ has no edge) and direct sums (when $\Gamma$ is a complete graph). 

\medskip \noindent
The strategy to prove the theorem is the following. Let $Y$ denote the subgraph of $M(\mathcal{P},\mathcal{G},w)$ given by the diagrams over $\mathcal{P}$. Its stabiliser under the action of $D(\mathcal{P},\mathcal{G},w)$ coincides with $D(\mathcal{P},w)$. Also, let $\mathcal{J}$ denote the collection of all the hyperplanes in $M(\mathcal{P},\mathcal{G},w)$ tangent to $Y$. In a quasi-median graph, a hyperplane is a class of parallel \emph{cliques} (i.e.\ maximal complete subgraphs) instead of parallel edges in the case of median graphs. To each hyperplane corresponds a \emph{rotative-stabiliser}, i.e.\ the subgroup of its stabiliser that stabilises each of the cliques it contains. Thanks to a Ping-Pong Lemma, it can be shown that the subgroup generated by the rotative-stabilisers of all the hyperplanes in $\mathcal{J}$ decomposes as the graph product of the rotative-stabilisers over the crossing graph of $\mathcal{J}$ (i.e.\ the graph whose vertex-set is $\mathcal{J}$ and whose edges connect two hyperplanes whenever they are transverse). This is our graph product $A(\mathcal{P},\mathcal{G},w)$. One easily verifies that $D(\mathcal{P},w)$ normalises and intersects trivially $A(\mathcal{P},\mathcal{G},w)$. The proof is complete once it is shown that $A(\mathcal{P},\mathcal{G},w)$ and $D(\mathcal{P},w)$ generate the diagram product.

\medskip \noindent
As an application of Theorem~\ref{thm:SemiProduct}, it is possible to prove the following theorem, initially obtained in \cite{MR2193191} by using different methods.

\begin{thm}\label{thm:Short}
Every diagram group $G$ satisfies a short exact sequence
$$1 \to R \to G \to S \to 1$$
where $R$ is a subgroup of a right-angled Artin group and where $S$ is a subgroup of Thompson's group $F$.
\end{thm}

\noindent
The strategy is the following. Let $\mathcal{P}= \langle \Sigma \mid \mathcal{R} \rangle$ be an arbitrary semigroup presentation and $w \in \Sigma^+$ a baseword. The goal is to prove that that $D(\mathcal{P},w)$ embeds into the diagram product $D(\mathcal{Q}, \mathcal{G},x)$ where $\mathcal{Q}= \langle x \mid x=x^2 \rangle$ and where $G_x$ is the free group $\mathbb{F}$ formally generated by the relations in $\mathcal{R}$. The desired conclusion will follow since $D(\mathcal{Q}, \mathcal{G},x)$ splits as a semidirect product of a right-angled Artin group with Thompson group's $F$ as a consequence of Theorem~\ref{thm:SemiProduct}.

\medskip \noindent
In order construct such an embedding, we fix, for every $n \geq 1$, a well-chosen $(x^n,x)$-diagram $\Gamma_n$ over $\mathcal{Q}$. Then, we transform a diagram $\Delta$ over $\mathcal{P}$ into a diagram over $(\mathcal{Q}, \mathcal{G})$ by replacing each cell of $\Delta$ labelled by a relation $u=v$ with the diagram $\Gamma_n \circ \epsilon(f) \circ \Gamma_m^{-1}$, where $n$ is the length of $u$, $m$ the length of $v$, and $\epsilon(f)$ the diagram reduced to a single edge labelled by $(x,f)$ with $f$ the (inverse of the) generator of $\mathbb{F}$ corresponding to $u=v$. 

\medskip \noindent
Because right-angled Artin groups and Thompson's group $F$ are locally indicable, and since locally indicable groups are automatically orderable, Theorem~\ref{thm:Short} immediately implies:

\begin{cor}\label{cor:LocallyIndicable}
Diagram groups are locally indicable, hence orderable. 
\end{cor}

\begin{remark}\label{remark:QMvsM}
It is worth noticing that one can make a diagram product $D(\mathcal{P},\mathcal{G},w)$ act on a median graph instead of a quasi-median graph. Given a diagram $\Delta$ and a (possibly empty) collection of edges $\epsilon$ in the bottom path of $\Delta$, let $\Delta(\epsilon)$ denote the set of the diagrams obtained from $\Delta$ by modifying the second coordinates in the labels of the edges in $\epsilon$. Observe that $\Delta(\epsilon)$ defines a coset in the groupoid of the diagrams over $(\mathcal{P},\mathcal{G})$. Let $M(\mathcal{P},\mathcal{G})$ denote the graph whose vertices are the cosets $\Delta(\epsilon)$ and whose edges connect two cosets $\Phi(\zeta)$ and $\Psi(\xi)$ either if $\Psi= \Phi$ and the symmetric difference between $\zeta,\xi$ has cardinality one or if $\Psi$ is obtained from $\Phi$ by gluing a $2$-cell below $\Phi$ along a path disjoint from $\zeta$ and $\xi$ is the image of $\zeta$ after the process. Then the diagram product $D(\mathcal{P},\mathcal{G},w)$ acts on a connected component $M(\mathcal{P},\mathcal{G},w)$ of $M(\mathcal{P},\mathcal{G})$ with stabilisers isomorphic to products of factors. In fact, $M(\mathcal{P},\mathcal{G})$ and $\mathrm{QM}(\mathcal{P},\mathcal{G})$ are basically subdivisions of a common space, but in practice it turns out to be easier and more natural to work with the quasi-median graph. 
\end{remark}

\section{Generalisations}\label{section:generalisations}

\subsection{Planar, annular, and symmetric diagram groups}

\noindent
Let $\mathcal{P}= \langle \Sigma \mid \mathcal{R} \rangle$ be a semigroup presentation. Given a diagram over $\mathcal{P}$, one can construct a \emph{dual picture} made of \emph{transistors} which are connected by \emph{wires} labelled by $\Sigma$. See Figure~\ref{Dual}. Roughly speaking, a transistor is a black box with top and bottom wires such that, if $u \in \Sigma^+$ (resp. $v \in \Sigma^+$) is the word obtained by reading from left to right the labels of its top (resp. bottom) wires, then $u=v$ or $v=u$ belongs to $\mathcal{R}$. 

\begin{figure}[h!]
\begin{center}
\includegraphics[width=0.35\linewidth]{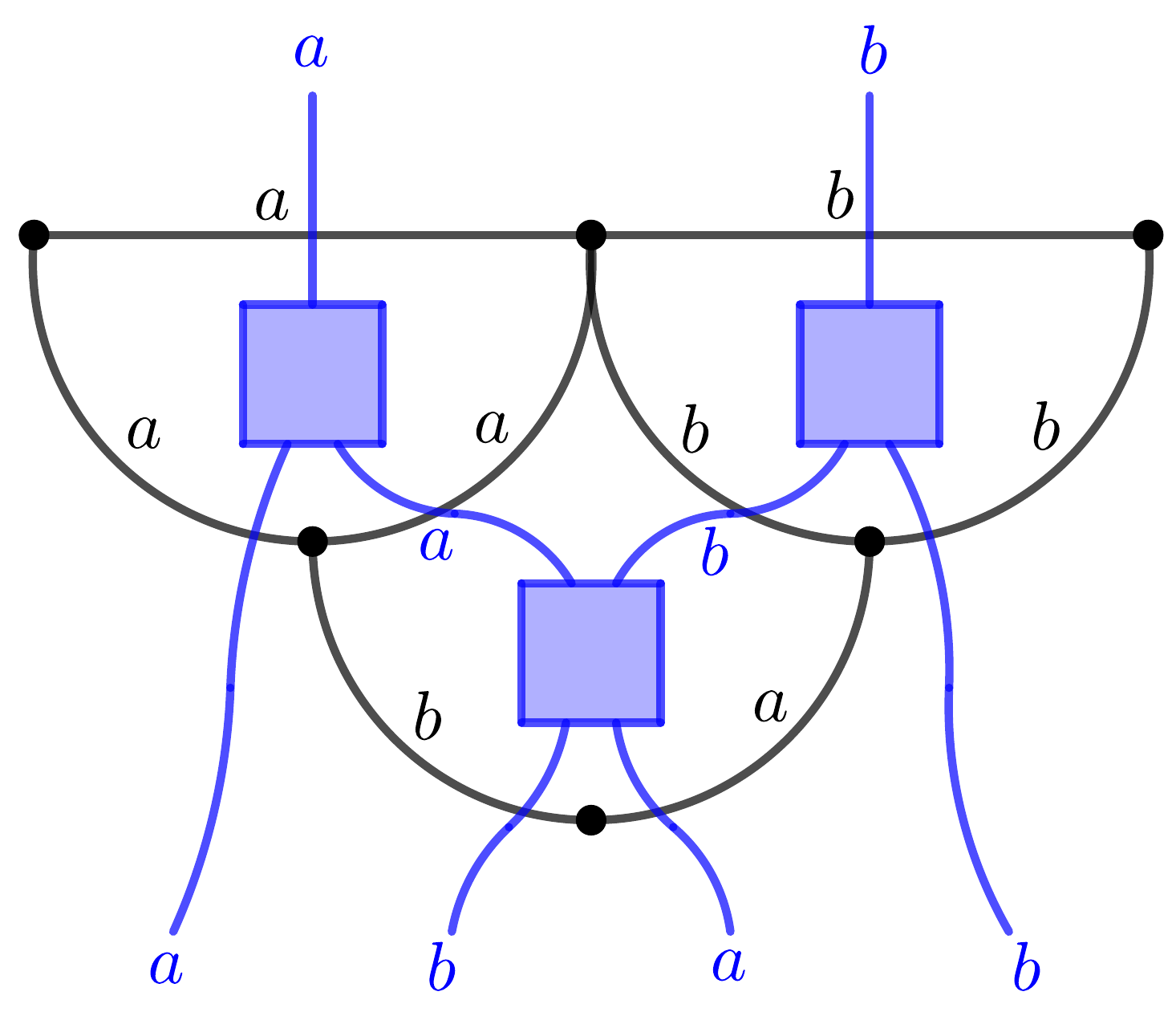}
\caption{A diagram over $\langle a,b \mid a=a^2,b=b^2,ab=ba \rangle$ and its dual picture.}
\label{Dual}
\end{center}
\end{figure}

\noindent
All the terminology introduced in Section~\ref{section:FirstTaste} for diagrams can be easily adapted to pictures. For instance, a dipole in a picture refers to two transistors $\tau_1,\tau_2$ labelled by $u=v$ and $v=u$ for some relation in $\mathcal{R}$ such that bottom wires of $\tau_1$ and the top wires of $\tau_2$ coincide. One reduces the dipole by removing the two transistors and connecting the top wires of $\tau_1$ with the bottom wires of $\tau_2$.  

\medskip \noindent
This alternative point of view allows us to generalise diagram groups in several natural directions. Our initial diagram groups, now referred to as \emph{planar diagram groups}, have their elements represented by pictures whose wires do not cross and are monotonic in the vertical direction. A similar definition on annuli yields \emph{annular diagram groups}. And allowing wires to cross in an arbitrary way yields \emph{symmetric}\footnote{In \cite{MR1396957}, these groups are referred to as \emph{braided diagram groups}, even though the wires are not braided but just permuted. We suggest to replace \emph{braided} with \emph{symmetric}, and to keep the initial terminology for the groups considered in Section~\ref{section:Braided}.} \emph{diagram groups}. We do not give precise definitions and refer to \cite{MR2136028} for more details. See Figures~\ref{Annular} and~\ref{Symmetric} for examples of products of annular and symmetric pictures.

\begin{figure}[h!]
\includegraphics[width=\linewidth]{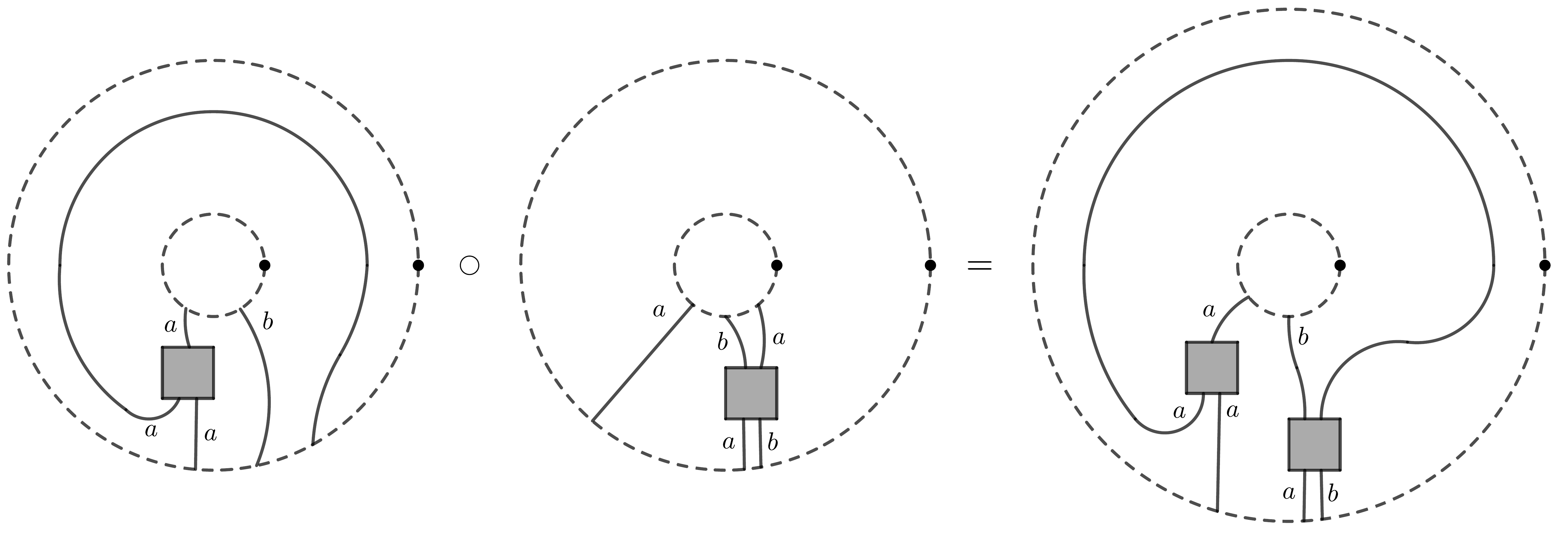}
\caption{A product of annular pictures over $\langle a,b \mid a=a^2,ab=ba \rangle$.}
\label{Annular}
\end{figure}
\begin{figure}[h!]
\begin{center}
\includegraphics[width=0.7\linewidth]{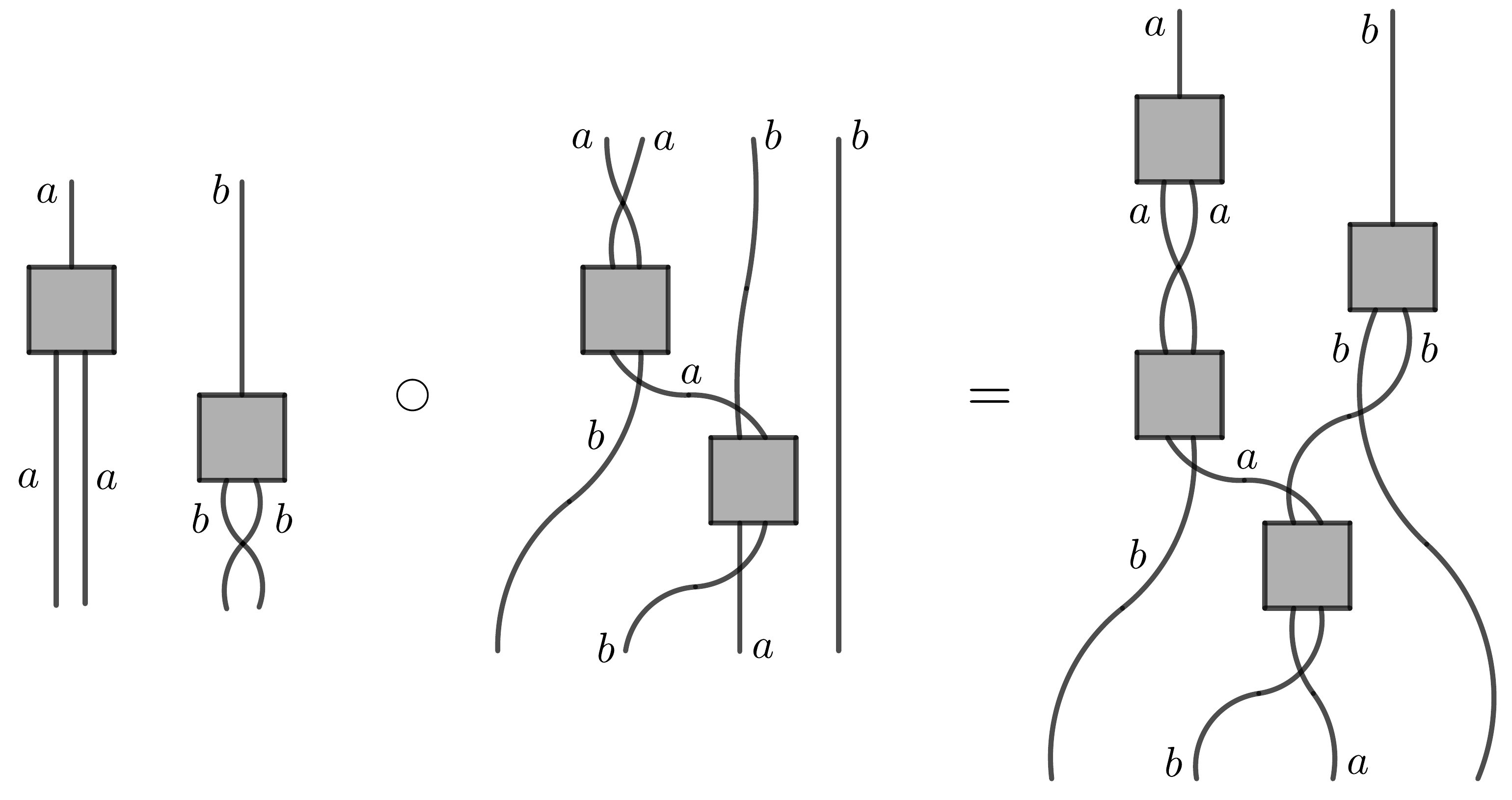}
\caption{A product of symmetric pictures over $\langle a,b \mid a=a^2,b=b^2,ab=ba \rangle$.}
\label{Symmetric}
\end{center}
\end{figure}

\medskip \noindent
Thus, given a semigroup presentation $\mathcal{P}= \langle \Sigma \mid \mathcal{R} \rangle$ and a baseword $w \in \Sigma^+$, we have three nested groups
$$D_p(\mathcal{P},w) \subset D_a(\mathcal{P},w) \subset D_s(\mathcal{P},w)$$
given by the corresponding planar, annular, and symmetric diagram groups. As shown below, for $\mathcal{P}= \langle x \mid x=x^2 \rangle$ and $w=x$, one recovers the three classical Thompson groups $F \subset T \subset V$.

\medskip \noindent
Annular and symmetric diagram groups are sketched in \cite{MR1396957} and further investigated in \cite{MR2136028}. Compared to (planar) diagram groups, much less is known about these groups. Nevertheless, let us record what can be extracted from the existing literature. 

\paragraph{Word problem.} Similarly to diagrams, every planar, annular, or symmetric picture admits a unique \emph{reduced} representative, i.e.\ reducing the dipoles, whatever the order we follow, always yields the same picture. See \cite{MR2136028}. As a consequence:

\begin{prop}
Planar, annular, and symmetric diagram groups have solvable word problems.
\end{prop} 

\noindent
Indeed, given a picture representing an element of the group, it suffices to reduce its dipoles. Once we get a reduced picture, we know that our element is trivial if and only if the picture is reduced to a collection of vertical wires.

\paragraph{Median geometry.} The construction described in Section~\ref{section:MedianDiag} generalises to annular and symmetric diagram groups \cite{MR2136028}. Given a semigroup presentation $\mathcal{P}$, we can construct a median graph $M_s(\mathcal{P})$ as follows. The vertices are the symmetric pictures over $\mathcal{P}$ modulo dipole reduction and right-multiplication by a \emph{symmetric picture} (i.e.\ a picture with no transistor). And the edges connect the classes of two pictures when one can be obtained from the other by adding a transistor at its bottom. The graph $M_s(\mathcal{P})$ turns out to be median, and the symmetric diagram group $D_s(\mathcal{P},w)$, given a baseword $w$, naturally acts by left-multiplication on the connected component $M_s(\mathcal{P},w)$ containing the picture with only vertical wires labelled by the word $w$. Restricting to annular pictures, one gets similarly a median graph $M_a(\mathcal{P},w)$ on which the annular diagram group $D_a(\mathcal{P},w)$ acts. Consequently, Theorem~\ref{thm:FarleyCC} generalises as:

\begin{thm}[\cite{MR2136028}]\label{thm:PictureCC}
Annular and symmetric diagram groups act properly on median graphs locally of finite cubical dimension.
\end{thm}

\noindent
Contrary to planar diagram groups, annular and symmetric diagram groups are usually not torsion-free, since it is possible to permute the wires. More precisely, given an annular (resp. a symmetric) diagram group $D_a(\mathcal{P},w)$ (resp. $D_s(\mathcal{P},w)$) and a picture $\Delta$ with top label $w$, the subgroup given by the pictures
$$\Delta \cdot \left( \begin{array}{c} \text{permutation of the} \\ \text{bottom wires of $\Delta$} \end{array} \right) \cdot \Delta^{-1}$$
is always finite (but possibly trivial). We refer to such a subgroup as a \emph{permutation subgroup}. In annular diagram groups, permutation subgroups are always cyclic. It turns out that permutation subgroups are sufficient to understand finite subgroups, and more generally finitely generated torsion subgroups:

\begin{cor}\label{cor:Torsion}
In an annular or symmetric diagram group, every finitely generated torsion subgroup is contained in a permutation subgroup.
\end{cor}

\begin{proof}
Because the median graph on which our diagram group acts is locally of finite cubical dimension, every finitely generated torsion subgroup must fix a vertex \cite[Theorem~4.23]{NeretinCremona}. But it follows from the construction of the median graph that vertex-stabilisers are permutation subgroups. 
\end{proof}

\noindent
It is worth noticing, in view of Corollary~\ref{cor:Torsion}, that annular and symmetric diagram groups may contain infinite torsion subgroups. For instance, the Houghton groups described below contain the group of finitely supported permutations of an infinite set.

\medskip \noindent
We can also deduce from Theorem~\ref{thm:PictureCC} analogues of Theorem~\ref{thm:Nilpotent} and Proposition~\ref{prop:Roots}

\begin{cor}
In annular and symmetric diagram groups, polycyclic subgroups are virtually abelian and undistorted.
\end{cor}

\begin{proof}
If a group acts properly on a median graph, then its finitely generated abelian subgroups are automatically undistorted \cite[Corollary~6.B.5]{CornulierCommensurated} and its polycyclic subgroups are automatically virtually abelian \cite[Corollary~6.C.8]{CornulierCommensurated}. 
\end{proof}

\begin{cor}
In annular and symmetric diagram groups, an infinite-order element cannot be an arbitrarily large power.
\end{cor}

\begin{proof}
If $g$ is an infinite-order element in our annular or symmetric diagram group, it follows from \cite{arXiv:0705.3386} that it acts as a translation on some bi-infinite geodesic line. Let $\|g\|$ denote the length of this translation. If the equality $a=b^k$ holds in our group, where $a,b$ are elements of infinite order, then 
$$\|a\| = \|b^k\|= |k| \cdot \| b \| \geq |k|,$$ 
since translation length are integers. Thus, $|k|$ cannot be arbitrarily large. 
\end{proof}

\noindent
As a consequence, annular and symmetric diagram groups cannot contain $\mathbb{Q}$. This contrasts with, for instance, the lift $\bar{T}$ of Thompson's group $T$ in $\mathrm{Homeo}(\mathbb{R})$ \cite{MR4478033}.

\paragraph{Finiteness properties.} Even though annular and symmetric diagram groups do not act freely on their median graphs, the properness of the actions is still allows us to deduce that some diagram groups satisfy good finiteness properties. The easiest case is when the actions are cocompact, providing an analogue of Proposition~\ref{prop:EasyF}. 

\begin{prop}
Let $\mathcal{P}= \langle \Sigma \mid \mathcal{R} \rangle$ be a semigroup presentation and $w \in \Sigma^+$ a baseword. If $[w]_\mathcal{P}$ is finite, then the annular and symmetric diagram groups $D_a(\mathcal{P},w)$ and $D_s(\mathcal{P},w)$ are of type $F_\infty$.
\end{prop}

\noindent
A more interesting result, an analogue of Theorem~\ref{thm:Farley} which is proved using the same strategy based on Morse theory, is:

\begin{thm}[\cite{MR2146639}]
Let $\mathcal{P}= \langle \Sigma \mid \mathcal{R} \rangle$ be a semigroup presentation and $w \in \Sigma^+$ a baseword. If the oriented graph $\Gamma_a(\mathcal{P})$ (resp. $\Gamma_s(\mathcal{P})$) does not contain an infinite oriented ray and contains only finitely many sink vertices, then $D_a(\mathcal{P},w)$ (resp. $D_s(\mathcal{P},w)$) is of type $F_\infty$. 
\end{thm}

\noindent
Here, $\Gamma_a(\mathcal{P})$ (resp. $\Gamma_s(\mathcal{P})$) is the oriented graph whose vertices are the words in $\Sigma^+$ modulo cyclic permutation (resp. modulo permutation) and whose oriented edges connect a given word $aub$ to a word $avb$ whenever $u=v$ belongs $\mathcal{R}$.

\paragraph{Diagram products.} Let $\mathcal{P}= \langle \Sigma \mid \mathcal{R} \rangle$ be a semigroup presentation, $\mathcal{G}= \{ G_s \mid s \in \Sigma\}$ a collection of groups indexed by $\Sigma$, and $w \in \Sigma^+$ a baseword. By labelling the wires of annular and symmetric pictures with elements from groups in $\mathcal{G}$, one can naturally extend the definitions from Section~\ref{section:DiagramProducts} in order to define annular and symmetric diagram products. This is done in \cite{MR4033512}, where the quasi-median geometry is also extended. This allows us to prove an analogue of Theorem~\ref{thm:SemiProduct} and to prove a statement similar to Theorem~\ref{thm:Short}, namely:

\begin{thm}[\cite{MR4033512}]
Every annular (resp. symmetric) diagram group satisfies a short exact sequence
$$1 \to R \to G \to S \to 1$$
where $R$ is a subgroup of a right-angled Artin group and $S$ is a subgroup of Thompson's group $T$ (resp. $V$). 
\end{thm}

\noindent
For instance, this implies that every simple annular (resp. symmetric) diagram group embeds into Thompson's group $T$ (resp. $V$). This excludes in particular the higher dimensional Thompson groups $nV$ for $n \geq 2$ \cite{MR2112673}.

\paragraph{Examples.} Let $\mathcal{P}= \langle x \mid x=x^2 \rangle$. A reduced annular (resp. symmetric) $(x,x)$-picture over $\mathcal{P}$ can always be written as the concatenation of a \emph{positive} picture (i.e.\ all of whose transistors are labelled by $x=x^2$), followed by a cyclic permutation (resp. a permutation) of the wires, and finally by a \emph{negative} picture (i.e.\ all of whose transistors are labelled by $x^2=x$). Such a structure immediately yields a triple $(T_1,\sigma,T_2)$ where $T_1,T_2$ are two finite rooted $2$-regular trees with the same number of leaves, say $n$, and where $\sigma$ is a cyclic permutation (resp. a permutation) of $n$ elements. This correspondence induces an isomorphism from the annular diagram group $D_a(\mathcal{P},x)$ to Thompson's group $T$ (resp. from the symmetric diagram group $D_s(\mathcal{P},x)$ to Thompson's group $V$). We refer to \cite{MR1426438} for more information about the representations of elements of $T$ and $V$ as pairs of trees. 

\medskip \noindent
These examples can be found in \cite{MR1396957}. Other interesting examples can be found in \cite{MR3822286}, namely the Houghton groups $H_n$ and the symmetrisation $QV$ of Thompson's group $V$. 

\medskip \noindent
Given an $n \geq 1$, let $R_n$ denote the union of $n$ disjoint infinite rays. The \emph{Houghton group} $H_n$ is the group of the bijections $R_n^{(0)}\to R_n^{(0)}$ that preserve the ends of the graph $R_n$ and that preserve adjacency and non-adjacency for all but finitely many pairs of vertices. In other words, an element of $H_n$ acts by translations outside a finite set of $R_n$ and it permutes the vertices of this set in an arbitrary way. Since, on each infinite ray, an element of $H_n$ acts as a translation, we have $n$ morphisms $H_n \to \mathbb{Z}$ giving the corresponding translation lengths (positive if the vertices are translated towards infinity, negative otherwise). In order to get a bijection $R_n^{(0)} \to R_n^{(0)}$, the sum of these translation lengths must be zero. Thus, one deduces a short exact sequence
$$1 \to S_\infty \to H_n \to \mathbb{Z}^{n-1} \to 1$$
where $S_\infty$ is identified with the finitely supported permutations of $R_n^{(0)}$. 

\medskip \noindent
The definition of $QV$ has a similar flavor. Let $\mathcal{T}$ denote the rooted $2$-regular tree with a fixed embedding in the plane (so each vertex has two children: a left and a right). Then $QV$ is the group of the bijections $\mathcal{T}^{(0)} \to \mathcal{T}^{(0)}$ that preserves adjacency, non-adjacency, and left-right order on the children with only finitely many exceptions. Notice that $QV$ has a well-defined action on the boundary of $\mathcal{T}$, which is a Cantor set. This action turns out to coincide with the usual definition of Thompson's group $V$ as a group of homeomorphisms of the Cantor set. Hence a short exact sequence
$$1 \to S_\infty \to QV \to V \to 1$$
where $S_\infty$ is identified with the finitely supported permutations of $\mathcal{T}^{(0)}$. 

\begin{prop}[\cite{MR3822286}]
For every $n \geq 2$, the Houghton group $H_n$ is isomorphic to the symmetric diagram group $D_s(\mathcal{P}_n,r)$ where
$$\mathcal{P}_n := \langle a,x_1, \ldots, x_n \mid x_1=x_1a, \ldots, x_n=x_na \rangle.$$
The group $QV$ is isomorphic to the symmetric diagram group $D_s(\mathcal{Q},x)$ where $\mathcal{Q}:= \langle a,x \mid x=xax \rangle.$
\end{prop}

\noindent
Figure~\ref{Houghton} illustrates how to describe the second Houghton group as a symmetric diagram group. The philosophy is the same for the other Houghton groups and for $QV$. See \cite[Examples~4.3 and~4.4]{MR3822286} for more details. Notice that, by the same argument, one can show that the groups $QV_{n,r}$, which are defined by replacing our previous rooted tree $\mathcal{T}$ with a forest of $r$ rooted $n$-regular trees, are also symmetric diagram groups. 

\begin{figure}[h!]
\begin{center}
\includegraphics[width=0.9\linewidth]{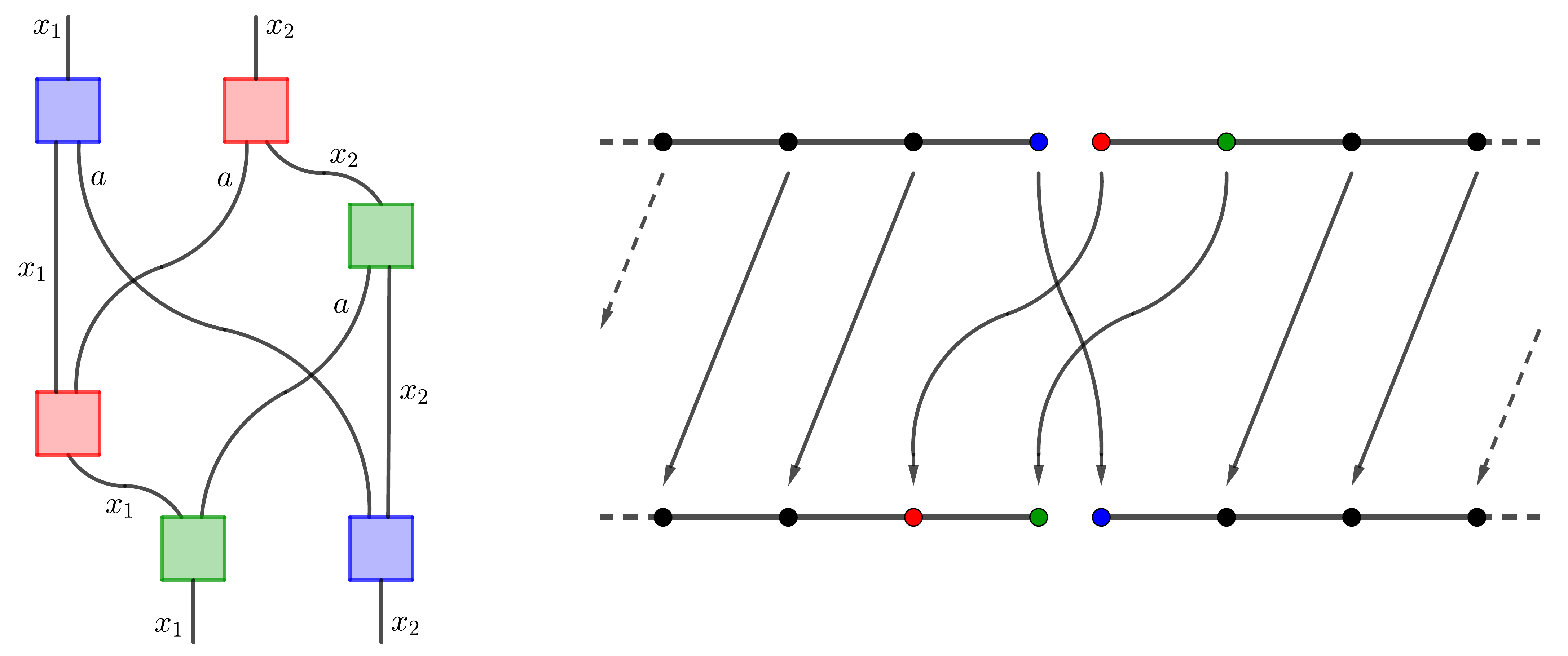}
\caption{$D_s(\langle a,x_1,x_2 \mid x_1=x_1a, x_2=ax_2 \rangle, x_1x_2)$ isomorphic to $H_2$.}
\label{Houghton}
\end{center}
\end{figure}

\subsection{Braided-like diagram groups}\label{section:Braided}

\noindent
In addition to planar, annular, and symmetric diagram groups, it would be natural to define \emph{braided diagram groups}, where wires in pictures are thought of as strands in braids. This variation has not been defined properly in the literature yet. Nevertheless, the following results can be reasonably expected:
\begin{itemize}
	\item The braided diagram group $D_b(\mathcal{P},x)$, where $\mathcal{P}= \langle x \mid x=x^2 \rangle$, coincides with the braided Thompson group $\mathrm{br}V$ \cite{MR2258261, MR2364825}.
	\item A braided diagram group acts on a median graph with vertex-stabilisers isomorphic to finitely generated braid groups. As a consequence, it is torsion-free.
	\item Polycyclic subgroups in braided diagram groups are virtually abelian and undistorted.		
	\item Every braided diagram group $G$ satisfies a short exacst sequence $$1 \to R \to G \to S \to 1$$ where $R$ is a subgroup of a right-angled Artin group and where $S$ is a subgroup of $\mathrm{br}V$. 
\end{itemize}
It would interesting to know whether braided diagram groups are bi-orderable. In other words, is it possible to combine the bi-orderability of braid groups and (planar) diagram groups in order to deduce the bi-orderability of braided diagram groups?

\medskip \noindent
In this direction, we can imagine as many diagram-like groups as there are braided-like groups. This could include \emph{virtual diagram groups}, in reference to virtual braid groups; \emph{loop diagram groups}, in reference with loop braid groups; \emph{cactus diagram groups}, in reference to the cactus groups; etc. The variations are virtually endless.

\medskip \noindent
In another direction, it is possible to treat wires differently according to the letters labelling it. There are at least two such examples available in the literature.

\medskip \noindent
The first example is \cite{MR3797069}, dedicated to the symmetrisation $QF$ (resp. $QT$) of Thompson's group $F$ (resp. $T$). The group $QF$ (resp. $QT$) is naturally a subgroup of $QV$, which is isomorphic to the symmetric diagram group $D_s(\mathcal{P},x)$ where $\mathcal{P}= \langle a,x \mid x=axa \rangle$, and its elements can be described as the pictures where the wires labelled $a$ behave like in a planar (resp. annular) picture. This point of view allows the authors to show that the groups $QF$, $QT$, $QV$ are of type $F_\infty$.

\medskip \noindent
As another example, we can mention the \emph{Chambord groups} introduced in \cite{Chambord}. Formally, they are defined thanks to specific diagrams with both strands (which are braided) and wires (which are cyclically permuted). This formalism allows the authors to describe combinatorially the asymptotically rigid mapping class groups studied in \cite{MR4466651}, including the braided Ptolemy-Thompson from \cite{MR2390352} and the braided Houghton groups from \cite{brHoughton, MR2386796}.

\subsection{Monoid presentations}

\noindent
So far, we have only considered diagram groups defined from semigroup presentations, but it is also possible to start from monoid presentations, i.e.\ allowing relations of the form $u=1$ where $1$ corresponds to the empty word. Given a monoid presentation $\mathcal{P}= \langle \Sigma \mid \mathcal{R} \rangle$, one can define a Squier square complex $S(\mathcal{P})$ by repeating word for word Definition~\ref{def:SquierSquare}; and, fixing a baseword $w \in \Sigma^+$, the diagram group $D(\mathcal{P},w)$ is defined as the fundamental group of the connected component $S(\mathcal{P},w)$ of $S(\mathcal{P})$. Following Section~\ref{section:Diagrams}, a diagrammatic description od $D(\mathcal{P},w)$ can be obtained, but it is more convenient to use pictures instead of diagrams. The main difference with semigroup presentations is that, in a picture over our monoid presentation, a transistor labelled by a relation $u=1$ (resp. $1=u$) has not bottom (resp. top) wires. 

\medskip \noindent
\begin{minipage}{0.66\linewidth}
\includegraphics[width=0.9\linewidth]{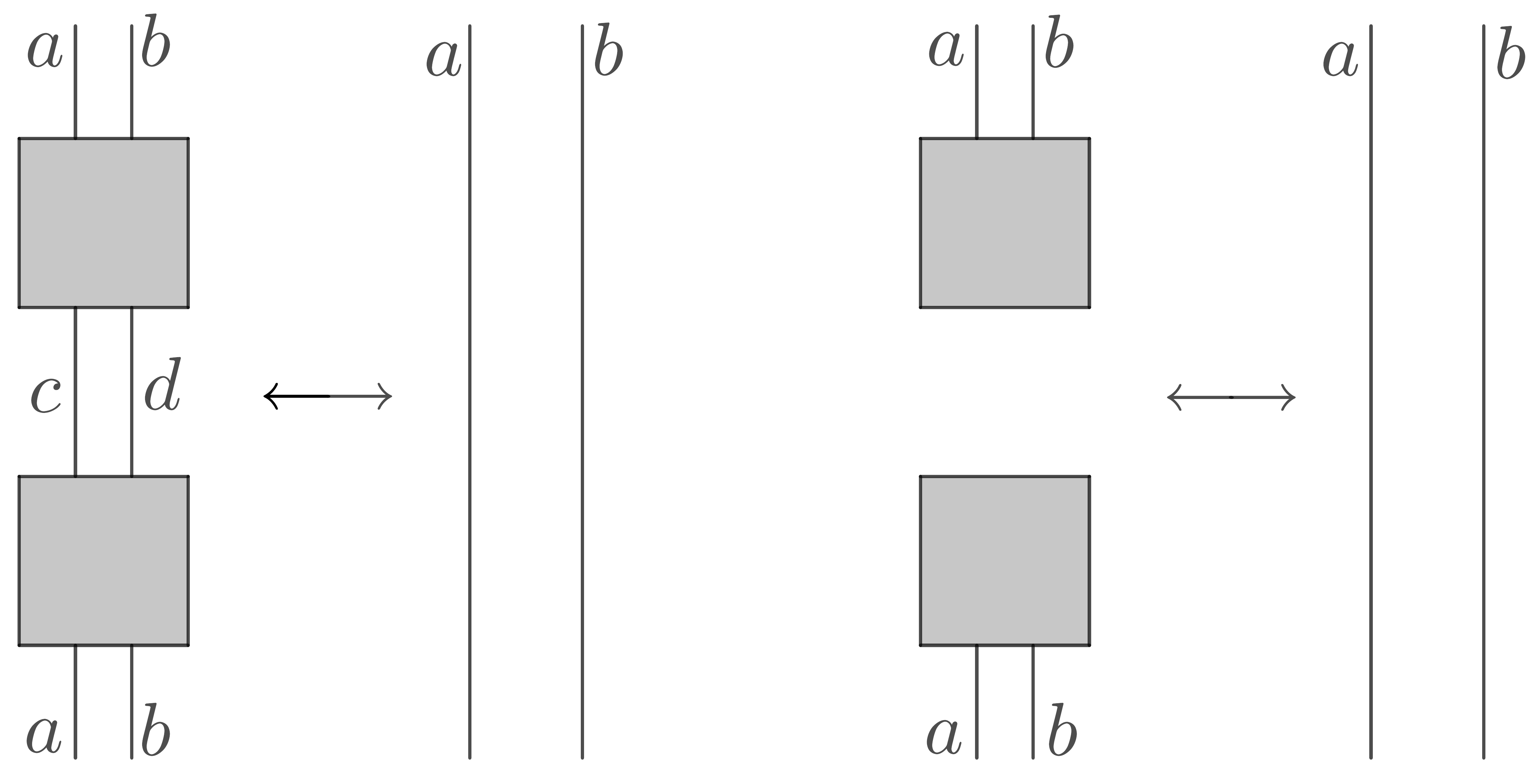}
\begin{center}
Two examples of dipole reduction.
\end{center}
\end{minipage}
\begin{minipage}{0.3\linewidth}
\begin{center}
\includegraphics[width=0.65\linewidth]{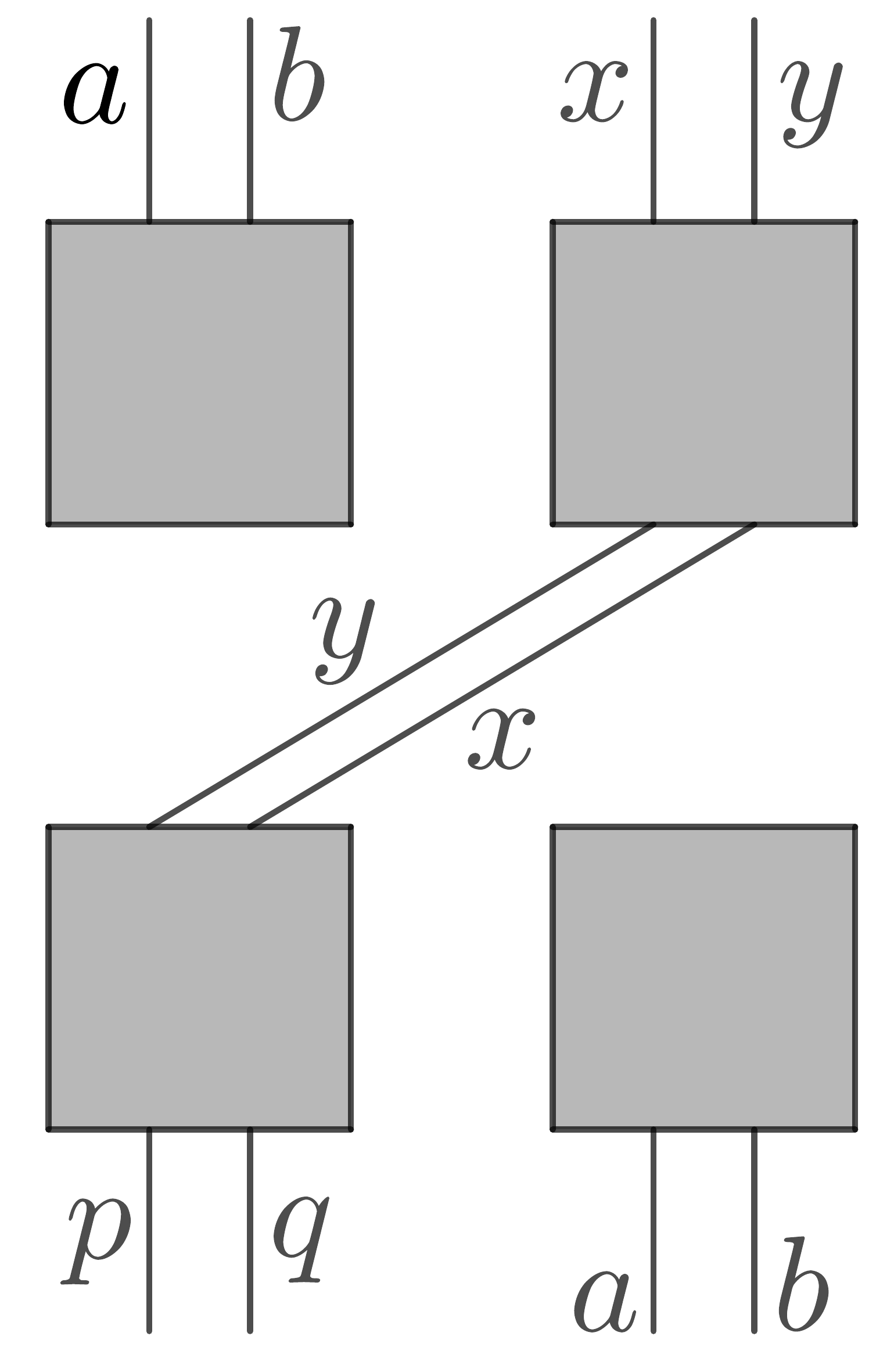}
\end{center}
\begin{center}
A reduced picture.
\end{center}
\end{minipage}

\medskip \noindent
We refer to \cite{MR1396957} for more details. Even though many techniques can be transfered from diagram groups over semigroup presentation to diagram groups over monoid presentations, a major difference is that a picture over monoid presentation may admit several reduced forms, i.e.\ reducing the dipoles in one order or another may lead to distinct pictures. 
\begin{figure}[h!]
\begin{center}
\includegraphics[width=0.8\linewidth]{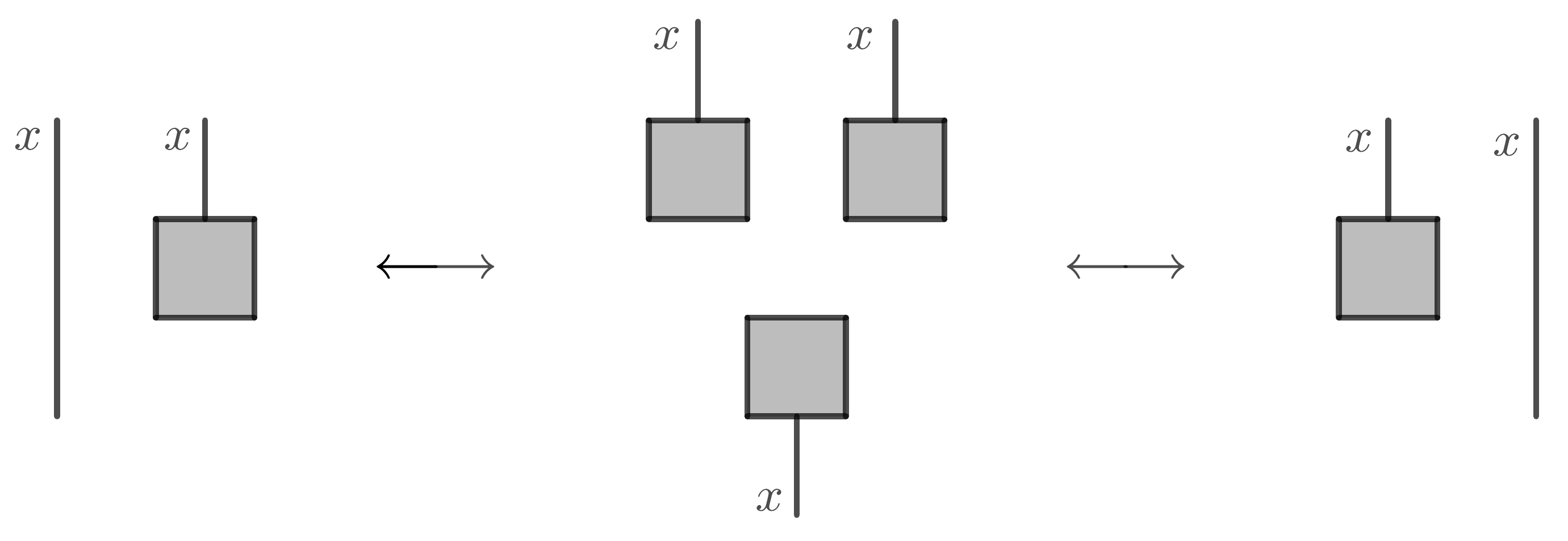}
\caption{Two distinct reductions of a picture over $\langle x \mid x=1 \rangle$.}
\label{Reduced}
\end{center}
\end{figure}

\noindent
Even worse, the problem of determining whether or not two given pictures are equivalent modulo dipole reduction may be undecidable \cite[Theorem~13.3]{MR1396957}. This explains why it is usually much more difficult to work with diagram groups over monoid presentations than diagram groups over semigroup presentations. As an illustration of this phenomenon, let us mention the following open question:

\begin{question}[\cite{MR1396957}]\label{question:monoid}
Are diagram groups over monoid presentations torsion-free?
\end{question}

\noindent
This question illustrate how little is known about diagram groups over monoid presentations. 

\medskip \noindent
It is worth noticing that Question~\ref{question:monoid} cannot be answered with the help of median geometry as mentioned in Section~\ref{section:MedianDiag}. Indeed, a (connected component of the) Squier cube complex over a monoid presentation may not be nonpositively curved. For instance, if $\mathcal{P}= \langle x \mid x=1 \rangle$, then the $4$-cycle
$$xx \overset{(x,x=1, 1)}{\longrightarrow} x \overset{(1,x=1,1)}{\longrightarrow} 1 \overset{(1,1=x,1)}{\longrightarrow} x \overset{(1,x=1,x)}{\longrightarrow} xx$$
shows that the link of the vertex $1$ is not simplicial since there is a loop based at the vertex given by the edge $(1,1=x,1)$.

\section{Open questions}\label{section:OpenQuestion}

\noindent
In this final section, we record various open questions related to diagram groups.

\paragraph{Classification of diagram groups.} In view of the known examples of diagram groups, there seem to be three main families of diagram groups: those which contain a copy of Thompson's group $F$, those which contain a copy of the lamplighter group $\mathbb{Z} \wr \mathbb{Z}$ but not of $F$, and those which embed into a right-angled Artin group. Are there other big families of diagram groups? Let us ask a few precise question in this direction. 

\begin{conj}[\cite{MR3868219}]
A diagram group embeds into a right-angled Artin group if and only if it does not contain $\mathbb{Z} \wr \mathbb{Z}$.
\end{conj}

\begin{question}[\cite{MR3868219}]
Is a diagram group not containing Thompson's group $F$ linear or residually finite?
\end{question}

\noindent
It is worth noticing that, given a semigroup presentation $\mathcal{P}= \langle \Sigma \mid \mathcal{R} \rangle$ and a baseword $w \in \Sigma^+$, we know exactly when the diagram group $D(\mathcal{P},w)$ contains $F$ or $\mathbb{Z} \wr \mathbb{Z}$. Indeed, $D(\mathcal{P},w)$ contains $F$ if and only if there exist words $a,b,x$ such that $w=axb$ and $x^2=x$ hold modulo $\mathcal{P}$ \cite{MR1902358}; and $D(\mathcal{P},w)$ contains $\mathbb{Z} \wr \mathbb{Z}$ if and only if there exist words $a,b,p$ such that $D(\mathcal{P},p) \neq \{1\}$ and such that $w=ab$, $a=ap$, $b=pb$ hold modulo $\mathcal{P}$ \cite{MR1725439}. 

\begin{question}[\cite{MR2193191}]
If a diagram group does not contain any non-abelian free subgroup, does it embed into Thompson's group $F$?
\end{question}

\noindent
Notice that Theorem~\ref{thm:Short} implies that a diagram with no free subgroup maps to $F$ with an abelian kernel.

\paragraph{Isomorphism problem.} Given any reasonable family of groups, a basic, but usually difficult, problem is to find how to determine whether two members of the family yield isomorphic groups. 

\begin{question}\label{question:IsoProblem}
Given two semigroup presentations $\mathcal{P}_1= \langle \Sigma_1 \mid \mathcal{R}_1 \rangle$, $\mathcal{P}_2= \langle \Sigma_2 \mid \mathcal{R}_2 \rangle$ and two basewords $w_1 \in \Sigma_1^+$, $w_2 \in \Sigma_2^+$, how to determine whether the diagram groups $D(\mathcal{P}_1,w_1)$ and $D(\mathcal{P}_2,w_2)$ are isomorphic?
\end{question}

\noindent
As an easy sufficient condition, notice that, given a semigroup presentation $\mathcal{P}= \langle \Sigma \mid \mathcal{R} \rangle$ and two words $u,v \in \Sigma^+$, the diagram groups $D(\mathcal{P},u)$ and $D(\mathcal{P},v)$ are isomorphic if $u$ and $v$ represent the same element in the semigroup defined by $\mathcal{P}$. This is a immediate consequence of the fact that $u$ and $v$ are two vertices in the same connected component of the Squier square complex $S(\mathcal{P})$, which implies that the fundamental groups are conjugate in the fundamental groupoid. 

\medskip \noindent
More interesting conditions are given by \cite[Theorem~4.1]{MR2193190}. Nevertheless, we are far from a reasonable answer to Question~\ref{question:IsoProblem}. In order to illustrate the difficulty of the problem, let us mention a few examples of non-trivial isomorphisms. Considering the semigroup presentations
$$\begin{array}{l} \mathcal{A}= \langle a,b,c \mid a=b,b=c,c=a \rangle \\ \mathcal{B}= \langle a,b,p \mid a=ap, b=pb \rangle \\ \mathcal{C}= \langle a,b,x,y \mid ax=ay, xb=yb \rangle \end{array}$$
the diagram groups $D(\mathcal{A},a)$, $D(\mathcal{B},ab)$, and $D(\mathcal{C},axb)$ are all infinite cyclic. Considering the semigroup presentations
$$\begin{array}{l} \mathcal{D}= \langle a,b,c, p \mid a=b,b=c,c=a, a=ap \rangle \\ \mathcal{E} = \langle a,b,c,x,y,p \mid a=b, b=c, c=a, ax=ay, xp=yp \rangle \\ \mathcal{F}= \langle a,b,c \mid ab=ba, ac=ca, bc=cb \rangle \end{array}$$
the diagram groups $D(\mathcal{D},a)$, $D(\mathcal{E},axp)$, and $D(\mathcal{F}, abc^2)$ are all free of rank two. Considering the presentation 
$$\begin{array}{l} \mathcal{G}= \langle a,b,c \mid a=bc,b=ca,c=ab \rangle \\ \mathcal{H}= \langle x \mid x=x^9 \rangle \end{array}$$
the diagram groups $D(\mathcal{G},a)$ and $D(\mathcal{H},x)$ are both isomorphic to the generalised Thompson group $F_9$ \cite{MR4412796}. 

\medskip \noindent
The triviality problem is probably a more reasonable question before attacking Question~\ref{question:IsoProblem}. 

\begin{question}
Given a semigroup presentation $\mathcal{P}= \langle \Sigma \mid \mathcal{R} \rangle$ and a baseword $w \in \Sigma^+$, how to determine whether the diagram group $D(\mathcal{P},w)$ is non-trivial? 
\end{question}

\paragraph{Algorithmic aspects.} In addition to the isomorphism and triviality problems previously mentioned, there are other interesting problem that remained to be solved algorithmically. For instance:

\begin{question}[\cite{MR1396957}]
Given a finite semigroup presentation $\mathcal{P}= \langle \Sigma \mid \mathcal{R} \rangle$ and a baseword $w \in \Sigma^+$, is there an algorithm computing the rank of the abelianisation of the diagram group $D(\mathcal{P},w)$?
\end{question}

\noindent
In free groups, many problems are solved algorithmically thanks to Stallings' foldings. As mentioned in Section~\ref{section:Fold}, a similar theory exists for directed $2$-complexes and their diagram groups. However, the theory remains to be developed. 

\begin{problem}
Develop the theory of foldings of directed $2$-complexes for diagram groups. In particular, solve \cite[Conjecture~3.11]{Golan}.
\end{problem}

\noindent
Nevertheless, it is worth mentioning that Stallings' foldings have been applied with success to Thompson's group $F$, see for instance \cite{MR3710646, Golan, GolanSapir}.

\paragraph{Solvable diagram groups.} Theorem~\ref{thm:Nilpotent} shows that there are no interesting nilpotent diagram groups. There are interesting solvable diagram groups, such as the lamplighter group $\mathbb{Z} \wr \mathbb{Z}$, but only few examples are known. Finding more examples, or proving that there are only few possible examples, would be interesting. 

\begin{question}
Which metabelian (resp. solvable) groups are diagram groups?
\end{question}

\begin{question}
Are there finitely presented solvable diagram groups?
\end{question}

\begin{conj}
A diagram group does not contain $\mathbb{Z} \wr \mathbb{Z}^2$.
\end{conj}

\paragraph{Right-angled Artin groups.} As mentioned in Section~\ref{section:Examples}, right-angled Artin groups given by finite trees and finite interval graphs are diagram groups. On the other hand, we saw with Corollary~\ref{cor:NotRAAG} that right-angled Artin groups defined by triangle-free graphs with induced cycles of odd length $\geq 5$ are not diagram groups. This naturally raises the following question:

\begin{question}\label{question:RAAG}
Which (finitely generated) right-angled Artin groups can be described as diagram groups?
\end{question}

\noindent
So far, no expected nor reasonable answer is known. A related question is:

\begin{question}[\cite{MR3868219}]
Are all right-angled Artin groups subgroups of diagram groups?
\end{question}

\noindent
For instance, even though right-angled Artin groups defined by cycles of odd lengths $\geq 5$ are not diagram groups, they do embed into diagram groups\footnote{Contrary to what \cite{MR2422070} claims. There, the authors cite \cite{MR1725439} in order to justify that such right-angled Artin groups do not embed into diagram groups. But the result only proves that these right-angled Artin groups are not diagram groups themselves, and the proof (sketched in Section~\ref{section:Commutation}) does not imply this statement either.} \cite{MR3868219}.

\paragraph{Miscellaneous.} Recall that, given two groups $G,H$, the group $G$ is \emph{residually $H$} if, for every non-trivial $g \in G$, there exists a morphism $\varphi : G \to H$ satisfying $\varphi(g) \neq 1$. 

\begin{question}[\cite{MR1396957, MR2193191}]\label{question:ResiduallyPL}
Are all diagram groups residually $\mathrm{PL}_+([0,1])$? or even residually $F$?
\end{question}

\noindent
A possible strategy is the following. Let $\mathcal{P}= \langle \Sigma \mid \mathcal{R} \rangle$ be a semigroup presentation and let $w \in \Sigma^+$ a baseword. Every atomic $(aub,avb)$-diagram $\Delta$ over $\mathcal{P}$ can be written as $\epsilon(a) + \Psi(u=v) + \epsilon(b)$, where $\epsilon(a),\epsilon(b)$ are diagrams with no cells labelled by $a,b$ and where $\Psi(u=v)$ is a single cell labelled by $u=v$. To such a diagram $\Delta$, we associate a continuous, nondecreasing, piecewise linear homeomorphism $H(\Delta) : [0,|aub|] \to [0,|avb|]$ which has slope $1$ on both $[0,|a|]$ and $[|au|,|aub|]$. Then, one gets a morphism $$ \rho : D(\mathcal{P},w) \to \mathrm{PL}_+([0,|w|])$$ by taking a diagram $\Phi$, decomposing it as a concatenation of atomic diagrams $\Phi_1 \circ \cdots \circ \Phi_n$, and sending it to $H(\Phi_1) \circ \cdots \circ H(\Phi_n)$. Is it is possible to obtain a non-trivial homeomorphism $\rho(\Phi)$ by choosing the homeomorphisms $H(\cdot)$ carefully?

\medskip \noindent
We saw a particular instance of this construction in Section~\ref{section:Examples}, where we described how to embed $D(\langle x \mid x=x^2 \rangle,x)$ into $\mathrm{PL}_+([0,1])$ and recognise that it coincides with Thompson's group $F$. 

\medskip \noindent
As mentioned in Section~\ref{section:Properties}, it is proved in \cite{HypDiag} that diagram groups with $\mathbb{Z}^2$ are locally free, which implies in particular that there are no interesting hyperbolic diagram groups. First, this raises the following question:

\begin{question}[\cite{HypDiag}]
Is a locally free diagram group necessarily free?
\end{question}

\noindent
Next, since we saw in Section~\ref{section:Acyl} that there are many interesting examples of acylindrically hyperbolic diagram groups, it is natural to look for relatively hyperbolic diagram groups. We conjecture that there are no interesting examples there either.

\begin{conj}\label{conj:FreeProduct}
Relatively hyperbolic diagram groups are free products.
\end{conj}

\begin{question}[\cite{MR2193190}]
Do the cohomological and algebraic dimensions of diagram groups always coincide? 
\end{question}

\noindent
Recall that the algebraic dimension of a group is the maximal rank of a free abelian subgroup. According to \cite{MR2193191}, the answer is positive for diagram groups over complete presentations. 

\begin{question}[\cite{MR1725439}]
Is every subgroup of Thompson's group $F$ a diagram group?
\end{question}

\begin{question}[\cite{MR1725439}]
Is the derived subgroup of a diagram group again a diagram group?
\end{question}

\noindent
Recall from Section~\ref{section:Hilbert} that a finitely generated diagram group satisfies the Burillo property if every word length is biLipschitz equivalent the diagram length $\#( \cdot)$, i.e.\ the number of $2$-cells in a reduced representative. 

\begin{question}[\cite{MR2271228}]
Does there exist a finitely generated diagram that does not satisfy the Burillo property?
\end{question}

\noindent
There are examples of infinitely presented diagram groups (such as abelian and non-abelian free groups of infinite rank), examples of finitely generated diagram groups that are not finitely presented (such as $\mathbb{Z} \wr \mathbb{Z}$ and $\mathbb{Z} \bullet \mathbb{Z}$), examples of diagram groups of type $F_\infty$ but not of type $F$ (such as Thompson's group $F$), and examples of diagram groups of type $F$ (such as abelian and non-abelian free groups). But there is a gap in the progression, hence:

\begin{question}
Given an $n \geq 2$, does there exist a diagram group of type $F_{n+1}$ but not $F_{n}$?
\end{question}

\noindent
It is worth noticing that such examples exist for symmetric diagram groups: the Houghton group $H_n$ is of type $F_{n-1}$ but not of type $F_n$ \cite{MR885095}.

\addcontentsline{toc}{section}{References}

\bibliographystyle{alpha}
{\footnotesize\bibliography{DiagGroups}}

\Address

%

\end{document}